\documentclass[12pt,a4paper,reqno,oneside]{amsart}

\usepackage{amsmath}
\usepackage{amsthm}
\usepackage{amssymb}
\usepackage{bbm}
\usepackage[sort,nocompress]{cite} 
\usepackage{enumerate}
\usepackage{xcolor}
\usepackage{geometry}
\usepackage{stmaryrd}
\usepackage[colorlinks,citecolor=red]{hyperref}
\usepackage[notref,notcite,final]{showkeys}

\usepackage[linesnumbered,ruled,vlined]{algorithm2e}
\SetKwInput{KwInput}{Input}
\SetKwInput{KwOutput}{Output}

\usepackage{graphicx}
\usepackage[capitalise]{cleveref}
\usepackage{array}
\usepackage{etoolbox}
\AtBeginEnvironment{tabular}{\small}
\newcolumntype{L}[1]{>{\raggedright\let\newline\\\arraybackslash\hspace{0pt}}m{#1}}
\newcolumntype{R}[1]{>{\raggedleft\let\newline\\\arraybackslash\hspace{0pt}}m{#1}}

\newtheorem{lemma}{Lemma}[section]
\newtheorem{remark}[lemma]{Remark}

\newtheorem{theorem}[lemma]{Theorem}

\newtheorem{corollary}[lemma]{Corollary}

\newtheorem{setting}[lemma]{Setting}
\theoremstyle{definition}

\renewcommand{\gets}{\curvearrowleft}
\providecommand{\N}{{\ensuremath{\mathbbm{N}}}}
\providecommand{\Z}{{\ensuremath{\mathbbm{Z}}}}
\providecommand{\R}{{\ensuremath{\mathbbm{R}}}}

\renewcommand{\P}{{\ensuremath{\mathbbm{P}}}}
\providecommand{\E}{{\ensuremath{\mathbbm{E}}}}
\providecommand{\V}{{\ensuremath{\mathbbm{V}}}}
\providecommand{\1}{{\ensuremath{\mathbbm{1}}}}
\providecommand{\F}{{\ensuremath{\mathbbm{F}}}}

\newcommand{\threenorm}[1]{{\left\vert\kern-0.25ex\left\vert\kern-0.25ex\left\vert #1 
    \right\vert\kern-0.25ex\right\vert\kern-0.25ex\right\vert}}

\newcommand{\unif}{\mathfrak{u}}
\newcommand{\rvV}{\mathfrak{v}}
\providecommand{\var}{{\ensuremath{\mathbbm{V}}}}
\newcommand{\rdown}[1] {\lfloor #1\rfloor }

\author[A. Neufeld]{Ariel Neufeld}
\address{Division of Mathematical Sciences, Nanyang Technological University, Singapore}
\email{ariel.neufeld@ntu.edu.sg}

\thanks{Financial support by the MOE AcRF Tier 2 Grant \emph{MOE-T2EP20225-0011} is gratefully acknowledged.}

\author[T.A. Nguyen]{Tuan Anh Nguyen}
\address{Faculty of Mathematics, Bielefeld University, Germany}
\email{tnguyen@math.uni-bielefeld.de}

\author[P. Schmocker]{Philipp Schmocker}
\address{Department of Mathematics, ETH Zurich, Switzerland}
\email{philipp.schmocker@math.ethz.ch}

\title[MLP for McKean--Vlasov SDEs with jumps]{Multilevel Picard approximations for McKean--Vlasov stochastic differential equations with jumps}
\allowdisplaybreaks
\begin{document}
	
\begin{abstract}
We introduce multilevel Picard (MLP) approximations of McKean--Vlasov stochastic differential equations (SDEs) with jumps of finite or infinite activity. Under Lipschitz and integrability assumptions on the coefficients, we show that the MLP algorithm does not suffer from the curse of dimensionality when approximating the solution of the McKean--Vlasov SDE. The latter means that its computational cost grows at most polynomially in both the state-space dimension of the SDE and the reciprocal of the prescribed error tolerance. In two numerical experiments, we demonstrate the practical applicability of the MLP algorithm for two different McKean--Vlasov SDEs in dimensions up to 200.
\end{abstract}

\maketitle
\setcounter{tocdepth}{3} 

\tableofcontents

\section{Introduction}
\tableofcontents

McKean--Vlasov stochastic differential equations (SDEs) were introduced by H.P.~McKean in \cite{McKean1966} in connection with the Vlasov equation and provide a natural framework for analyzing the limiting behavior of stochastic systems of weakly interacting particles. The coefficients of these SDEs depend not only on the state of the process but also on its distribution, which is why they are also referred to as distribution-dependent SDEs or mean-field SDEs. Applications of McKean--Vlasov SDEs arise in various fields, including statistical physics (see, e.g., \cite{Kac1956,Sznitman1991,DV1995,Meleard1996,BCP1997}), mathematical biology (see, e.g., \cite{Graham1992,ME1999,BCM2007,BFFT2012,MSSZ2020}), social sciences (see, e.g., \cite{CFT2012,CJLW2017,GGSP2022}), mean-field games (see, e.g., \cite{CL2018,Carmona2018,CDLL2019}), stochastic control (see, e.g., \cite{BLM2017,Agram2023}), mathematical finance (see, e.g., \cite{GSS2020}), and the analysis of neural networks (see, e.g., \cite{MMN2018,SS2020,HRSS2021,RVE2022}).

In this paper, we are interested in the numerical approximation of McKean--Vlasov SDEs \emph{with jumps}, which are for $d \in \N$ and terminal time $T > 0$ of the form
\begin{align}
	\label{eq:mckv}
	X^{d}_{\infty}(t)&= \xi_d+\int_{0}^{t}
	\E\! \left[\mu_d (x, X^d_{\infty}(s-) )\right]\!|_{x=X^d_{\infty}(s-)}\,ds+\int_{0}^{t}
	\E\! \left[\sigma_d (x, X^d_{\infty}(s-) )\right]\!|_{x=X^d_{\infty}(s-)}\,dW^d(s)\nonumber
	\\
	&\quad +\int_{0}^{t}\int_{\R^d}
	\E\! \left[\eta_d (x, X^d_{\infty}(s-),z )\right]\!|_{x=X^d_{\infty}(s-)}\,\tilde{N}^d(dz,ds), \qquad t \in [0,T],
\end{align}
where $\xi_d \in \R^d$ is the initial value, where $\mu_d: \R^d \times \R^d \rightarrow \R^d$, $\sigma_d: \R^d \times \R^d \rightarrow \R^{d \times d}$, and $\eta_d: \R^d \times \R^d \times \R^d \rightarrow \R^d$ are sufficiently regular functions, where $W^d$ is a $d$-dimensional Brownian motion, and where $\tilde{N}^d$ is a compensated Poisson random measure. Under standard Lipschitz assumptions on the coefficients, together with suitable integrability conditions on $\eta_d$ with respect to the L\'evy measure associated to $\tilde{N}^d$, existence and uniqueness of strong solutions to \eqref{eq:mckv} have been established in \cite{Sznitman1991,Graham1992,Andreis2018,Papapantoleon2024,BaoWang2026}. Replacing the expectation over the coefficients in \eqref{eq:mckv} by the empirical measure of $N \in \N$ interacting particles, we obtain a system of SDEs given by
\begin{align}
	\mathcal{X}^{d,n}_N(t)&= \xi_d+\frac{1}{N}\!\sum_{m=1}^N\int_{0}^{t}\!\mu_d(\mathcal{X}^{d,n}_N(s-),\mathcal{X}^{d,m}_N(s-))\,ds\nonumber+\frac{1}{N}\!\sum_{m=1}^N\int_{0}^{t}\!
	\sigma_d(\mathcal{X}^{d,n}_N(s-),\mathcal{X}^{d,m}_N(s-))\,dW^{d,n}(s)\nonumber
	\\
	&\quad +\frac{1}{N}\sum_{m=1}^N\int_{0}^{t}\int_{\R^d}
	\eta_d(\mathcal{X}^{d,n}_N(s-),\mathcal{X}^{d,m}_N(s-),z)\,\tilde{N}^{d,n}(dz,ds), \quad\,\,\,\, t \in [0,T], \qquad n = 1,...,N,
\end{align}
where $(W^{d,n})_{n=1,...,N}$ and $(\tilde{N}^{d,n})_{n=1,...,N}$ are independent copies of the Brownian motion and the compensated Poisson random measure in \eqref{eq:mckv}, respectively. In this case, the phenomenon of \emph{propagation of chaos} ensures that each particle $\mathcal{X}^{d,n}$ converges in law (and under suitable assumptions also in stronger norms) to the solution $X^{d}_{\infty}$ of \eqref{eq:mckv} (see \cite{Kac1956,McKean1967,Sznitman1991,Meleard1996,Mischler2015,Andreis2018,Lacker2018,Cavallazzi2022,Erny2022,Erny2022locLip,Cavallazzi2025,Hernandez2025}).

\begin{remark}\label{r01}
	The McKean--Vlasov SDEs \eqref{eq:mckv} can be equivalently written using the laws of $(	X^{d}_{\infty}(t))_{t \in [0,T]}$ instead of the expectation. Indeed, let $\mathcal{P}_2(\R^d)$ denote the set of probability measures $\lambda$ on $\R^d$ satisfying $\int_{\R^d} \Vert x \Vert^2 \lambda(dx) < \infty$ and define the functions $\R^d \times \mathcal{P}_2(\R^d) \ni (x,\lambda) \mapsto \widetilde{\mu}_d(x,\lambda) := \int_{\R^d} \mu_d(x,y) \, \lambda(dy) \in \R^d$, $\R^d \times \mathcal{P}_2(\R^d) \ni (x,\lambda) \mapsto \widetilde{\sigma}_d(x,\lambda) := \int_{\R^d} \sigma_d(x,y) \, \lambda(dy) \in \R^{d\times d}$, and $\R^d \times \mathcal{P}_2(\R^d) \times \R^d \ni (x,\lambda,z) \mapsto \widetilde{\eta}_d(x,\lambda,z) := \int_{\R^d} \eta_d(x,y,z) \, \lambda(dy) \in \R^d$. Then, it holds $\mathbb{P}$-a.s.~that
	\begin{align}
		X^{d}_{\infty}(t)&= \xi_d+\int_{0}^{t} \widetilde{\mu}_d\left(X^{d}_{\infty}(s-),\lambda_{X^{d}_{\infty}(s-)}\right) \,ds+\int_{0}^{t}
		\widetilde{\sigma}_d\left(X^{d}_{\infty}(s-),\lambda_{X^{d}_{\infty}(s-)}\right)\,dW^d(s)\nonumber
		\\
		&\quad +\int_{0}^{t}\int_{\R^d}\widetilde{\eta}_d\left(X^{d}_{\infty}(s-),\lambda_{X^{d}_{\infty}(s-)},z\right)\,\tilde{N}^d(dz,ds), \qquad t \in [0,T],
	\end{align}
	where $\lambda_{X^{d}_{\infty}(s-)}$ denotes the law of the $\R^d$-valued random variable $X^{d}_{\infty}(s-)$, which is defined as $\lambda_{X^{d}_{\infty}(s-)}(A) := \mathbb{P}[X^{d}_{\infty}(s-) \in A]$ for all Borel subsets $A \subseteq \R^d$.
\end{remark}

In the literature, only a few works address the numerical approximation of \emph{high-dimensional} McKean--Vlasov SDEs. In the \emph{absence of jumps}, \cite{HKN2022} showed that multilevel Picard (MLP) approximations overcome the curse of dimensionality for McKean--Vlasov SDEs with constant diffusion coefficients, which was extended in \cite{Neufeld2025} to nonconstant diffusion coefficients. More recently, \cite{HHL2024,Germain2022,PW2024} have applied deep learning methods to the numerical approximation of McKean--Vlasov equations. While \cite{HHL2024} established an a priori error estimate under suitable conditions, the numerical experiments in \cite{HHL2024,Germain2022,PW2024} remain restricted to relatively low dimensions, with dimensions up to $50$. On the theoretical side, \cite{NN2025} proved that rectified deep neural networks overcome the curse of dimensionality when approximating McKean--Vlasov SDEs, but without providing an explicit numerical algorithm for computing the solution of \eqref{eq:mckv}. A broader literature on numerical approximation schemes for McKean--Vlasov SDEs without jumps includes, e.g., \cite{BST2019,STT2019,RSZ2020optimalreg,RSZ2020pathreg,BH2021,BRRS2021,dRES2021,ST2021,BCNdRR2022,CdR2022,CNRS2022,LRS2022,RS2022,agarwal2023numerical,BRRS2023,CdR2023,chen2023wellposedness,dRST2023,RSZ2023}.

The extension to McKean--Vlasov SDEs \emph{with jumps} is particularly relevant for models combining mean-field interactions with sudden event-driven dynamics. Applications include spiking neuron models in biology \cite{Delarue2015,Andreis2018,Chevallier2017}, kinetic models for collision processes in gases, such as the BGK equation \cite{Bhatnagar1954,Butta2021}, household epidemic models \cite{Forien2022}, systemic-risk and dynamic banking models \cite{Benazzoli2020,Ichiba2019}, and mean-field games and control problems with jump diffusions \cite{Agram2023}. Numerical approximations of McKean--Vlasov SDEs with jumps have been developed based on Fourier-based Picard iterations \cite{Agarwal2021}, Euler--Maruyama and tamed or adaptive Euler-type schemes \cite{Frikha2021,Neelima2020,Qin2024,Tran2025,Zhu2026}, Milstein-type schemes \cite{Biswas2024}, and Ito--Taylor schemes \cite{Sun2021}. Deep learning methods have also been proposed for control problems involving conditional McKean--Vlasov jump diffusions \cite{Agram2025}.

In this paper, we develop an MLP algorithm for high-dimensional McKean--Vlasov SDEs with finite or infinite jump-activity, allowing the driving L\'evy measures to be finite or infinite. The construction builds on the scheme of \cite{Neufeld2025}, but incorporates additional Monte Carlo approximations for the jump compensator and, in the case of infinite activity, a truncation of small jumps. The corresponding analysis also requires new variance estimates for the resulting jump terms. We prove that the resulting MLP approximation overcomes the curse of dimensionality in the $L^2$-sense, i.e., the computational cost grows at most polynomially in the state-space dimension $d$ and the reciprocal $\varepsilon^{-1}$ of a prescribed error tolerance $\varepsilon>0$. Moreover, we demonstrate its practical applicability in two numerical examples of McKean--Vlasov SDEs in dimensions up to $200$.

The remainder of this article is structured as follows. In Section~\ref{sec:setting}, we outline the setting and main results stating that the MLP algorithm overcomes the curse of dimensionality when approximating the McKean--Vlasov SDE~\eqref{eq:mckv}. In Section~\ref{sec:numerics}, we illustrate in two numerical examples how the MLP algorithm can be applied to approximate the solution of \eqref{eq:mckv}. Finally, Section~\ref{sec:proofs} contains the proofs of the main results.

\subsection{Notations}
For every $d\in \N$, $x=(x_j)_{j\in [1,d]\cap\Z}$ let $\lVert x\rVert$ be the standard norm of $x$, i.e., $\lVert x\rVert=\sqrt{\sum_{i=1}^d \lvert x_i\rvert^2 }$.
For every $d\in \N$, $x=(x_{ij})_{i,j\in [1,d]\cap\Z}$ let $\lVert x\rVert$ be the Hilbert-Schmidt norm of $x$, i.e.,
$\lVert x\rVert=\sqrt{\sum_{i,j=1}^d \lvert x_{ij}\rvert^2 }$.
Let $\1_A$ denote the indicator function of the set $A$.
For conditional expectations and conditional variances write $\E_{\mathcal{G}}[\cdot]$ and $\V_\mathcal{G}[\cdot]$ instead of $\E[\cdot |\mathcal{G}]$ and $\V[\cdot|\mathcal{G}]$
to lighten the notation. Here, the conditional variances are defined as follows: for any $d\in \N$, for any probability space $(\Omega,\mathcal{F},\P)$, for any random variable $X\colon \Omega\to \R^d$, and for any sigma-algebra $\mathcal{G}\subseteq \mathcal{F}$ we have
\begin{align}\V_\mathcal{G} [X]= \E_\mathcal{G}\!\left[\lVert X-\E_\mathcal{G}[X]\rVert^2\right] .
\end{align}
We write $\N=\{1,2,\ldots\}$, $\N_0=\N\cup\{0\}$. When applying a result we often use a phrase like ``Lemma 3.8 with $d\gets (d-1)$''
that should be read as ``Lemma 3.8 applied with $d$ (in the notation of Lemma 3.8) replaced
by $(d-1 )$ (in the current notation)''.
\section{Setting and main results}
\label{sec:setting}

In this section, we introduce the setting and state the convergence results of our MLP algorithm.

\begin{setting}[Common setting]\label{c01}
Let $T\in (0,\infty)$, 
$c\in [1,\infty)$, $\Theta=\cup_{n\in \N}(\N_0)^n$.
For every $d\in \N$ let $\xi_d\in \R^d$,
$\mu_d 
\in C(\R^d\times \R^d,\R^d)$, $\sigma_d\in C(\R^d\times \R^d,\R^{d\times d})$, $\eta_d\in C(\R^d\times\R^d\times\R^d,\R^d)$. Assume for all $d\in \N$, $x_1,x_2,y_1,y_2\in \R^d$ that
\begin{align}
\max \left\{
\left\lVert
\mu_d(x_1,x_2)-\mu_d(y_1,y_2)\right\rVert
,
\left\lVert
\sigma_d(x_1,x_2)-\sigma_d(y_1,y_2)\right\rVert
\right\}\leq c \lVert x_1-y_1\rVert+c \lVert x_2-y_2\rVert.
\label{e11}
\end{align}
For every $K\in \N$ let $\rdown{\cdot}_K : [0,T] \rightarrow [0,T]$ satisfy for all $t\in (0,T]$ that 
$\rdown{t}_K=\sup( [0,t)\cap\{0,\frac{T}{K},\frac{2T}{K}\ldots,T\})$ and
$\rdown{0}_K=0$. For convenience we write for every $t\in [0,T]$ that  $\rdown{t}_\infty=t $. For every $m\in \N$, 
$K\in \N\cup\{\infty\}$, and every
c\`adl\`ag function $f\colon [0,T]\to \R^m$ we write $f(\rdown{s-}_K)=\lim _{r\uparrow s}f(\rdown{r}_K)$ and $f(\rdown{0-}_K)=f(0)$.
Let $(\Omega,\mathcal{F}, \P, (\F_t)_{t\in [0,T]})$ be a filtered probability space which satisfies the usual conditions\footnote{Let $T \in [0,\infty)$ and ${\bf \Omega} = (\Omega,\mathcal{F},\P, (\F_t)_{t\in[0,T]})$ be 
	a filtered probability space. We say that ${\bf \Omega}$
	satisfies the usual conditions if and only if 
	$\{ A \subseteq \Omega: \exists N \in \mathcal{F} \text{ with } A \subseteq N \text{ and } \P(N) = 0 \} \subseteq \F_0$ and $\forall \, t\in [0,T) \colon \F_t 
	= \cap_{ s \in (t,T] } \F_s$.}. For every $d\in \N$ let 
$W^{d,\theta}=(W^{d,\theta}(t))_{t\in [0,T]}\colon [0,T]\times\Omega\to \R^d$, $\theta\in \Theta$, be independent standard $(\F_t)_{t\in [0,T]}$-Brownian motions. 
Let $\unif^\theta\colon \Omega\to [0,1]$, $\theta\in \Theta$, be independent and identically distributed random variables satisfying for all $\theta\in \Theta$, $t\in [0,1]$ that \begin{align}\label{u01b}
\P(\unif^{\theta}\leq t)=t.
\end{align}
For every $d\in \N$ let $\nu^d\colon \mathcal{B}(\R^d)\to[0,\infty] $ be a L\'evy measure on $\R^d$ and let $N^{d,\theta}\colon 
\mathcal{B}(\R^d\times [0,T])\to \N_0\cup \{\infty\}
$, $\theta\in \Theta$, be independent $(\F_t)_{t\in [0,T]}$-Poisson random measures on 
$\R^d\times [0,T]$
with non-trivial compensator  measure $\nu^d(dz)\,dt$. For every $d\in \N$, $\theta\in \Theta$ denote by 
\begin{align}
\tilde{N}^{d,\theta}(dz, dt) := N^{d,\theta}(dz, dt) - \nu^d(dz) dt\label{s03b} 
\end{align}
the compensated Poisson random measure of $N^{d,\theta}$.
Assume for all 
$d\in \N$,
$x_1,x_2,y_1,y_2\in \R^d$ that
\begin{align}
\int_{\R^d} \left\lVert\eta_d(x_1,x_2,z)-\eta_d(y_1,y_2,z)\right\rVert^2\nu^d(dz)\leq c \lVert x_1-y_1\rVert^2 + c \lVert x_2-y_2\rVert^2.\label{e11b}
\end{align}
Assume for all $d\in \N$ that
\begin{align} \lVert\xi_d\rVert+ \lVert\mu_d(0,0)\rVert+\lVert\sigma_d(0,0)\rVert+
\left(
\int_{\R^d} \left\lVert\eta_d(0,0,z)\right\rVert^2\nu^d(dz)
\right)^\frac{1}{2}\leq cd^c.\label{e12}
\end{align}
\end{setting}
\begin{theorem}[Existence and uniqueness of the exact solution --{\cite[Theorem~2.1]{Neelima2020}}]\label{b02a}
Assume Setting~\ref{c01}.
Then for every  $d\in \N$ there exists a unique $(\F_t)_{t\in [0,T]}$-adapted process $X^{d}_\infty=(X^d_\infty(t))_{t\in [0,T]}$ with c\`adl\`ag sample paths satisfying $
\sup_{s\in [0,T]}\lVert X^d_\infty(s)\rVert_{L^2(\P;\R^d)}<\infty
$ such that $\P$-a.s.\ for all $t\in [0,T]$ we have that  
\begin{align}
\label{eq:thm_mckv}
X^{d}_{\infty}(t)&= \xi_d+\int_{0}^{t}
\E\! \left[\mu_d (x, X^d_{\infty}(s-) )\right]\!|_{x=X^d_{\infty}(s-)}\,ds\nonumber\\
&\quad 
+\int_{0}^{t}
\E\! \left[\sigma_d (x, X^d_{\infty}(s-) )\right]\!|_{x=X^d_{\infty}(s-)}\,dW^{d,0}(s)\nonumber
\\
&\quad +\int_{0}^{t}\int_{\R^d}
\E\! \left[\eta_d (x, X^d_{\infty}(s-),z )\right]\!|_{x=X^d_{\infty}(s-)}\,\tilde{N}^{d,0}(dz,ds).
\end{align}
\end{theorem}
\begin{proof}[Proof of Theorem~\ref{b02a}]For every $d\in \N$ define $\widetilde{\mu}_d$,
$\widetilde{\sigma}_d$,
$\widetilde{\eta}_d$ as in Remark~\ref{r01}. Using \eqref{e11}, \eqref{e11b}, and \eqref{e12} we can check that for every $d\in \N$
Conditions A--1 to A--4 in \cite[Page 4]{Neelima2020} are satisfied with 
$p_0\gets 2$,
$x_0\gets \xi_d$,
$b\gets \widetilde{\mu}_d$,
$\sigma\gets \widetilde{\sigma}_d$,
$\gamma\gets \widetilde{\eta}_d$. Thus, \cite[Theorem 2.1]{Neelima2020} completes the proof of Theorem~\ref{b02a}.
\end{proof}

\begin{setting}\label{s01c}Assume Setting \ref{c01}. For every 
$d\in \N$, $\delta\in [0,\infty)$
 let
\begin{align}
A^d_\delta= \{z\in \R^d\colon \lVert z\rVert\geq \delta\}. \label{p03}
\end{align}
For every 
$d\in \N$,
$\delta\in [0,\infty)$
let $\rvV^{d,\theta,\delta}_{s}\colon \Omega\to \R^d$, $s\in [0,T]$, $\theta\in \Theta$, be independent and identically distributed random variables which satisfy for all 
$s\in [0,T]$, $\theta\in \Theta$, $B\in \mathcal{B}(\R^d)$ that 
\begin{align}\label{f09}
\P(\rvV^{d,\theta,\delta}_{s}\in B)=\begin{cases}
\displaystyle
\frac{\nu^d(B)}{\nu^d(\R^d)}&\text{ if }\nu^d(\R^d)<\infty,\, \delta=0\\
\displaystyle
\frac{\nu^d(B\cap A^d_\delta)}{\nu^d( A^d_\delta)}&\text{ if }\nu^d(\R^d)=\infty,\,\delta>0,\text{ and }\nu^d( A^d_\delta)>0\\
\1_{B}(\delta(1,0,\ldots,0))&\text{ else}.
\end{cases}
\end{align}
Assume for all $d\in \N$, $\delta\in [0,\infty)$ that the families
$(\unif^\theta)_{\theta\in \Theta}$,
$(W^{d,\theta})_{\theta\in \Theta}$,
$(N^{d,\theta})_{\theta\in \Theta}$, and
$(\rvV^{d,\theta,\delta}_{s})_{\theta\in\Theta,s\in[0,T]}$ are independent.
For every $d\in \N$, $ \delta\in[0,\infty)$
let $\mathcal{M}^{d}_{\delta}\in \N$ satisfy that
\begin{align}\label{s05b}
\mathcal{M}^{d}_{\delta}=
\begin{cases}
\lceil\nu^d(\R^d)\rceil
&\text{ if }\nu^d(\R^d)<\infty,\, \delta=0\\
\lceil\nu^d(A^d_\delta )\rceil&\text{ if }\nu^d(\R^d)=\infty,\,\delta>0,\text{ and }\nu^d( A^d_\delta)>0\\
1&\text{ else}.
\end{cases}
\end{align}
Let $X_{n,m,K,\delta}^{d,\theta}\colon [0,T]\times \Omega\to \R^d$, $\theta\in \Theta$, $d,m,K\in \N$, 
$\delta\in[0,\infty)$,
$n\in \N_0\cup\{-1\}$, satisfy for all $\theta\in \Theta$,
$d,K,m,n\in \N$, $t\in [0,T]$ that 
\begin{align}\label{s05}
X^{d,\theta}_{-1,m,K,\delta}(t)=
X^{d,\theta}_{0,m,K,\delta}(t)=0
\end{align}
and
\begin{align} \begin{split} 
&
X^{d,\theta}_{n,m,K,\delta}(t)= \xi_d 
%+t\mu_d(0,0)+\sigma_d(0,0)W^{d,\theta}(t)
%+\int_{0}^{t}\int_{A^d_\delta}\eta_d(0,0,z)N^{d,\theta}(dz,ds)
\\
&+ \sum_{\ell=0}^{n-1}\sum_{k=1}^{m^{n-\ell}}
\frac{t}{m^{n-\ell}}
\Bigl[
\mu_d (X_{\ell,m,K,\delta}^{d,\theta} (\rdown{t\unif^{(\theta,n,k,\ell)}-}_K ),X_{\ell,m,K,\delta}^{d,(\theta,n,k,\ell)} (\rdown{t\unif^{(\theta,n,k,\ell)}-}_K ))\\
&\qquad\qquad\qquad
-\1_{\N}(\ell)
\mu_d (X_{\ell-1,m,K,\delta}^{d,\theta} (\rdown{t\unif^{(\theta,n,k,\ell)}-}_K ),X_{\ell-1,m,K,\delta}^{d,(\theta,n,k,\ell)} (\rdown{t\unif^{(\theta,n,k,\ell)}-}_K ))
\Bigr]
\\
&+ \sum_{\ell=0}^{n-1}\sum_{k=1}^{m^{n-\ell}}\int_{0}^{t}
\frac{1}{m^{n-\ell}}
\Bigl[
\sigma_d(X_{\ell,m,K,\delta}^{d,\theta} (\rdown{s-}_K ),X_{\ell,m,K,\delta}^{d,(\theta,n,k,\ell)} (\rdown{s-}_K ))\\
&\qquad\qquad\qquad
-\1_{\N}(\ell)
\sigma_d(X_{\ell-1,m,K,\delta}^{d,\theta} (\rdown{s-}_K ),X_{\ell-1,m,K,\delta}^{d,(\theta,n,k,\ell)} (\rdown{s-}_K ))\Bigr]
dW^{d,\theta}(s)
\\
&+ \sum_{\ell=0}^{n-1}\sum_{k=1}^{m^{n-\ell}}\int_{0}^{t}\int_{A^d_\delta}
\frac{1}{m^{n-\ell}}
\Bigl[\eta_d(X_{\ell,m,K,\delta}^{d,\theta} (\rdown{s-}_K ),X_{\ell,m,K,\delta}^{d,(\theta,n,k,\ell)} (\rdown{s-}_K ),z)\\
&\qquad\qquad\qquad
-\1_{\N}(\ell)
\eta_d(X_{\ell-1,m,K,\delta}^{d,\theta} (\rdown{s-}_K ),X_{\ell-1,m,K,\delta}^{d,(\theta,n,k,\ell)} (\rdown{s-}_K ),z)\Bigr]
N^{d,\theta}(dz,ds)\\
&- \sum_{\ell=0}^{n-1}\sum_{k=1}^{m^{n-\ell}}\sum_{i=1}^{\mathcal{M}^d_\delta}\int_{0}^{t}
\frac{\nu^d(A^d_\delta)}{m^{n-\ell}\mathcal{M}^d_\delta}
\Bigl[\eta_d(X_{\ell,m,K,\delta}^{d,\theta} (\rdown{s-}_K ),X_{\ell,m,K,\delta}^{d,(\theta,n,k,\ell)} (\rdown{s-}_K ),\rvV^{d,(\theta,n,k,\ell,i),\delta}_{\rdown{s}_K })\\
&\qquad\qquad\qquad
-
\1_{\N}(\ell)\eta_d(X_{\ell-1,m,K,\delta}^{d,\theta} (\rdown{s-}_K ),X_{\ell-1,m,K,\delta}^{d,(\theta,n,k,\ell)} (\rdown{s-}_K ),\rvV^{d,(\theta,n,k,\ell,i),\delta}_{\rdown{s}_K })\Bigr]\,ds.
\end{split}\label{s02}
\end{align}
\end{setting}
\begin{remark}
The numerical scheme \eqref{s02} is designed for both cases: finite jump-activity ($\nu^d(\R^d)<\infty$) and infinite jump-activity ($\nu^d(\R^d)=\infty$). To make the case distinction, we introduce \eqref{f09}. If $\nu^d(\R^d)<\infty$, we do not need to truncate the measure and take $\delta=0$.
If $\nu^d(\R^d)=\infty$, we need to truncate the measure and take $\delta > 0$ with $\nu^d (A^d_\delta)>0$. Moreover, we include the Dirac measure in \eqref{f09} for the case $\nu^d(\R^d)=\infty$ but $\nu^d(A^d_\delta)=0$, which ensures that the law of the random variables $\rvV^{d,\theta,\delta}_s$ is well-defined, for small $\delta > 0$. The latter is however irrelevant for practical applications.
\end{remark}

\subsection{Finite jump-activity}

Theorem~\ref{p01} below gives an upper bound for the distance between the MLP approximation and the exact solution if the L\'evy measure is finite. Since $\nu^d(\R^d)<\infty$, no truncation of jumps is needed and we take $\delta=0$. Moreover, Theorem \ref{p02} below provides the complexity of the algorithm and shows that, when the input functions can be calculated with polynomial complexity in the dimension (see \eqref{s12}) and the L\'evy measure grows polynomially in the dimension (see \eqref{s13}), then the MLP approximations can also  be evaluated with polynomial complexity (see \eqref{c19d}), i.e., the algorithm overcomes the curse of dimensionality. 

\begin{theorem}[Global error]\label{p01}
Assume Setting~\ref{s01c} and assume for all $d\in \N$ that $\nu ^d(\R^d)\in (0,\infty) $. Moreover, for every $d\in \N$ let $X^d_\infty$ be the process in Theorem~\ref{b02a}. Then there exists $\kappa\in (0,\infty)$ such that  for all 
$d, n,m,K\in \N$,  $t\in [0,T]$ we have that
\begin{align}
\left\lVert 
X^{d,0}_{n,m,K,0}(t)-X^{d}_\infty(t)
\right\rVert_{L^2(\P;\R^d)}\leq \kappa^n d^\kappa e^\frac{m}{2}m^{-\frac{n}{2}}+\frac{\kappa d^\kappa}{\sqrt{K}}.
\end{align}
\end{theorem}

\begin{theorem}[Computational complexity]\label{p02}
Assume Setting~\ref{s01c}. For every $d\in \N$ let $X^d_\infty$ be 
 the process in Theorem~\ref{b02a}. For every $d\in \N$ let  $\mathsf{cost}_{\mu_d}$,
$\mathsf{cost}_{\sigma_d}$,
$\mathsf{cost}_{\eta_d}$
be an upper bound of the computational cost to evaluate  $\mu_d(x,y)$, $\sigma_d(x,y)$, $\eta_d(x,y,z)$, respectively,  
for any $x,y,z\in \R^d$. For every $d\in \N$ let $\mathsf{cost}_{\mathsf{rv},d} $ be an upper bound of the computational cost to generate a $d$-dimensional random variable. Assume for all $d\in \N$ that 
\begin{align}\label{s12}
\mathsf{cost}_{\mu_d}+
\mathsf{cost}_{\sigma_d}+
\mathsf{cost}_{\eta_d}+\mathsf{cost}_{\mathsf{rv},d}\leq cd^c
\end{align} and
\begin{align}\label{s13}
0<\nu^d(\R^d)\leq cd^c.
\end{align}
For every 
$d,m,K\in \N$, $n\in \N_0\cup\{-1\}$ let 
$\mathsf{C}_{n,m,K}^d$ denote the expectation of the computational cost to construct the whole discrete process $(X_{n,m,K,0}^{d,\theta}(\frac{kT}{K}))_{k\in [1,K]\cap\Z}$ given that we have prepared a discrete Brownian path $W^{d,\theta}(\frac{kT}{K}))_{k\in [1,K]\cap\Z} $ and a discretization
$(N^{d,\theta}(dz,(\frac{(k-1)T}{K},\frac{kT}{K} ]))_{k\in [1,K]\cap\Z}$
of the Poisson measure. 
Then there exist $\kappa\in (0,\infty)$, $(\mathfrak{C}_\gamma)_{\gamma\in (0,1)}\subseteq (0,\infty)$,
$(\mathsf{n}_{d,\epsilon})_{d\in \N,\epsilon\in (0,1)}\subseteq \N$ such that for all $d\in \N$, $\gamma,\epsilon\in (0,1)$ we have that
\begin{align}
	\label{c19c}
\sup_{k\in [\mathsf{n}_{d,\epsilon},\infty)\cap\Z}
\sup _{t\in [0,T]}
\left\lVert
X^{d,0}_{k,k,k^k,0}(t)-X^{d}_\infty(t)
\right\rVert_{L^2(\P;\R^d)}<\epsilon
\end{align}
and 
\begin{align}
	\label{c19d}
\mathsf{C}_{\mathsf{n}_{d,\epsilon},\mathsf{n}_{d,\epsilon}, (\mathsf{n}_{d,\epsilon})^{\mathsf{n}_{d,\epsilon}}}^d
\leq d^{\kappa}\mathfrak{C}_\gamma\epsilon^{-4-\gamma}.
\end{align}
\end{theorem}
\subsection{Infinite jump-activity}

Theorem~\ref{p01b} below gives an upper bound for the distance between the MLP approximation and the exact solution if the L\'evy measure is infinite. Since $\nu^d(\R^d)=\infty$, we truncate the jumps and  take $\delta>0$. Moreover, Theorem \ref{p02b} below provides the complexity of the algorithm and shows that, when the input functions can be calculated with polynomial complexity in the dimension (see \eqref{s12}) and the L\'evy measure satisfies the additional conditions~\eqref{e13} and \eqref{s13b}, then the MLP approximations can also be evaluated with polynomial complexity (see \eqref{c19e}), i.e., the algorithm overcomes the curse of dimensionality also in the case of infinite activity.

\begin{setting}[Additional assumption for infinite jump-activity]\label{s01}
Assume Setting~\ref{s01c}, assume for all $d\in \N$ that $\nu^d(\R^d)=\infty$, and let $q\in (0,\infty)$. Furthermore, assume for all $d\in \N$, $x,y\in \R^d$, $\delta\in (0,\infty)$ that
\begin{align}
\int_{\R^d\setminus A^d_\delta} \lVert\eta_d(x,y,z)\rVert^2\,\nu^d(dz)
\leq c\delta^{q}(cd^c+\lVert x\rVert^2+\lVert y\rVert^2).\label{e13}
\end{align}

\end{setting}

\begin{theorem}[Global error]\label{p01b}
Assume Setting~\ref{s01}. For every $d\in \N$ let $X^d_\infty$ be 
 the process in Theorem~\ref{b02a}. Then there exists $\kappa\in (0,\infty)$ such that for all 
$d\in \N$, $\delta\in (0,\infty)$, $n,m,K\in \N$, $t\in [0,T]$ we have that
$
\left\lVert 
X^{d,0}_{n,m,K,\delta}(t)-X^{d}_\infty(t)
\right\rVert_{L^2(\P;\R^d)}\leq \kappa^n d^\kappa e^\frac{m}{2}m^{-\frac{n}{2}}+\delta^\frac{q}{2}\kappa d^\kappa+\frac{\kappa d^\kappa}{\sqrt{K}}.
$
\end{theorem}

\begin{theorem}[Computational complexity]\label{p02b}
Assume Setting~\ref{s01}. For every $d\in \N$ let $X^d_\infty$ be 
 the process in Theorem~\ref{b02a}. For every $d\in \N$ let $\mathsf{cost}_{\mu_d}$,
$\mathsf{cost}_{\sigma_d}$, $\mathsf{cost}_{\eta_d}$ be an upper bound of the computational cost to evaluate  $\mu_d(x,y)$, $\sigma_d(x,y)$, $\eta_d(x,y,z)$, respectively, for any $x,y,z\in \R^d$. For every $d\in \N$ let $\mathsf{cost}_{\mathsf{rv},d} $ be an upper bound of the computational cost to generate a $d$-dimensional random variable. Assume for all $d\in \N$ that 
\begin{align}\label{s12b}
\mathsf{cost}_{\mu_d}+
\mathsf{cost}_{\sigma_d}+
\mathsf{cost}_{\eta_d}+\mathsf{cost}_{\mathsf{rv},d}\leq cd^c
\end{align} and
\begin{align}\label{s13b}
\int_{\R^d}(1\wedge \lVert z\rVert^2)\,\nu^d(dz)\leq cd^c.
\end{align}
For every $\delta\in (0,\infty)$, $d,m,K\in \N$, $n\in \N_0\cup\{-1\}$ let $\mathsf{C}_{n,m,K,\delta}^d$ denote the expectation of the computational cost to construct the whole discrete process $(X_{n,m,K,\delta}^{d,\theta}(\frac{kT}{K}))_{k\in [1,K]\cap\Z}$ given that we have prepared a discrete Brownian path $W^{d,\theta}(\frac{kT}{K}))_{k\in [1,K]\cap\Z} $ and a discretization
$(N^{d,\theta}(dz,(\frac{(k-1)T}{K},\frac{kT}{K} ]))_{k\in [1,K]\cap\Z}$
of the Poisson measure. Then there exist $\kappa\in (0,\infty)$, $(\mathfrak{C}_\gamma)_{\gamma\in (0,1)}\subseteq(0,\infty)$, $(\mathsf{n}_{d,\epsilon})_{d\in \N,\epsilon\in (0,1)}\subseteq \N$ such that for all $d\in \N$, $\gamma,\epsilon\in (0,1)$ we have that
\begin{align}
	\label{c19f}
\sup_{k\in [\mathsf{n}_{d,\epsilon},\infty)\cap\Z}
\sup _{t\in [0,T]}
\left\lVert
X^{d,0}_{k,k,k^k, k^{-{k/q}}}(t)-X^{d}_\infty(t)
\right\rVert_{L^2(\P;\R^d)}<\epsilon
\end{align}
and 
\begin{align}
	\label{c19e}
\mathsf{C}_{\mathsf{n}_{d,\epsilon},\mathsf{n}_{d,\epsilon}, (\mathsf{n}_{d,\epsilon})^{\mathsf{n}_{d,\epsilon}}, (\mathsf{n}_{d,\epsilon})^{- \mathsf{n}_{d,\epsilon}/q }}^d
\leq d^{\kappa}\mathfrak{C}_\gamma\epsilon^{-4(1+\frac{1}{q})-\gamma}.
\end{align}
\end{theorem}
\subsection{Motivations of the scheme}
Let us give some motivations for the numerical scheme \eqref{s02}. First, we start with an iteration defined as follows $X^d_{0}(t)=0$ and
\begin{align}
X^d_n(t) &= \xi_d+\int_{0}^{t}\E \!\left[  \mu_d(x,X^d_{n-1}(s-))\right]\!|_{x=X^d_{n-1}(s-)}\,ds
+ \int_{0}^{t}\E \!\left[\sigma_d(x,X^d_{n-1}(s-))\right]\!|_{x=X^d_{n-1}(s-)}\,dW^{d,0}(s)\nonumber\\
&\quad + \int_{0}^{t}\int_{\R^d}\E \!\left[\eta_d(x,X^d_{n-1}(s-),z)\right]\!|_{x=X^d_{n-1}(s-)}\,\tilde{N}^{d,0}(dz,ds)
\end{align} for $n\in \N$. Note that the right-hand side can be written as a telescoping sum: 
\begin{align}
 &X^d_n(t) =
\xi_d+\int_{0}^{t}\sum_{\ell=0}^{n-1}\E \!\left[  \mu_d(x,X^d_{\ell}(s-))\right]\!|_{x=X^d_{\ell}(s-)}-
\1_{\N}(\ell)
\E \!\left[  \mu_d(x,X^d_{\ell-1}(s-))\right]\!|_{x=X^d_{\ell-1}(s-)}
\,ds\nonumber\\
& 
+ \int_{0}^{t}\sum_{\ell=0}^{n-1}\E \!\left[\sigma_d(x,X^d_{\ell}(s-))\right]\!|_{x=X^d_{\ell}(s-)}
-\1_{\N}(\ell)\E \!\left[\sigma_d(x,X^d_{\ell-1}(s-))\right]\!|_{x=X^d_{\ell-1}(s-)}
\,dW^{d,0}(s)\nonumber\\
& 
+ \int_{0}^{t}\int_{\R^d}\sum_{\ell=0}^{n-1}\E \!\left[\eta_d(x,X^d_{\ell}(s-),z)\right]\!|_{x=X^d_{\ell}(s-)}
-\1_{\N}(\ell)\E \!\left[\eta_d(x,X^d_{\ell-1}(s-),z)\right]\!|_{x=X^d_{\ell-1}(s-)}
\,\tilde{N}^{d,0}(dz,ds)\label{s10a}
\end{align} with $X^d_{-1}:=0$, where the index $\ell$ can be interpreted as the corresponding level. 

To motivate our scheme, let us first consider the case \emph{without jumps}, i.e., the case when $\eta_d=0$, which was analysed in \cite{Neufeld2025} and is recalled here for convenience of the reader.
In this case, we replace in \eqref{s10a} all expectations and integrals by corresponding Monte-Carlo approximations to get our multilevel Picard approximation scheme \eqref{s02}. More precisely, the ideas can be described as follows. First, we approximate an integral $\int_0^t h(s)\,ds$ for a random function $h$ by 
$th(t\mathfrak{u})$ where $\mathfrak{u}$ is uniformly distributed on $[0,1]$ and independent of $h$.  
In order to approximate the expectations
$
 \E \!\left[  \mu_d(x,X^d_{\ell}(s))\right]\!|_{x=X^d_{\ell}(s)}$ and
$
 \E \!\left[  \sigma_d(x,X^d_{\ell}(s))\right]\!|_{x=X^d_{\ell}(s)}$
we take one sample $X^{d,\theta}_{\ell,m,K,0}$, which is also generated by the Brownian motion $W^{d,\theta}$ and which corresponds to $X^d_\ell(s)$ outside the expectation, and
$m^{n-\ell}$ independent samples $X_{\ell,m,K,0}^{d,(\theta,n,k,\ell)}$, $k\in \{1,2,\ldots,m^{n-\ell}\}$, which are respectively generated by the independent Brownian motions $W^{d,(\theta,n,k,\ell)}$, $k\in \{1,2,\ldots,m^{n-\ell}\}$, and which correspond to $X^d_\ell(s)$ inside the expectation. The main novelty of the numerical scheme in \eqref{s02} compared to \cite[Theorem~1.1]{HKN2022} lies in storing the entire process $(X_{n,m,K,0}^{d,\theta}(t))_{t\in [0,T]}$ rather than only one random variable $X_{n,m,K,0}^{d,\theta}(t)$. Indeed, by retaining the whole path $(X_{n,m,K,0}^{d,\theta}(t))_{t\in [0,T]}$, we can approximate the stochastic integral arising in McKean--Vlasov SDEs with nonconstant diffusion coefficient, which was not necessary for the constant diffusion case in \cite{HKN2022} where the stochastic integral is just a normally distributed random variable. Since we cannot save a process $Y$ at all time points $t\in [0,T]$ we only save the discrete trajectory $(Y(\rdown{t}_K))_{t\in [0,T]}$. Here and in \eqref{s02} the index $K$ denotes the number of steps of the time discretization to approximate the stochastic integrals. While this approach may seem inefficient at first sight, our numerical experiments  demonstrate its effectiveness even in high dimensions, see the simulations for the case without jumps in \cite{Neufeld2025}. Moreover, for many applications of McKean--Vlasov SDEs, one requires the computation of the entire solution trajectory $(X^{d,\theta}_{\infty}(t))_{t\in [0,T]}$ instead of one single space-time point $X^{d,\theta}_{\infty}(t)$, which our method can do.

Now, we consider the case \emph{with jumps}, which is the main aim of this paper. In this case, we need to find approximations of the $\eta_d$-term
\begin{align}
\int_{0}^{t}\int_{\R^d}\sum_{\ell=0}^{n-1}\E \!\left[\eta_d(x,X^d_{\ell}(s-),z)\right]\!|_{x=X^d_{\ell}(s-)}
-\1_{\N}(\ell)\E \!\left[\eta_d(x,X^d_{\ell-1}(s-),z)\right]\!|_{x=X^d_{\ell-1}(s-)}
\,\tilde{N}^{d,0}(dz,ds)\label{k15}
\end{align}
in \eqref{s10a}. First, we consider the case when $\nu^d(\R^d)<\infty$. In this case, a truncation of the jumps is not needed and we take $\delta=0$ (see Theorems~\ref{p01} and~\ref{p02}) implying $A^d_\delta=\R^d$ (recall~\eqref{p03}). Thus, \eqref{k15} becomes
\begin{align}
\int_{0}^{t}\int_{A^d_\delta}\sum_{\ell=0}^{n-1}\E \!\left[\eta_d(x,X^d_{\ell}(s-),z)\right]\!|_{x=X^d_{\ell}(s-)}
-\1_{\N}(\ell)\E \!\left[\eta_d(x,X^d_{\ell-1}(s-),z)\right]\!|_{x=X^d_{\ell-1}(s-)}
\,\tilde{N}^{d,0}(dz,ds)\label{k15c}
\end{align}
In order to approximate the first term in the sum we write
\begin{align}
&
\int_{0}^{t}\int_{A^d_\delta}\E \!\left[\eta_d(x,X^d_{\ell}(s-),z)\right]\!|_{x=X^d_{\ell}(s-)}
\tilde{N}^{d,0}(dz,ds)\nonumber\\
&
=
\int_{0}^{t}\int_{A^d_\delta}\E \!\left[\eta_d(x,X^d_{\ell}(s-),z)\right]\!|_{x=X^d_{\ell}(s-)}
{N}^{d,0}(dz,ds)\nonumber\\&\quad
-
\int_{0}^{t}\int_{A^d_\delta}\E \!\left[\eta_d(x,X^d_{\ell}(s-),z)\right]\!|_{x=X^d_{\ell}(s-)}
\nu^d(dz)\,ds\label{k15b}
\end{align} using \eqref{s03b}.
Now, we also replace the expectations by Monte-Carlo sums as done for the non-jump term. The main issue here is to approximate the integral $\int_{A^d_\delta}\ldots\nu^d(dz)\,ds$. This can also be approximated by a Monte-Carlo sum: for a function $h$ we write
\begin{align}
\int_{A^d_\delta}h(\rdown{s-}_K,z)\,\nu^d(dz)= \nu^d(A^d_\delta)\E [h(\rdown{s-}_K,\rvV_{\rdown{s}_K}^{d,0,\delta})]\approx\frac{1}{\mathcal{M}^d_\delta}\sum_{i=1}^{\mathcal{M}^d_\delta}\nu^d(A^d_\delta)h(\rdown{s-}_K,\rvV_{\rdown{s}_K}^{d,i,\delta})
\end{align}
for $\rvV^{d,i,\delta}_{\rdown{s}_K}$ defined as in \eqref{f09}, also see Remark~\ref{r01b}. We do the same thing for the second term in the sum in \eqref{k15}. 
Moreover, similar to the case without jumps discussed above, we approximate integrals and expectations by Monte Carlo sums and use the whole time-discretization of the solution.
Plugging all of these approximations into \eqref{k15} we then obtain
\begin{align}
& \sum_{\ell=0}^{n-1}\sum_{k=1}^{m^{n-\ell}}\int_{0}^{t}\int_{A^d_\delta}
\frac{1}{m^{n-\ell}}
\Bigl[\eta_d(X_{\ell,m,K,\delta}^{d,\theta} (\rdown{s-}_K ),X_{\ell,m,K,\delta}^{d,(\theta,n,k,\ell)} (\rdown{s-}_K ),z)\nonumber\\
&\qquad\qquad\qquad
-\1_{\N}(\ell)
\eta_d(X_{\ell-1,m,K,\delta}^{d,\theta} (\rdown{s-}_K ),X_{\ell-1,m,K,\delta}^{d,(\theta,n,k,\ell)} (\rdown{s-}_K ),z)\Bigr]
N^{d,\theta}(dz,ds)\nonumber\\
&- \sum_{\ell=0}^{n-1}\sum_{k=1}^{m^{n-\ell}}\sum_{i=1}^{\mathcal{M}^d_\delta}\int_{0}^{t}
\frac{\nu^d(A^d_\delta)}{m^{n-\ell}\mathcal{M}^d_\delta}
\Bigl[\eta_d(X_{\ell,m,K,\delta}^{d,\theta} (\rdown{s-}_K ),X_{\ell,m,K,\delta}^{d,(\theta,n,k,\ell)} (\rdown{s-}_K ),\rvV^{d,(\theta,n,k,\ell,i),\delta}_{\rdown{s}_K })\nonumber\\
&\qquad\qquad\qquad
-
\1_{\N}(\ell)\eta_d(X_{\ell-1,m,K,\delta}^{d,\theta} (\rdown{s-}_K ),X_{\ell-1,m,K,\delta}^{d,(\theta,n,k,\ell)} (\rdown{s-}_K ),\rvV^{d,(\theta,n,k,\ell,i),\delta}_{\rdown{s}_K })\Bigr]\,ds.\label{s29}
\end{align}
This yields the approximation in \eqref{s02}. This choice of $\mathcal{M}^d_\delta$ in \eqref{s05b} is based on a technical reason: $\mathcal{M}^d_\delta$ should be large enough so that we can bound the
left-hand side of \eqref{k16} by the right-hand side of \eqref{k16}, which allows us to close the iteration.
Also note that from the technical point of view the main difficulty in the jump case is seemingly to estimate the variance of the terms in \eqref{s29}. 
The main idea here is to realize that the integrals in \eqref{s29} are in fact  discrete sums of conditionally independent random variables (see also \eqref{k16}).

Next, we consider the case when $\nu^d(\R^d)=\infty$.
Following, e.g., \cite[Sections~4.2 and~4.4]{gonon2023deep}, we consider the set $A_\delta^d$ in \eqref{p03} with $\delta\in (0,\infty)$ and approximate \eqref{k15} by \eqref{k15c}.
Note that, by \eqref{p03} and
the definition of L\'evy measures, we have for all $d\in \N$, $\delta\in (0,\infty)$
that
\begin{align}
\nu^d(A^d_\delta)&\leq 
\nu^d(
\{z\in \R^d\colon \lVert z\rVert\geq \min \{1,\delta\}\})\nonumber\\&\leq \frac{1}{\min \{1,\delta^2\}}\int_{\lVert z\rVert\geq \min \{1,\delta\}}(1\wedge\lVert z\rVert^2)\,\nu^d(dz)
<\infty.\label{p04b}
\end{align}
In this case we also replace the expectations and the integral $\int_{A^d_\delta}\ldots \nu^d(dz)$ in \eqref{k15b} by the corresponding Monte Carlo sums and end up with the approximations in \eqref{s29}. This allows us to derive the numerical scheme in \eqref{s02} for both cases: finite jump-activity ($\nu^d(\R^d)<\infty$) and infinite jump-activity ($\nu^d(\R^d)=\infty$).
 
\section{Numerical experiments}
\label{sec:numerics}

In this section, we provide two numerical examples\footnote{The numerical experiments have been implemented in \texttt{Python} and executed on a HPC cluster of ETH Zurich. The code is available under the following link: \url{https://github.com/psc25/McKeanVlasovMLPJumps}.} to illustrate how the solution of the McKean--Vlasov SDE~\eqref{eq:mckv} can be approximated by the multilevel Picard (MLP) scheme~\eqref{s02}. First, we summarize the MLP scheme~\eqref{s02} as pseudocode in Algorithm~\ref{alg:mlp}.

\begin{algorithm}[p]
	\DontPrintSemicolon
	\SetInd{0.25em}{0.5em}
	\begin{small}
		\KwInput{$K,m \in \N$, $n \in \N_0$, $T > 0$, $\xi_d \in \R^d$, $\mu_d \in C(\R^d \times \R^d,\R^d)$, $\sigma_d \in C(\R^d \times \R^d,\R^{d \times d})$, $\eta_d \in C(\R^d \times \R^d \times \R^d,\R^d)$, L\'evy measure $\nu^d$, and $\delta = 0$ (if $\nu^d(\R^d) < \infty$) or $\delta \in (0,\infty)$ with $\nu(A^d_\delta) > 0$ (if $\nu^d(\R^d) = \infty$).}
		\KwOutput{MLP approximation $X = (X_j)_{j=0,...,K}$ of the solution to \eqref{eq:mckv} at times $(\frac{jT}{K})_{j=0,...,K}$.}
		
		\SetKwProg{Fn}{Function}{}{end}
		
		Define $\Delta t \leftarrow T/K$, $A^d_\delta := \{z\in\R^d:\Vert z\Vert\geq\delta\}$, $\mathcal{M}^{d}_{\delta}=\max \{1,\lceil\nu^d(A^d_\delta )\rceil\}$, and $\mathcal{B}(\R^d) \ni B \mapsto \nu^d_\delta(B) := \frac{\nu^d(B \cap A^d_\delta)}{\nu(A^d_\delta)} \in [0,1]$.
		
		\Fn{\texttt{Simulate\_Path}()}{
			Generate $K$ realizations $\Delta W_k \sim \mathcal{N}_d(0,I_d \Delta t)$, $k = 1,...,K$, of i.i.d.~random normal vectors, independent of other random variables, and set $W := \big( \sum_{k=1}^j \Delta W_k \big)_{j=0,...,K} \in \R^{(K+1) \times d}$.
			
			Generate $K$ realizations $\Delta N_k \sim \mathrm{Poi}(\nu(A^d_\delta) \Delta t)$, $k = 1,\ldots,K$, of i.i.d.~Poisson random variables, independent of other random variables, and set $N := \big( \sum_{k=1}^j \Delta N_k \big)_{j=0,...,K} \in \N_0^{K+1}$.
			
			Generate $N_K$ realizations $Z_r \sim \nu^d_\delta$, $r = 1,\ldots,N_K$, of i.i.d.~jump sizes, independent of other random variables, and set $Z := (Z_r)_{r=1,...,N_K} \in \R^{N_K \times d}$.
			
			\KwRet $W, N, Z$.
		}
		
		\Fn{\texttt{Simulate\_Compensator}()}{
			Generate $K \!\cdot\! \mathcal{M}^d_\delta$ realizations $\rvV_{k,r} \sim \nu^d_\delta$, $(k,r) \in \lbrace 1,\ldots,K \rbrace \times \lbrace 1,...,\mathcal{M}^d_\delta \rbrace$, of i.i.d.~jump sizes, independent of other random variables, and set $\rvV := (\rvV_{k,r})_{k=1,...,K, \, r=1,...,\mathcal{M}^d_\delta} \in \times_{k=1}^K \R^{\mathcal{M}^d_\delta \times d}$.
			
			\KwRet $\rvV$.
		}
		
		\Fn{\texttt{MLP}($\xi_d,m,n,W,N,Z$)}{
			\If{$n \leq 0$}{
				\KwRet $0 \in \R^{(K+1) \times d}$.
			}
			
			\Else{
				$X := (X_j)_{j=0,...,K} \leftarrow (\xi_d)_{j=0,...,K} \in \R^{(K+1) \times d}$
				
				\For{$\ell = 0,...,n-1$}{
					$X^{(1)} := (X^{(1)}_j)_{j=0,...,K} \leftarrow \texttt{MLP}(\xi_d, m, \ell, W, N, Z) \in \R^{(K+1) \times d}$
					
					$X^{(2)} := (X^{(2)}_j)_{j=0,...,K} \leftarrow \texttt{MLP}(\xi_d, m, \ell-1, W, N, Z) \in \R^{(K+1) \times d}$
					
					\For{$i = 1,...,m^{n-\ell}$}{
						$\overline{W}, \overline{N}, \overline{Z} \leftarrow \texttt{Simulate\_Path}()$
						
						$\widehat{\rvV} \leftarrow \texttt{Simulate\_Compensator}()$
						
						$\overline{X}^{(1)} := (\overline{X}^{(1)}_j)_{j=0,...,K} \leftarrow \texttt{MLP}(\xi_d, m, \ell, \overline{W}, \overline{N}, \overline{Z}) \in \R^{(K+1) \times d}$
						
						$\overline{X}^{(2)} := (\overline{X}^{(2)}_j)_{j=0,...,K} \leftarrow \texttt{MLP}(\xi_d, m, \ell-1, \overline{W}, \overline{N}, \overline{Z}) \in \R^{(K+1) \times d}$
						
						For $\mathbf{B} \in \lbrace \mathbf{W}, \mathbf{N}, \mathbf{C} \rbrace$ initialize $I^\mathbf{B} := (I^\mathbf{B}_j)_{j=0,...,K} \leftarrow 0 \in \R^{(K+1) \times d}$.
						
						\For{$j = 1,...,K$}{
							
							$I^\mathbf{W}_j \leftarrow I^\mathbf{W}_{j-1} + \frac{1}{m^{n-\ell}} \big( \sigma_d(X^{(1)}_{j-1},\overline{X}^{(1)}_{j-1}) - \1_{\N}(\ell) \sigma_d(X^{(2)}_{j-1},\overline{X}^{(2)}_{j-1}) \big) (W_j - W_{j-1}) \in \R^d$
							
							$I^\mathbf{N}_j \leftarrow I^\mathbf{N}_{j-1} + \frac{1}{m^{n-\ell}} \sum_{r=N_{j-1}+1}^{N_j} \big( \eta_d(X^{(1)}_{j-1},\overline{X}^{(1)}_{j-1},Z_r) - \1_{\N}(\ell) \eta_d(X^{(2)}_{j-1},\overline{X}^{(2)}_{j-1},Z_r) \big) \in \R^d$
							
							$I^\mathbf{C}_j \leftarrow I^\mathbf{C}_{j-1} + \frac{\Delta t \, \nu(A^d_\delta)}{m^{n-\ell} \mathcal{M}^d_\delta} \sum_{r=1}^{\mathcal{M}^d_\delta} \big( \eta_d(X^{(1)}_{j-1},\overline{X}^{(1)}_{j-1},\widehat{\rvV}_{j,r}) - \1_{\N}(\ell) \eta_d(X^{(2)}_{j-1},\overline{X}^{(2)}_{j-1},\widehat{\rvV}_{j,r}) \big) \in \R^d$
							
						}
						
						$X \leftarrow X + I^\mathbf{W} + I^\mathbf{N} - I^\mathbf{C} \in \R^{(K+1) \times d}$
						
						Generate a realization $\unif \sim \mathcal{U}(0,1)$ of a uniformly distributed random variable on $[0,1)$.
						
						\For{$j = 1,...,K$}{
							
							$u \leftarrow \lfloor j \unif \rfloor \in \lbrace 0,...,j-1 \rbrace$
							
							$X_j \leftarrow X_j + \frac{j \Delta t}{m^{n-\ell}} \big( \mu_d(X^{(1)}_u,\overline{X}^{(1)}_u) - \1_{\N}(\ell) \mu_d(X^{(2)}_u,\overline{X}^{(2)}_u) \big) \in \R^d$
							
						}
					}
				}
				
				\KwRet $X \in \R^{(K+1) \times d}$ 
				
			}
		}
		
		$W, N, Z \leftarrow \texttt{Simulate\_Path}()$.
		
		$X \leftarrow \texttt{MLP}(\xi_d,m,n,W,N,Z) \in \R^{(K+1) \times d}$
		
		\KwRet $X \in \R^{(K+1) \times d}$.
		
	\end{small}
	\caption{MLP algorithm for McKean--Vlasov SDEs with jumps}
	\label{alg:mlp}
\end{algorithm}

\subsection{Finite jump-activity: Mean-field Ornstein-Uhlenbeck process with jumps}
\label{subsec:ou}

In the first example, we consider an Ornstein-Uhlenbeck process with jumps and a mean-field term in the drift and diffusion coefficient (see \cite[Section~5.3]{AO2023} and \cite[Section~2.1]{Neufeld2025}). More precisely, we assume that the $(\F_t)_{t \in [0,T]}$-adapted process $X^{d,\theta}_\infty = (X^{d,\theta}_\infty(t))_{t \in [0,T]}$ for every $t \in [0,T]$ is given by
\begin{equation}
	\label{eq:ou}
	\begin{aligned}
		X^{d,\theta}_\infty(t) & = \xi_d + \int_0^t \left( a_0 + A_1 X^{d,\theta}_\infty(s-) + A_2 \E\!\left[ X^{d,\theta}_\infty(s-) \right] \right) ds \\
		& \quad\quad + \sum_{k=1}^d \int_0^t \left( b_k + B_k \E\!\left[ X^{d,\theta}_\infty(s-) \right] \right) dW^{d,\theta,k}(s) \\
		& \quad\quad + \sum_{k=1}^d \int_0^t \int_{\R^d} C_k X^{d,\theta}_\infty(s-) z_k \, \tilde{N}^{d,\theta}(dz,ds), 
	\end{aligned}
\end{equation}
where $\xi_d \in \R^d$ is the initial value, where $\nu^d(dz) = \frac{\lambda}{(2\pi)^{d/2}} \exp\left( -0.5 \, \Vert z \Vert^2 \right) dz$ is the finite L\'evy measure (with $\lambda = 5$), and where the vectors $a_0 := (a_{0,i})_{i=1,...,d}^\top, b_k := (b_{k,i})_{i=1,...,d}^\top \in \R^d$, $k \in \lbrace 1,...,d \rbrace$, and the matrices $A_k := (A_{k,i,j})_{i,j=1,...,d}$, $k \in \lbrace 1,2 \rbrace$, $B_k := (B_{k,i,j})_{i,j=1,...,d}, C_k := (C_{k,i,j})_{i,j=1,...,d} \in \R^{d \times d}$, $k \in \lbrace 1,...,d \rbrace$, respectively, are randomly initialized, i.e.,
\begin{align}
	a_{0,i} & \sim \mathcal{U}\!\left( -\tfrac{1}{2d^{1/2}},\tfrac{1}{2d^{1/2}} \right), & & i \in \lbrace 1,...,d \rbrace, & A_{k,i,j} & \sim \mathcal{U}\!\left( -\tfrac{1}{4d}, \tfrac{1}{4d} \right), & & k \in \lbrace 1,2 \rbrace, \, i,j \in \lbrace 1,...,d \rbrace, \\
	b_{k,i} & \sim \mathcal{U}\!\left( -\tfrac{1}{2d},\tfrac{1}{2d} \right), & & k,i \in \lbrace 1,...,d \rbrace, \,\,\, & B_{k,i,j} & \sim \mathcal{U}\!\left( -\tfrac{1}{4d^{3/2}}, \tfrac{1}{4d^{3/2}} \right), & & k,i,j \in \lbrace 1,...,d \rbrace, \\
	& & & & C_{k,i,j} & \sim \mathcal{U}\!\left( -\tfrac{1}{4d^{3/2}}, \tfrac{1}{4d^{3/2}} \right), & & k,i,j \in \lbrace 1,...,d \rbrace,
\end{align}
for $i,j \in \lbrace 1,...,d \rbrace$, which are all independent of each other, where $\mathcal{U}(a,b)$ denotes the uniform distribution on $[a,b]$. Moreover, by taking the expectation in \eqref{eq:ou} and using Fubini's theorem, we have
\begin{equation}
	\E\!\left[ X^{d,\theta}_\infty(t) \right] = \xi_d + a_0 t + (A_1 + A_2) \int_0^t \E\!\left[ X^{d,\theta}_\infty(s-) \right] ds,
\end{equation}
whose solution is given by
\begin{equation}
	m_d(t) := \E\!\left[ X^{d,\theta}_\infty(t) \right] = e^{(A_1+A_2)t} \xi_d + \int_0^t e^{(A_1+A_2)(t-s)} a_0 \, ds.
\end{equation}
Comparing to Setting~\ref{c01}, this corresponds to the drift coefficient $\mu_d(x_1,x_2) = a_0 + A_1 x_1 + A_2 x_2$, diffusion coefficient $\sigma_d(x_1,x_2) = (b_k+B_k x_2)_{k=1,...,d}$ with $k$-th column $b_k+B_k x_2$, and jump coefficient $\eta_d(x_1,x_2,z) := \sum_{k=1}^d C_k x_1 z_k$. 

We now verify that the assumptions of Setting~\ref{c01} are satisfied. Indeed, since $\Vert A_k \Vert \leq \frac{1}{4}$ and $\Vert B_k \Vert \leq \frac{1}{4d^{1/2}}$, we observe for every $x_1,x_2,y_1,y_2 \in \R^d$ that
\begin{equation}
	\label{eq:ou:cond1}
	\begin{aligned}
		& \max\lbrace \Vert \mu_d(x_1,x_2) - \mu_d(y_1,y_2) \Vert, \Vert \sigma_d(x_1,x_2) - \sigma_d(y_1,y_2) \Vert \rbrace \\
		& \leq \max\lbrace \Vert A_1 (x_1 - y_1) + A_2 (x_2 - y_2) \Vert, \Vert (B_k (x_2 - y_2))_{k=1,...,d} \Vert \rbrace \\
		& \leq \max\left\lbrace \Vert A_1 \Vert, \Vert A_2 \Vert, \left( \Vert B_1 \Vert^2 + ... + \Vert B_d \Vert^2 \right)^\frac{1}{2} \right\rbrace \left( \Vert x_1 - y_1 \Vert + \Vert x_2 - y_2 \Vert \right) \\
		& \leq \frac{1}{4} \left( \Vert x_1 - y_1 \Vert + \Vert x_2 - y_2 \Vert \right).
	\end{aligned}
\end{equation}
In addition, by using that $\big\Vert \sum_{k=1}^d C_k v z_k \big\Vert^2 = \sum_{i=1}^d \big( \sum_{k=1}^d (C_k v)_i z_k \big)^2 = \sum_{k,l=1}^d v^\top C_k^\top C_l v z_k z_l$, that $\int_{\R^d} z_k z_l \, \nu^d(dz) = \lambda \1_{\lbrace k \rbrace}(l)$, and that $\Vert C_k \Vert \leq \frac{1}{4d^{1/2}}$, it holds for every $x_1,x_2,y_1,y_2 \in \R^d$ that
\begin{equation}
	\label{eq:ou:cond2}
	\begin{aligned}
		\int_{\R^d} \Vert \eta_d(x_1,x_2,z) - \eta_d(y_1,y_2,z) \Vert^2 \, \nu^d(dz) & = \int_{\R^d} \left\Vert \sum_{k=1}^d C_k (x_1-y_1) z_k \right\Vert^2 \nu^d(dz) \\
		& = \sum_{k,l=1}^d (x_1-y_1)^\top C_k^\top C_l (x_1-y_1) \int_{\R^d} z_k z_l \, \nu^d(dz) \\
		& \leq \lambda \sum_{k=1}^d \Vert C_k \Vert^2 \Vert x_1-y_1 \Vert^2 \leq \frac{\lambda}{16} \Vert x_1 - y_1 \Vert^2.
	\end{aligned}
\end{equation}
Furthermore, by assuming that $\xi_d := \xi_1 \, (1,...,1)^\top \in \R^d$ and using that $\Vert a_0 \Vert \leq \frac{1}{2}$ as well as $\Vert (b_k)_{k=1,...,d} \Vert \leq \frac{1}{2}$, we conclude that
\begin{equation}
	\label{eq:ou:cond3}
	\begin{aligned}
		\Vert \xi_d \Vert + \Vert \mu_d(0,0) \Vert + \Vert \sigma_d(0,0) \Vert + \left( \int_{\R^d} \Vert \eta_d(0,0,z) \Vert^2 \, \nu^d(dz) \right)^\frac{1}{2} & \leq \Vert \xi_d \Vert + \Vert a_0 \Vert + \Vert (b_k)_{k=1,...,d} \Vert \\
		& \leq \vert \xi_1 \vert d^\frac{1}{2} + 1 \leq \left( \vert \xi_1 \vert + 1 \right) d^\frac{1}{2}.
	\end{aligned}
\end{equation}
Hence, \eqref{eq:ou:cond1}, \eqref{eq:ou:cond2}, and \eqref{eq:ou:cond3} show that \eqref{e11}, \eqref{e11b}, and \eqref{e12} are satisfied with $c := \max(1,\frac{\lambda}{16},\vert\xi_1\vert+1)$. In particular, Theorem~\ref{p01} and Theorem~\ref{p02} hold for the mean-field Ornstein-Uhlenbeck model~\eqref{eq:ou}.

For different dimensions $d \in \lbrace 10, 25, 50, 100, 200 \rbrace$ and levels $n = m \in \lbrace 1,...,5 \rbrace$ with $K = m^n$, we run Algorithm~\ref{alg:mlp} independently $10$ times with terminal time $T = 1$ and initial value $\xi_d = (20,...,20)^\top$. The results are reported in Table~\ref{tab:ou} with ($d^{-1/2}$-adjusted) $L^2$-error given by
\begin{equation}
	\label{eq:L2_err}
	\left( \frac{1}{10 d K} \sum_{r=1}^{10} \sum_{k=1}^K \left\Vert X^{d,\theta}_\infty(\omega_r)(t_k) - X^{d,\theta}_{n,m,K}(\omega_r)(t_k) \right\Vert^2 \right)^\frac{1}{2},
\end{equation}
where the true solution $X^{d,\theta}_\infty = (X^{d,\theta}_\infty(t))_{t \in [0,T]}^\top$ is approximated by an Euler-Maruyama scheme of \eqref{eq:ou} over a finer equidistant time grid $(j T/K_1)_{j=0,...,K_1}$ with $K_1 := K \lfloor \frac{500}{K} \rfloor + 1$ points. The realizations of the Brownian increments $\big(\Delta W^{d,\theta}(j T/K_1)\big)_{j=0,...,K_1}$ and the Poisson increments $(\Delta N^{d,\theta}(j T/K_1))_{j=0,...,K_1}$ on this finer Euler-Maruyama grid are then summed over each time subinterval of the coarser MLP grid to obtain the Brownian increments $\big(\Delta W^{d,\theta}(j T/K)\big)_{j=1,...,K}$ and the Poisson increments $(\Delta N^{d,\theta}(j T/K))_{j=1,...,K}$, which are then used for the MLP algorithm.

\begin{table}[ht!]
	\begin{small}
		\begin{tabular}{rl|R{2.3cm}R{2.3cm}R{2.3cm}R{2.3cm}R{2.3cm}|}
			& & \multicolumn{5}{c|}{Level $n$, samples per level $m$, and time points $K = m^n$} \\
			$d$ & & $n = m = 1$ & $n = m = 2$ & $n = m = 3$ & $n = m = 4$ & $n = m = 5$ \\
			\hline
			10 & $L^2$-Error  & 1.6259 & 0.5160 & 0.0925 & 0.0147 & 0.0049 \\ 
 & Time  & $0.0009$  & $0.0107$  & $0.5596$  & $99.2031$  & $30674.6427$  \\ 
 & Cost  & $1.71 \cdot 10^{4}$  & $7.46 \cdot 10^{5}$  & $7.22 \cdot 10^{7}$  & $1.31 \cdot 10^{10}$  & $3.92 \cdot 10^{12}$  \\ 
\hline 
25 & $L^2$-Error  & 1.3444 & 0.3210 & 0.0609 & 0.0091 & 0.0031 \\ 
 & Time  & $0.0008$  & $0.0090$  & $0.7281$  & $127.3445$  & $39515.7786$  \\ 
 & Cost  & $2.19 \cdot 10^{5}$  & $1.00 \cdot 10^{7}$  & $9.85 \cdot 10^{8}$  & $1.79 \cdot 10^{11}$  & $5.36 \cdot 10^{13}$  \\ 
\hline 
50 & $L^2$-Error  & 0.8265 & 0.2290 & 0.0395 & 0.0056 & 0.0027 \\ 
 & Time  & $0.0040$  & $0.0187$  & $1.0943$  & $172.2256$  & $52763.0796$  \\ 
 & Cost  & $1.62 \cdot 10^{6}$  & $7.60 \cdot 10^{7}$  & $7.52 \cdot 10^{9}$  & $1.37 \cdot 10^{12}$  & $4.10 \cdot 10^{14}$  \\ 
\hline 
100 & $L^2$-Error  & 0.5533 & 0.1661 & 0.0282 & 0.0038 & 0.0026 \\ 
 & Time  & $0.0023$  & $0.0779$  & $6.4294$  & $948.9596$  & $221126.8970$  \\ 
 & Cost  & $1.25 \cdot 10^{7}$  & $5.92 \cdot 10^{8}$  & $5.87 \cdot 10^{10}$  & $1.07 \cdot 10^{13}$  & $3.20 \cdot 10^{15}$  \\ 
\hline 
200 & $L^2$-Error  & 0.4160 & 0.1164 & 0.0208 & 0.0026 & 0.0025 \\ 
 & Time  & $0.0054$  & $0.1465$  & $10.4943$  & $2216.1338$  & $913409.7277$  \\ 
 & Cost  & $9.79 \cdot 10^{7}$  & $4.67 \cdot 10^{9}$  & $4.64 \cdot 10^{11}$  & $8.47 \cdot 10^{13}$  & $2.53 \cdot 10^{16}$  \\ 
\hline 

		\end{tabular}
		\caption{MLP solution of the mean-field Ornstein-Uhlenbeck model \eqref{eq:ou} for different $d \in \lbrace 10, 25, 50, 100, 200 \rbrace$, $n = m \in \lbrace 1,...,5 \rbrace$, and $10$ independent runs of the algorithm. While the $L^2$-error \eqref{eq:L2_err} is displayed in the rows ``$L^2$-Error'', the rows ``Time'' and ``Cost'' report the computational times (in seconds) and the computational costs $\mathsf{C}_{n,m,K}^d$, respectively, which are both averaged over the $10$ runs.}
		\label{tab:ou}
	\end{small}
\end{table}

\subsection{Infinite jump-activity: Mean-field Variance-Gamma model}

In the second example, we consider a Variance-Gamma model with mean-field term in the drift and jump coefficient (see \cite{Madan1990,Madan1998,Buchmann2017} without mean-field term). More precisely, we assume that the $(\F_t)_{t \in [0,T]}$-adapted process $X^{d,\theta}_\infty = (X^{d,\theta}_\infty(t))_{t \in [0,T]}$ for every $t \in [0,T]$ is given by
\begin{equation}
	\label{eq:vg}
	\begin{aligned}
		X^{d,\theta}_\infty(t) & = \xi_d + \int_0^t \left( a_0 + A_1 X^{d,\theta}_\infty(s-) + A_2 \E\!\left[ X^{d,\theta}_\infty(s-) \right] \right) \, ds \\
		& \quad\quad + \sum_{k=1}^d \int_0^t \int_{\R^d} C_k \, \E\!\left[ X^{d,\theta}_\infty(s-) \right] z_k \, \tilde{N}^{d,\theta}(dz,ds), 
	\end{aligned}
\end{equation}
where $\xi_d \in \R^d$, where $\nu^d(dz) = \frac{(2\alpha)^{d/4+1}}{(2\pi)^{d/2} \kappa^{d/4} \Vert z \Vert^{d/2}} K_{d/2}\big( \sqrt{2\alpha/\kappa} \Vert z \Vert \big) \, dz$ is the infinite L\'evy measure (with $\alpha = 1.0$, $\kappa = 0.1$, and $K_{d/2}$ denoting the modified Bessel function of the second kind with parameter $d/2$), and where the vector $a_0 := (a_{0,i})_{i=1,...,d}^\top \in \R^d$ and the matrices $A_k := (A_{k,i,j})_{i,j=1,...,d}$, $k \in \lbrace 1,2 \rbrace$, $C_k := (C_{k,i,j})_{i,j=1,...,d} \in \R^{d \times d}$, $k \in \lbrace 1,...,d \rbrace$ are randomly initialized, i.e.,
\begin{align}
	a_{0,i} & \sim \mathcal{U}\!\left( -\tfrac{1}{d^{1/2}},\tfrac{1}{d^{1/2}} \right), & & i \in \lbrace 1,...,d \rbrace, \quad & A_{k,i,j} & \sim \mathcal{U}\!\left( -\tfrac{1}{d}, \tfrac{1}{d} \right), & & k \in \lbrace 1,2 \rbrace, \, i,j \in \lbrace 1,...,d \rbrace, \\
	& & & & C_{k,i,j} & \sim \mathcal{U}\!\left( -\tfrac{1}{d^{3/2}}, \tfrac{1}{d^{3/2}} \right), & & k,i,j \in \lbrace 1,...,d \rbrace,
\end{align}
which are all independent of each other. Moreover, by taking the expectation in $\eqref{eq:vg}$ and using Fubini's theorem, we have
\begin{equation}
	\E\!\left[ X^{d,\theta}_\infty(t) \right] = \xi_d + a_0 t + (A_1 + A_2) \int_0^t \E\!\left[ X^{d,\theta}_\infty(s-) \right] ds,
\end{equation}
whose solution is given by
\begin{equation}
	m_d(t) := \E\!\left[ X^{d,\theta}_\infty(t) \right] = e^{(A_1+A_2)t} \xi_d + \int_0^t e^{(A_1+A_2)(t-s)} a_0 ds.
\end{equation}
Comparing to Setting~\ref{s01}, this corresponds to the drift coefficient $\mu_d(x_1,x_2) = a_0 + A_1 x_1 + A_2 x_2$, diffusion coefficient $\sigma_d(x_1,x_2) = 0$, and jump coefficient $\eta_d(x_1,x_2,z) := \sum_{k=1}^d C_k x_2 z_k$. 

We now verify that the assumptions of Setting~\ref{c01} are satisfied. Indeed, since $\Vert A_1 \Vert \leq 1$ and $\Vert A_2 \Vert \leq 1$ we observe for every $x_1,x_2,y_1,y_2 \in \R^d$ that
\begin{equation}
	\label{eq:vg:cond1}
	\begin{aligned}
		& \max\lbrace \Vert \mu_d(x_1,x_2) - \mu_d(y_1,y_2) \Vert, \Vert \sigma_d(x_1,x_2) - \sigma_d(y_1,y_2) \Vert \rbrace \leq \Vert A_1 (x_1 - y_1) + A_2 (x_2 - y_2) \Vert \\
		& \leq \max\lbrace \Vert A_1 \Vert, \Vert A_2 \Vert \rbrace \left( \Vert x_1 - y_1 \Vert + \Vert x_2 - y_2 \Vert \right) \leq \left( \Vert x_1 - y_1 \Vert + \Vert x_2 - y_2 \Vert \right).
	\end{aligned}
\end{equation}
In addition, by employing the identity $\big\Vert \sum_{k=1}^d C_k v z_k \big\Vert^2 = \sum_{i=1}^d \big( \sum_{k=1}^d (C_k v)_i z_k \big)^2 = \sum_{k,l=1}^d v^\top C_k^\top C_l v z_k z_l$, that $\int_{\R^d} z_k z_l \, \nu^d(dz) = \frac{\1_{\lbrace k \rbrace}(l)}{d} \int_{\R^d} \Vert z \Vert^2 \, \nu^d(dz)$ with $\int_{\R^d} \Vert z \Vert^2 \nu^d(dz) = \frac{(2\alpha)^{d/4+1}}{2^{d/2-1} \kappa^{d/4} \Gamma(d/2)} \int_0^\infty r^{2-d/2} K_{d/2}\big( \sqrt{2\alpha/\kappa} r \big) r^{d-1} dr = \frac{\kappa}{2^{d/2-1} \Gamma(d/2)} \int_0^\infty s^{1+d/2} K_{d/2}(s) ds = 2\kappa \frac{\Gamma(d/2+1)}{\Gamma(d/2)} = 2\kappa \frac{d}{2} = \kappa d$ (using polar coordinates, substitution $s \mapsto \sqrt{2\alpha/\kappa} r$, and \cite[Equation~10.43.22]{CLMF26}), and that $\Vert C_k \Vert \leq \frac{1}{d^{1/2}}$, it holds for every $x_1,x_2,y_1,y_2 \in \R^d$ that
\begin{equation}
	\label{eq:vg:cond2}
	\begin{aligned}
		\int_{\R^d} \Vert \eta_d(x_1,x_2,z) - \eta_d(y_1,y_2,z) \Vert^2 \nu^d(dz) & = \int_{\R^d} \left\Vert \sum_{k=1}^d C_k (x_2-y_2) z_k \right\Vert^2 \nu^d(dz) \\
		& = \sum_{k,l=1}^d (x_2-y_2)^\top C_k^\top C_l (x_2-y_2) \int_{\R^d} z_k z_l \nu^d(dz) \\
		& \leq \kappa \sum_{k=1}^d \Vert C_k \Vert^2 \Vert x_2 - y_2 \Vert^2 \leq \kappa \Vert x_2 - y_2 \Vert^2.
	\end{aligned}
\end{equation}
Furthermore, by assuming that $\xi_d := \xi_1 \, (1,...,1)^\top \in \R^d$ and using that $\Vert a_0 \Vert \leq 1$, we have
\begin{equation}
	\label{eq:vg:cond3}
	\begin{aligned}
		\Vert \xi_d \Vert + \Vert \mu_d(0,0) \Vert + \Vert \sigma_d(0,0) \Vert + \left( \int_{\R^d} \Vert \eta_d(0,0,z) \Vert^2 \nu^d(dz) \right)^\frac{1}{2} & \leq \Vert \xi_d \Vert + \Vert a_0 \Vert \leq \vert \xi_1 \vert d^\frac{1}{2} + 1 \\
		& \leq \left( \vert \xi_1 \vert + 1 \right) d^\frac{1}{2}.
	\end{aligned}
\end{equation}
Moreover, by using the identity $\big\Vert \sum_{k=1}^d C_k v z_k \big\Vert^2 = \sum_{k,l=1}^d v^\top C_k^\top C_l v z_k z_l$, that $\int_{\R^d \setminus A^d_\delta} z_k z_l \, \nu^d(dz) = \frac{\1_{\lbrace k \rbrace}(l)}{d} \int_{\R^d \setminus A^d_\delta} \Vert z \Vert^2 \nu^d(dz)$, that $\int_{\R^d\setminus A^d_\delta} \Vert z \Vert^2 \nu^d(dz) = \frac{(2\alpha)^{d/4+1}}{2^{d/2-1} \kappa^{d/4} \Gamma(d/2)} \int_0^\delta \! r^{2-d/2} K_{d/2}\big( \sqrt{2\alpha/\kappa} r \big) r^{d-1} dr \leq \delta \frac{\sqrt{2\alpha\kappa}}{2^{d/2-1} \Gamma(d/2)} \int_0^\infty s^{d/2} K_{d/2}(s) ds = \delta \sqrt{2\pi\alpha\kappa} \frac{\Gamma((d+1)/2)}{\Gamma(d/2)} \leq c_{\alpha,\kappa} \delta d^{1/2}$ for some constant $c_{\alpha,\kappa} > 0$ (only depending on $\alpha,\kappa > 0$), and that $\Vert C_k \Vert \leq \frac{1}{d^{1/2}}$, it follows for every $\delta \in (0,\infty)$ and $x,y \in \R^d$ that
\begin{equation}
	\label{eq:vg:cond4}
	\begin{aligned}
		\int_{\R^d\setminus A^d_\delta} \Vert \eta_d(x,y,z) \Vert^2\,\nu^d(dz) & = \int_{\R^d\setminus A^d_\delta} \left\Vert \sum_{k=1}^d C_k y z_k \right\Vert^2 \nu^d(dz) = \sum_{k,l=1}^d y^\top C_k^\top C_l y \int_{\R^d \setminus A^d_\delta} z_k z_l \nu^d(dz) \\
		& \leq \frac{1}{d} c_{\alpha,\kappa} \delta d^\frac{1}{2} \sum_{k=1}^d \Vert C_k \Vert^2 \Vert y \Vert^2 \leq c_{\alpha,\kappa} \delta \Vert y \Vert^2.
	\end{aligned}
\end{equation}
In addition, by using that $1 \wedge \Vert z \Vert^2 \leq \Vert z \Vert^2$ and $\int_{\R^d} \Vert z \Vert^2 \, \nu^d(dz) = \kappa d$ (as argued above), we obtain
\begin{equation}
	\label{eq:vg:cond5}
	\int_{\R^d} (1 \wedge \Vert z \Vert^2) \, \nu^d(dz) \leq \int_{\R^d} \Vert z \Vert^2 \, \nu^d(dz) = \kappa d.
\end{equation}
Hence, \eqref{eq:vg:cond1}, \eqref{eq:vg:cond2}, \eqref{eq:vg:cond3}, \eqref{eq:vg:cond4}, and \eqref{eq:vg:cond5} show that \eqref{e11}, \eqref{e11b}, \eqref{e12}, \eqref{e13}, and \eqref{s13b} are satisfied with $c := \max(1,\kappa,\vert \xi_1 \vert + 1,c_{\alpha,\kappa})$ and $q := 1$. In particular, Theorem~\ref{p01b} and Theorem~\ref{p02b} hold for the mean-field Variance-Gamma model~\eqref{eq:vg}.

For different dimensions $d \in \lbrace 10, 25, 50, 100, 200 \rbrace$ and levels $n = m \in \lbrace 1,...,5 \rbrace$ with $K = m^n$, we run Algorithm~\ref{alg:mlp} independently $10$ times with terminal time $T = 1$, initial value $\xi_d = (10,...,10)^\top$, and $\delta = 0.02$. The results are reported in Table~\ref{tab:vg}, with the true solution $X^{d,\theta}_\infty = (X^{d,\theta}_\infty(t))_{t \in [0,T]}^\top$ approximated by an Euler-Maruyama scheme of \eqref{eq:vg} over a finer equidistant time grid $(j T/K_1)_{j=0,...,K_1}$ with $K_1 := K \lfloor \frac{500}{K} \rfloor + 1$ points and a lower small-jump truncation threshold $\delta_1 := 0.001$, where the Poisson increments as well as the jump sizes are aggregated as in Section~\ref{subsec:ou} to use it for the MLP algorithm.

\begin{table}[ht!]
	\begin{small}
		\begin{tabular}{rl|R{2.3cm}R{2.3cm}R{2.3cm}R{2.3cm}R{2.3cm}|}
			& & \multicolumn{5}{c|}{Level $n$, samples per level $m$, and time points $K = m^n$} \\
			$d$ & & $n = m = 1$ & $n = m = 2$ & $n = m = 3$ & $n = m = 4$ & $n = m = 5$ \\
			\hline
			10 & $L^2$-Error  & 1.9218 & 0.2534 & 0.0757 & 0.0250 & 0.0118 \\ 
 & Time  & $0.0011$  & $0.0111$  & $0.5800$  & $90.0110$  & $31499.5689$  \\ 
 & Cost  & $1.52 \cdot 10^{4}$  & $6.23 \cdot 10^{5}$  & $5.91 \cdot 10^{7}$  & $1.07 \cdot 10^{10}$  & $3.20 \cdot 10^{12}$  \\ 
\hline 
25 & $L^2$-Error  & 1.4106 & 0.1219 & 0.0309 & 0.0123 & 0.0047 \\ 
 & Time  & $0.0018$  & $0.0112$  & $0.7286$  & $119.4330$  & $37138.5412$  \\ 
 & Cost  & $1.73 \cdot 10^{5}$  & $7.59 \cdot 10^{6}$  & $7.36 \cdot 10^{8}$  & $1.34 \cdot 10^{11}$  & $3.99 \cdot 10^{13}$  \\ 
\hline 
50 & $L^2$-Error  & 0.8396 & 0.0704 & 0.0182 & 0.0057 & 0.0024 \\ 
 & Time  & $0.0011$  & $0.0114$  & $0.7622$  & $125.4114$  & $45642.8699$  \\ 
 & Cost  & $1.20 \cdot 10^{6}$  & $5.44 \cdot 10^{7}$  & $5.33 \cdot 10^{9}$  & $9.70 \cdot 10^{11}$  & $2.90 \cdot 10^{14}$  \\ 
\hline 
100 & $L^2$-Error  & 0.5714 & 0.0449 & 0.0101 & 0.0028 & 0.0016 \\ 
 & Time  & $0.0046$  & $0.0233$  & $1.5465$  & $263.2675$  & $77193.3681$  \\ 
 & Cost  & $8.84 \cdot 10^{6}$  & $4.11 \cdot 10^{8}$  & $4.06 \cdot 10^{10}$  & $7.40 \cdot 10^{12}$  & $2.21 \cdot 10^{15}$  \\ 
\hline 
200 & $L^2$-Error  & 0.3944 & 0.0302 & 0.0062 & 0.0019 & 0.0011 \\ 
 & Time  & $0.0040$  & $0.0623$  & $3.6263$  & $622.8335$  & $222600.4870$  \\ 
 & Cost  & $6.78 \cdot 10^{7}$  & $3.19 \cdot 10^{9}$  & $3.16 \cdot 10^{11}$  & $5.77 \cdot 10^{13}$  & $1.72 \cdot 10^{16}$  \\ 
\hline 

		\end{tabular}
		\caption{MLP solution of the mean-field Variance-Gamma model~\eqref{eq:vg} for different $d \in \lbrace 10, 25, 50, 100, 200 \rbrace$, $n = m \in \lbrace 1,...,5 \rbrace$, and $10$ independent runs of the algorithm. While the $L^2$-error \eqref{eq:L2_err} is displayed in the rows ``$L^2$-Error'', the rows ``Time'' and ``Cost'' report the computational times (in seconds) and the computational costs $\mathsf{C}_{n,m,K,\delta}^d$, respectively, which are both averaged over the $10$ runs.}
		\label{tab:vg}
	\end{small}
\end{table}

\section{Proof of the main results}
\label{sec:proofs}In the whole section we assume Setting \ref{s01c}.
Also recall that we consider two cases:
\begin{enumerate}[(i)]
\item The case of finite jump-activity, i.e., for all $d\in \N$ we have that 
$\nu^d(\R^d)\in (0,\infty)$.
\item The case of infinite jump-activity, i.e., for all $d\in \N$ we have that 
$\nu^d(\R^d)=\infty$. In this case we assume Setting~\ref{s01}.
\end{enumerate}
In the whole section let $\Delta$ be the set $\{0\}$ in the case of finite jump-activity and be the interval $(0,\infty)$ in the case of infinite jump-activity. Furthermore, for every $d,m,K\in \N$,
$n\in \N_0$, $\delta\in \Delta$ let $\mathcal{G}^d_{n,m,K,\delta}$ be the sigma-algebra generated by 
$W^{d,0}$,
$N^{d,0}$, $(X_{\ell,m,K,\delta}^{d,0})_{\ell\in [0,n]\cap\Z}$.

An important ingredient of the proof is the truncated Euler-Maruyama approximations $X^d_{K,\delta}$, see Lemma~\ref{a14}. We separate the error into the MLP error and 
the time-discretization error in the case of finite jump-activity (see the triangle inequality in~\eqref{k19}) and into the MLP error, the truncation error, and the time-discretization error  in the case of infinite jump-activity (see the triangle inequality in~\eqref{k18}). The MLP error is separated into the statistical error and the bias (see the triangle inequality in~\eqref{k17}).

%%%%Dis. error
%\subsection{The discretization error}
%We will next establish the discretization error of the approximations. First of all we need time regularity of the exact solution.

\subsection{Some simple estimates}
We first state and prove some simple estimates which are often used in the proof of the main theorems.

\begin{lemma}[Some simple estimates]\label{a05f}
Assume Setting~\ref{c01} and let $d\in \N$. Then the following items are true.
\begin{enumerate}[(i)]
\item\label{a05a} For all $x\in \R^d$, $X\in L^2(\P;\R^d)$ we have that $\mu_d(x,X)\in L^2(\P;\R^d)$ and 
$\sigma_d(x,X)\in L^2(\P;\R^{d\times d})$.
\item \label{a05b}
For  all $x\in \R^d$,
$X\in L^2(\P;\R^d)$
 we have that
$\eta_d (x,X,\cdot)\in L^2(\P\otimes \nu^d ;\R^d)$.
\item \label{a05c}
For all  $X\in L^2(\P;\R^d)$ we have that
\begin{align}
\max\{
\lVert
\E [\mu_d(x,X)]|_{x=X}\rVert_{L^2(\P;\R^d)} ,\lVert\E [\sigma_d(x,X)]|_{x=X} 
\rVert_{L^2(\P;\R^{d\times d})}
\}
\leq 
cd^c+ 2c
 \lVert X\rVert_{L^2(\P;\R^d)}.
\end{align}
\item \label{a05e}For all $X\in L^2(\P;\R^d)$ we have
\begin{align}
\left(
\int_{\R^d}\left\lVert
\E \!\left[\eta_d(x,X,z)\right]\!|_{x=X}
\right\rVert_{L^2(\P;\R^d)}^2\nu^d(dz)\right)^\frac{1}{2}
\leq cd^c+ \sqrt{2c} \lVert X\rVert_{L^2(\P;\R^d)}.
\end{align}
\item \label{f01} For all
 $X,Y\in L^2(\P;\R^d)$ we have that
\begin{align}
&
\max\Bigl\{\left\lVert
\E \!\left[\mu_d(x,X)\right]\!|_{x=X}
-
\E \!\left[\mu_d(y,Y)\right]\!|_{y=Y}
 \right\rVert_{L^2(\P;\R^d)},\nonumber\\
&\qquad\qquad
\left\lVert
\E \!\left[\sigma_d(x,X)\right]\!|_{x=X}
-
\E \!\left[\sigma_d(y,Y)\right]\!|_{y=Y}
 \right\rVert_{L^2(\P;\R^{d\times d})}\Bigr
\}\nonumber\\
&\leq 2c \lVert X -Y\rVert_{L^2(\P;\R^d)}.
\end{align}

\item \label{f02} For all $X,Y\in L^2(\P;\R^d)$ we have that 
\begin{align}
&
\int_{\R^d}\left\lVert
\E \!\left[\eta_d(x,X,z)\right]\!|_{x=X}
-
\E \!\left[\eta_d(x,Y,z)\right]\!|_{x=Y}
\right\rVert_{L^2(\P;\R^d)}^2
\nu^d(dz)
\leq 2c\left\lVert
X-Y
\right\rVert_{L^2(\P;\R^d)}^2.\label{a12d}
\end{align}
\end{enumerate}
\end{lemma}
\begin{proof}[Proof of Lemma~\ref{a05f}]
First,
\eqref{e11}, the triangle inequality, and \eqref{e12} show
for all 
$d\in \N$,
$x\in \R^d$, $X\in L^2(\P;\R^d)$ that
\begin{align} 
&
\max \!\left\{
\lVert \mu_d (x,X) \rVert_{L^2(\P;\R^d)}
,
\lVert \sigma_d (x,X) \rVert_{L^2(\P;\R^{d\times d})}
\right\}\nonumber\\&\leq \max \{\lVert \mu_d(0,0)\rVert, \lVert\sigma_d(0,0)\rVert\}+ c\lVert x\rVert+ c \lVert X\rVert_{L^2(\P;\R^d)}<\infty.
\label{a02b}
\end{align} This shows \eqref{a05a}.

Next, \eqref{e11b} and the triangle inequality show for all
 $x_1,x_2\in \R^d$
that
\begin{align}
&\left(\int_{\R^d} \left\lVert\eta_d(x_1,x_2,z)\right\rVert^2\nu^d(dz)
\right)^\frac{1}{2}\nonumber\\
&
\leq \left(
\int_{\R^d} \left\lVert\eta_d(x_1,x_2,z)-\eta_d(0,0,z)\right\rVert^2\nu^d(dz)
\right)^\frac{1}{2}
+\left(
\int_{\R^d} \left\lVert\eta_d(0,0,z)\right\rVert^2\nu^d(dz)
\right)^\frac{1}{2}\nonumber\\
&
\leq \left(c \lVert x_1\rVert^2 + c \lVert x_2\rVert^2\right)^\frac{1}{2}
+\left(
\int_{\R^d} \left\lVert\eta_d(0,0,z)\right\rVert^2\nu^d(dz)
\right)^\frac{1}{2}.\label{s02a}
\end{align}
This, Tonelli's theorem,  the triangle inequality, and \eqref{e12} imply for all 
$x\in \R^d$, $X\in L^2(\P;\R^d)$ that
\begin{align}
\left(
\int_{\R^d}\E \!\left[\lVert\eta_d(x,X,z)\rVert^2
\right] \nu^d (dz)\right)^\frac{1}{2}&\leq
\left(\E \!\left[
\int_{\R^d}\lVert\eta_d(x,X,z)\rVert^2\, \nu^d (dz)\right]\right)^\frac{1}{2}\nonumber\\
&
\leq 
\left(\E \!\left[
c\lVert x\rVert^2+c\lVert X\rVert^2
\right]\right)^\frac{1}{2}+\left(
\int_{\R^d} \left\lVert\eta_d(0,0,z)\right\rVert^2\nu^d(dz)
\right)^\frac{1}{2}<\infty.
\end{align}
This shows \eqref{a05b}.

Next, Jensen's inequality and \eqref{a02b} show for all  $X\in L^2(\P;\R^d)$ that
\begin{align}
&\max\{
\lVert
\E [\mu_d(x,X)]|_{x=X}\rVert_{L^2(\P;\R^d)} ,\lVert\E [\sigma_d(x,X)]|_{x=X} 
\rVert_{L^2(\P;\R^{d\times d})}
\}\nonumber\\
&\leq
\max\biggl\{
\left\lVert
\lVert\mu_d(x,X)\rVert_{L^2(\P;\R^d)}|_{x=X}\right\rVert_{L^2(\P;\R^d)}, \left\lVert\lVert\sigma_d(x,X)\rVert_{{L^2(\P;\R^{d\times d})}}|_{x=X} \right
\rVert_{L^2(\P;\R^{d\times d})}
\biggr\}\nonumber\\
&\leq 
 \max \{\lVert \mu_d(0,0)\rVert, \lVert\sigma_d(0,0)\rVert\}+
\left\lVert
\left[c\lVert x\rVert+ c\lVert X\rVert_{L^2(\P;\R^d)}\right]\!|_{x=X}
\right\rVert_{L^2(\P;\R^d)}\nonumber\\
&\leq 
 \max \{\lVert \mu_d(0,0)\rVert, \lVert\sigma_d(0,0)\rVert\}+ 2c
 \lVert X\rVert_{L^2(\P;\R^d)}.
\end{align} This and \eqref{e12} show \eqref{a05c}.

Next,
Jensen's inequality, Tonelli's theorem, the triangle inequality, and \eqref{s02a} demonstrate for all
$X\in L^2(\P;\R^d)$ that
\begin{align}
&\left(
\int_{\R^d}\left\lVert
\E \!\left[\eta_d(x,X,z)\right]\!|_{x=X}
\right\rVert_{L^2(\P;\R^d)}^2\nu^d(dz)\right)^\frac{1}{2}\nonumber\\
&\leq \left(
\int_{\R^d}\left\lVert
\left\lVert\eta_d(x,X,z)\right\rVert_{L^2(\P;\R^d)}\!|_{x=X}
\right\rVert_{L^2(\P;\R^d)}^2\nu^d(dz)\right)^\frac{1}{2}\nonumber\\
&=\left(
\int_{\R^d} \int_\Omega\int_\Omega \lVert \eta_d(x,y,z)\rVert^2
\P_{X}(dy)\P_{X}(dx)\,
\nu^d(dz)\right)^\frac{1}{2}\nonumber\\
&=\left(
\int_\Omega\int_\Omega
\int_{\R^d}  \lVert \eta_d(x,y,z)\rVert^2
\nu^d(dz)
\P_{X}(dy)\P_{X}(dx)\,
\right)^\frac{1}{2}\nonumber
\\
&\leq \left(
\int_\Omega\int_\Omega\left(
c \lVert x\rVert^2 + c \lVert y\rVert^2
\right)
\P_{X}(dy)\P_{X}(dx)
\right)^\frac{1}{2}
+\left(
\int_{\R^d} \left\lVert\eta_d(0,0,z)\right\rVert^2\nu^d(dz)
\right)^\frac{1}{2}\nonumber\\
&=\sqrt{2c} \left\lVert
X
\right\rVert_{L^2(\P;\R^d)}
+\left(
\int_{\R^d} \left\lVert\eta_d(0,0,z)\right\rVert^2\nu^d(dz)
\right)^\frac{1}{2}.\label{a03d}
\end{align}
This and \eqref{e12} show \eqref{a05e}.

Next, \eqref{e11} and Jensen's inequality imply for all
 $X,Y\in L^2(\P;\R^d)$ that
\begin{align} 
&
\left\lVert
\E \!\left[\mu_d(x,X)\right]\!|_{x=X}
-
\E \!\left[\mu_d(y,Y)\right]\!|_{y=Y}
 \right\rVert_{L^2(\P;\R^d)}\nonumber\\
&
=\left\lVert
\E \left[ \mu_d(x,X)-\mu_d(y,Y)\right]|_{(x,y)=(X,Y)}\right\rVert_{L^2(\P;\R^d)}
\nonumber\\
&\leq
\left\lVert
\left\lVert \mu_d(x,X)-\mu_d(y,Y)
\right\rVert_{L^2(\P;\R^d)}|_{(x,y)=(X,Y)}
\right\rVert_{L^2(\P;\R^d)}\nonumber\\
&\leq
\left\lVert
\left[c\lVert x-y\rVert+ c \left\lVert
X-Y \right\rVert_{L^2(\P;\R^d)}
\right]\!|_{(x,y)=(X,Y)}
\right\rVert_{L^2(\P;\R^d)}\nonumber\\
&\leq 2c \lVert X -Y\rVert_{L^2(\P;\R^d)}\label{a12}
\end{align}
and similarly
\begin{align} \begin{split} \label{a13}
&
\left\lVert
\E \!\left[\sigma_d(x,X)\right]\!|_{x=X}
-
\E \!\left[\sigma_d(y,Y)\right]\!|_{y=Y}
 \right\rVert_{L^2(\P;\R^{d\times d})}\leq
 2 c \lVert X -Y\rVert_{L^2(\P;\R^d)}.
\end{split}
\end{align}
This shows \eqref{f01}.

Next, Jensen's inequality and \eqref{e11b} show for all $X,Y\in L^2(\P;\R^d)$ that 
\begin{align}
&
\int_{\R^d}\left\lVert
\E \!\left[\eta_d(x,X,z)\right]\!|_{x=X}
-
\E \!\left[\eta_d(x,Y,z)\right]\!|_{x=Y}
\right\rVert_{L^2(\P;\R^d)}^2
\nu^d(dz)\nonumber\\
&=
\int_{\R^d}\E \!\left[\left\lVert
\E \!\left[\eta_d(x,X,z)\right]\!|_{x=X}
-
\E \!\left[\eta_d(x,Y,z)\right]\!|_{x=Y}
\right\rVert^2\right]
\nu^d(dz)\nonumber\\
&=
\int_{\R^d}\E \!\left[\left\lVert
\E \!\left[
\eta_d(x,X,z)
-
\eta_d(y,Y,z)
\right]|_{(x,y)=(X,Y)}
\right\rVert^2\right]
\nu^d(dz)\nonumber\\
&=
\int_{\R^d}\E \!\left[\left\lVert
\E \!\left[
\eta_d(x,X,z)
-
\eta_d(y,Y,z)
\right]
\right\rVert^2|_{(x,y)=(X,Y)}\,
\right]\nu^d(dz)\nonumber\\
&\leq 
\int_{\R^d}
\E \!\left[
\E \!\left[\left\lVert
\eta_d(x,X,z)
-
\eta_d(y,Y,z)
\right\rVert^2\right]\!|_{(x,y)=(X,Y)}\,
\right]\nu^d(dz)\nonumber\\
&= \int_{\R^d}\int_\Omega\int_\Omega
\left\lVert
\eta_d (x,\tilde{x},z)
-\eta_d(y,\tilde{y},z)
\right\rVert^2
\P_{(X,Y)} (d(\tilde{x},\tilde{y}))\,
\P_{(X,Y)} (d(x,y))
\nu^d(dz)\nonumber\\
&= \int_\Omega\int_\Omega\int_{\R^d}
\left\lVert
\eta_d (x,\tilde{x},z)
-\eta_d(y,\tilde{y},z)
\right\rVert^2
\nu^d(dz)
\P_{(X,Y)} (d(\tilde{x},\tilde{y}))
\P_{(X,Y)} (d(x,y))
\nonumber\\
&\leq  \int_\Omega \int_\Omega\left(
c \lVert x-y\rVert^2+
c \lVert \tilde{x}-\tilde{y}\rVert^2
\right)
\P_{(X,Y)} (d(\tilde{x},\tilde{y}))\,
\P_{(X,Y)} (d(x,y))
\nonumber\\
&
=2c\left\lVert
X-Y
\right\rVert_{L^2(\P;\R^d)}^2.\label{a12f}
\end{align}
This shows \eqref{f02}.
The proof of Lemma~\ref{a05f} is thus completed.
\end{proof}
\subsection{Polynomial moment bound for the exact solution}
\begin{lemma}[Polynomial moment bound for $X^d_\infty$]\label{b01b}
For every $d\in \N$ let $X^d_\infty$ be the process in Theorem~\ref{b02a}.
Then there exists $\kappa\in (0,\infty)$ such that for all
$d\in \N$,  $t\in [0,T]$ we have that
$\lVert X^d_\infty(t)\rVert_{L^2(\P;\R^d)}\leq \kappa d^\kappa$.
\end{lemma}
\begin{proof}[Proof of Lemma~\ref{b01b}]The triangle inequality, \eqref{eq:thm_mckv}, Jensen's inequality, It\^o's isometry, Lemma~\ref{a05f}, and \eqref{e12} show that there exist  $\kappa_1,\kappa_2\in (0,\infty)$ such that for all $d\in \N$, $t\in [0,T]$ that
\begin{align}
\left\lVert X^{d}_{\infty}(t)\right\rVert_{L^2(\P;\R^d)}& \leq \lVert\xi_d\rVert+
{\kappa_1}
\left(\int_{0}^{t}
\left\lVert
\E\! \left[\mu_d (x, X^d_{\infty}(s-) )\right]\!|_{x=X^{d}_{\infty}(s-)}
\right\rVert_{L^2(\P;\R^d)}^2ds
\right)^\frac{1}{2}
\nonumber\\
&\quad 
+\left\lVert\int_{0}^{t}
\E\! \left[\sigma_d (x, X^{d}_{\infty}(s-) )\right]\!|_{x=X^{d}_{\infty}(s-)}
\,dW^{d,0}(s)\right\rVert_{L^2(\P;\R^d)}\nonumber
\\
&\quad +\left\lVert\int_{0}^{t}\int_{\R^d}
\E\! \left[\eta_d (x, X^{d}_{\infty}(s-),z )\right]\!|_{x=X^{d}_{\infty}(s-)}\,\tilde{N}^{d,0}(dz,ds)
\right\rVert_{L^2(\P;\R^d)}\nonumber\\
&= 
\lVert\xi_d\rVert+
{\kappa_1}
\left(\int_{0}^{t}
\left\lVert
\E\! \left[\mu_d (x, X^d_{\infty}(s-) )\right]\!|_{x=X^{d}_{\infty}(s-)}
\right\rVert_{L^2(\P;\R^d)}^2ds
\right)^\frac{1}{2}\nonumber\\
&\quad
+
\left(\int_{0}^{t}
\left\lVert
\E\! \left[\sigma_d (x, X^d_{\infty}(s-) )\right]\!|_{x=X^{d}_{\infty}(s-)}
\right\rVert_{L^2(\P;\R^{d\times d})}^2ds
\right)^\frac{1}{2}\nonumber
\\
&\quad
+
\left(\int_{0}^{t}\int_{\R^d}
\left\lVert
\E\! \left[\eta_d (x, X^d_{\infty}(s-),z )\right]\!|_{x=X^{d}_{\infty}(s-)}
\right\rVert_{L^2(\P;\R^{d})}^2\,\nu^d(dz)\,ds
\right)^\frac{1}{2}\nonumber\\
&\leq \kappa_2 d^{\kappa_2}+\kappa_2\left(
\int_{0}^{t}\left\lVert
X^d_\infty(s-)
\right\rVert_{L^2(\P;\R^d)}^2\,ds
\right)^\frac{1}{2}.
\end{align}This shows that there exists $\kappa\in (0,\infty)$ such that for all $d\in \N$, $t\in [0,T]$ we have that
\begin{align}
\sup_{u\in [0,t]}
\left\lVert X^{d}_{\infty}(u)\right\rVert_{L^2(\P;\R^d)}\leq \kappa d^\kappa+\kappa\left(
\int_{0}^{t}
\sup_{u\in [0,s]}
\left\lVert
X_\infty^d(u)
\right\rVert_{L^2(\P;\R^d)}^2\,ds
\right)^\frac{1}{2}.
\end{align}
This, Gr\"onwall's lemma, and the fact that 
$
\sup_{s\in [0,T]}\lVert X^d_\infty(s)\rVert_{L^2(\P;\R^d)}<\infty
$ complete the proof of Lemma~\ref{b01b}.
\end{proof}

\subsection{Euler--Maruyama approximations and the corresponding polynomial moment bound}
We introduce Euler--Maruyama approximations in Lemma~\ref{a14} below. In Lemma~\ref{b01} we prove that, as in the standard case, the $L^2$-norm of the Euler-Maruyama approximations is bounded by a polynomial of the dimension.
\begin{lemma}[Truncated version of Euler--Maruyama approximations]\label{a14}
For every $K,d\in \N$, $\delta\in [0,\infty)$ there exists a unique $(\F_t)_{t\in [0,T]}$-adapted process $X^{d}_{K,\delta}=(X^d_{K,\delta}(t))_{t\in [0,T]}$ with c\`adl\`ag sample paths such that we have that 
$
\sup_{s\in [0,T]}\lVert X^d_{K,\delta}(s)\rVert_{L^2(\P;\R^d)}<\infty
$
and 
$\P$-a.s.\ for all $t\in [0,T]$ that
\begin{align}
X^{d}_{K,\delta}(t)&= \xi_d+\int_{0}^{t}
\E\! \left[\mu_d (x, X^d_{K,\delta}(\rdown{s-}_K) )\right]\!|_{x=X^{d}_{K,\delta}(\rdown{s-}_K)}\,ds\nonumber\\
&\quad 
+\int_{0}^{t}
\E\! \left[\sigma_d (x, X^{d}_{K,\delta}(\rdown{s-}_K) )\right]\!|_{x=X^{d}_{K,\delta}(\rdown{s-}_K)}\,dW^{d,0}(s)\nonumber
\\
&\quad +\int_{0}^{t}\int_{A^d_\delta}
\E\! \left[\eta_d (x, X^{d}_{K,\delta}(\rdown{s-}_K),z )\right]\!|_{x=X^{d}_{K,\delta}(\rdown{s-}_K)}\,\tilde{N}^{d,0}(dz,ds).\label{b02b}
\end{align}
\end{lemma}
\begin{proof}[Proof of Lemma~\ref{a14}]
 The existence and uniqueness of $X^d_{K,\delta}$, $d,K\in \N$, $\delta\in [0,\infty)$, follows from induction along the time grids and Lemma~\ref{a05f}.  The proof of Lemma~\ref{a14} is thus completed.
\end{proof}

\begin{lemma}[Polynomial moment bound for $X^d_{K,\delta}$]\label{b01}
For every $d,K\in \N$, $\delta\in [0,\infty)$ let $X^d_{K,\delta}$  be the process in Lemma~\ref{a14}. 
Then there exists $\kappa\in (0,\infty)$ such that for all
$d,K\in \N$, $t\in [0,T]$, $\delta\in [0,\infty)$ we have that
$\lVert X^d_{K,\delta}(t)\rVert_{L^2(\P;\R^d)}\leq \kappa d^\kappa$.
\end{lemma}
\begin{proof}[Proof of Lemma~\ref{b01}]The triangle inequality, \eqref{b02b}, Jensen's inequality, It\^o's isometry, Lemma~\ref{a05f}, and \eqref{e12} show that there exist  $\kappa_1,\kappa_2\in (0,\infty)$ such that for all $d,K\in \N$, $\delta\in [0,\infty)$, $t\in [0,T]$ that
\begin{align}
\left\lVert X^{d}_{K,\delta}(t)\right\rVert_{L^2(\P;\R^d)}& \leq \lVert\xi_d\rVert+
{\kappa_1}
\left(\int_{0}^{t}
\left\lVert
\E\! \left[\mu_d (x, X^d_{K,\delta}(\rdown{s-}_{K}))\right]\!|_{x=X^{d}_{{K,\delta}}(\rdown{s-}_K)}
\right\rVert_{L^2(\P;\R^d)}^2ds
\right)^\frac{1}{2}
\nonumber\\
&\quad 
+\left\lVert\int_{0}^{t}
\E\! \left[\sigma_d (x, X^{d}_{K,\delta}(\rdown{s-}_K) )\right]\!|_{x=X^{d}_{K,\delta}(\rdown{s-}_K)}
\,dW^{d,0}(s)\right\rVert_{L^2(\P;\R^d)}\nonumber
\\
&\quad +\left\lVert\int_{0}^{t}\int_{A^d_\delta}
\E\! \left[\eta_d (x, X^{d}_{K,\delta}(\rdown{s-}_K),z )\right]\!|_{x=X^{d}_{K,\delta}(\rdown{s-}_{K})}\,\tilde{N}^{d,0}(dz,ds)
\right\rVert_{L^2(\P;\R^d)}\nonumber\\
&= 
\lVert\xi_d\rVert+
{\kappa_1}
\left(\int_{0}^{t}
\left\lVert
\E\! \left[\mu_d (x, X^d_{K,\delta}(\rdown{s-}_K) )\right]\!|_{x=X^{d}_{K,\delta}(\rdown{s-}_K)}
\right\rVert_{L^2(\P;\R^d)}^2ds
\right)^\frac{1}{2}\nonumber\\
&\quad
+
\left(\int_{0}^{t}
\left\lVert
\E\! \left[\sigma_d (x, X^d_{K,\delta}(\rdown{s-}_K) )\right]\!|_{x=X^{d}_{K,\delta}(\rdown{s-}_K)}
\right\rVert_{L^2(\P;\R^{d\times d})}^2ds
\right)^\frac{1}{2}\nonumber
\\
&\quad
+
\left(\int_{0}^{t}\int_{A_\delta^d}
\left\lVert
\E\! \left[\eta_d (x, X^d_{K,\delta}(\rdown{s-}_K),z )\right]\!|_{x=X^{d}_{K,\delta}(\rdown{s-}_K)}
\right\rVert_{L^2(\P;\R^{d})}^2\,\nu^d(dz)\,ds
\right)^\frac{1}{2}\nonumber\\
&\leq \kappa_2 d^{\kappa_2}+\kappa_2\left(
\int_{0}^{t}\left\lVert
X^d_{K,\delta}(\rdown{s-}_K)
\right\rVert_{L^2(\P;\R^d)}^2\,ds
\right)^\frac{1}{2}.
\end{align}This shows that there exists $\kappa\in (0,\infty)$ such that for all $d,K\in \N$, $t\in [0,T]$, $\delta\in [0,\infty)$ we have that
\begin{align}
\sup_{u\in [0,t]}
\left\lVert X^{d}_{K,\delta}(u)\right\rVert_{L^2(\P;\R^d)}\leq \kappa d^\kappa+\kappa\left(
\int_{0}^{t}
\sup_{u\in [0,s]}
\left\lVert
X_{K,\delta}^d(u)
\right\rVert_{L^2(\P;\R^d)}^2\,ds
\right)^\frac{1}{2}.
\end{align}
This, Gr\"onwall's lemma, and the fact that 
$
\sup_{s\in [0,T]}\lVert X^d_{K,\delta}(s)\rVert_{L^2(\P;\R^d)}<\infty
$ complete the proof of Lemma~\ref{b01}.
\end{proof}
\subsection{Discretization error}
We establish the time-discretization error. First of all, we need to establish the temporal regularity of the exact solution.
\begin{lemma}[Time regularity of the exact solutions]\label{k01}
Assume Setting~\ref{c01}. 
For every $d\in \N$ let $X^d_\infty$ be the process in Theorem~\ref{b02a}.
Then there exists $\kappa\in (0,\infty)$ such that for all $d\in \N$,
$t_1\in [0,T]$, $t_2\in [t_1,T]$ we have that $\lVert X^d_\infty(t_1)-X^d_\infty(t_2)\rVert_{L^2(\P;\R^d)}
\leq \sqrt{t_2-t_1}\kappa d^\kappa.
$
\end{lemma}
\begin{proof}[Proof of Lemma~\ref{k01}]
The triangle inequality, \eqref{eq:thm_mckv}, Jensen's inequality, It\^o's isometry, and Lemmas~\ref{a05f} and~\ref{b01b} show that there exist  $\kappa_1,\kappa_2,\kappa_3\in (0,\infty)$ such that for all $d\in \N$, $t_1\in [0,T]$, $t_2\in [t_1,T]$ that
\begin{align}
\left\lVert X^{d}_{\infty}(t_1)-X^{d}_{\infty}(t_2)\right\rVert_{L^2(\P;\R^d)}& \leq 
{\kappa_1}
\left(\int_{t_1}^{t_2}
\left\lVert
\E\! \left[\mu_d (x, X^d_{\infty}(s-) )\right]\!|_{x=X^{d}_{\infty}(s-)}
\right\rVert_{L^2(\P;\R^d)}^2ds
\right)^\frac{1}{2}
\nonumber\\
&\quad 
+\left\lVert\int_{t_1}^{t_2}
\E\! \left[\sigma_d (x, X^{d}_{\infty}(s-) )\right]\!|_{x=X^{d}_{\infty}(s-)}
\,dW^{d,0}(s)\right\rVert_{L^2(\P;\R^d)}\nonumber
\\
&\quad +\left\lVert\int_{t_1}^{t_2}\int_{\R^d}
\E\! \left[\eta_d (x, X^{d}_{\infty}(s-),z )\right]\!|_{x=X^{d}_{\infty}(s-)}\,\tilde{N}^{d,0}(dz,ds)
\right\rVert_{L^2(\P;\R^d)}\nonumber\\
&= 
{\kappa_1}
\left(\int_{t_1}^{t_2}
\left\lVert
\E\! \left[\mu_d (x, X^d_{\infty}(s-) )\right]\!|_{x=X^{d}_{\infty}(s-)}
\right\rVert_{L^2(\P;\R^d)}^2ds
\right)^\frac{1}{2}\nonumber\\
&\quad
+
\left(\int_{t_1}^{t_2}
\left\lVert
\E\! \left[\sigma_d (x, X^d_{\infty}(s-) )\right]\!|_{x=X^{d}_{\infty}(s-)}
\right\rVert_{L^2(\P;\R^{d\times d})}^2ds
\right)^\frac{1}{2}\nonumber
\\
&\quad
+
\left(\int_{t_1}^{t_2}\int_{\R^d}
\left\lVert
\E\! \left[\eta_d (x, X^d_{\infty}(s-),z )\right]\!|_{x=X^{d}_{\infty}(s-)}
\right\rVert_{L^2(\P;\R^{d})}^2\,\nu^d(dz)\,ds
\right)^\frac{1}{2}\nonumber\\
&\leq \kappa_2\left(
\int_{t_1}^{t_2}\left(
\kappa_2 d^{\kappa_2}+
\left\lVert
X^d_\infty(s-)
\right\rVert_{L^2(\P;\R^d)}^2\right)ds
\right)^\frac{1}{2}\nonumber\\
&\leq \sqrt{t_2-t_1}\kappa_3 d^{\kappa_3}.
\end{align}
This completes the proof of Lemma~\ref{k01}.
\end{proof}

\begin{lemma}[Discretization error]\label{f07}
For every $d\in \N$ let $X^d_\infty$ be the process in Theorem~\ref{b02a}.
For every $d,K\in \N$ let $X^d_{K,0}$
be the process in Lemma~\ref{a14} for $\delta=0$. Then there exists $\kappa\in (0,\infty)$ such that for all
$d,K\in \N$, $t\in [0,T]$ we have that
$\lVert X^d_{K,0}(t)-X^d_\infty(t)
\rVert_{L^2(\P;\R^d)}\leq  {\kappa d^\kappa}/{ \sqrt{K}}$.
\end{lemma}
\begin{proof}[Proof of Lemma~\ref{f07}]
First, Jensen's inequality, Lemma~\ref{a05f},  the triangle inequality, and Lemma~\ref{k01} show that there exist  $\kappa_1,\kappa_2,\kappa_3\in (0,\infty)$ such that
for all $ d,K\in \N$, $t\in [0,T]$
we have that
\begin{align}
&\left\lVert
\int_{0}^{t}\left(
\E\! \left[\mu_d (x, {X}^{d}_{K,0}(\rdown{s-}_K) )\right]\!|_{x={X}^{d}_{K,0}(\rdown{s-}_K)}
-
\E\! \left[\mu_d (x, X_\infty^d(s-)) \right]\!|_{x=X_\infty^d(s-)}\right)ds\right\rVert_{L^2(\P;\R^d)}
\nonumber\\
&\leq \kappa_1\left(
\int_{0}^{t}
\left\lVert
\E\! \left[\mu_d (x, {X}^{d}_{K,0}(\rdown{s-}_K) )\right]\!|_{x={X}^{d}_{K,0}(\rdown{s-}_K)}
-
\E\! \left[\mu_d (x, X_\infty^d(s-) )\right]\!|_{x=X_\infty^d(s-)}
\right\rVert_{L^2(\P;\R^d)}^2\,ds
\right)^\frac{1}{2}\nonumber\\
&\leq \kappa_2\left(\int_{0}^{t}
\left\lVert
{X}^{d}_{K,0}(\rdown{s-}_K)-
X_\infty^d(s-)
\right\rVert_{L^2(\P;\R^d)}^2ds
\right)^\frac{1}{2}\nonumber\\
&\leq \kappa_2\left(\int_{0}^{t}
\left\lVert
{X}^{d}_{K,0}(\rdown{s-}_K)-
X_\infty^d(\rdown{s-}_K)
\right\rVert_{L^2(\P;\R^d)}^2ds
\right)^\frac{1}{2}\nonumber\\&\quad
+
\kappa_2\left(\int_{0}^{t}
\left\lVert
{X}^{d}_{\infty}(\rdown{s-}_K)-
X_\infty^d(s-)
\right\rVert_{L^2(\P;\R^d)}^2ds
\right)^\frac{1}{2}
\nonumber\\
&\leq 
\kappa_3\left(\int_{0}^{t}
\left\lVert
{X}^{d}_{K,0}(\rdown{s-}_K)-
X_\infty^d(\rdown{s-}_K)
\right\rVert_{L^2(\P;\R^d)}^2ds
\right)^\frac{1}{2}
+\frac{\kappa_3 d^{\kappa_3}}{\sqrt{K}}
.\label{f03b}
\end{align}
Next,
It\^o's isometry, Lemma~\ref{a05f},  the triangle inequality, and  Lemma~\ref{k01} show that there exists $\kappa\in (0,\infty)$ such that
for all $ d,K\in \N$, $t\in [0,T]$ we have that
\begin{align}
&\left\lVert
\int_{0}^{t}\left(
\E\! \left[\sigma_d (x, {X}^{d}_{K,0}(\rdown{s-}_K) )\right]\!|_{x={X}^{d}_{K,0}(\rdown{s-}_K)}
-
\E\! \left[\sigma_d (x, X^d_\infty(s-)) \right]\!|_{x=X^d_\infty(s-)}\right)dW^{d,0}(s)\right\rVert_{L^2(\P;\R^d)}
\nonumber\\
&=\left(
\int_{0}^{t}
\left\lVert
\E\! \left[\sigma_d (x, {X}^{d}_{K,0}(\rdown{s-}_K) )\right]\!|_{x={X}^{d}_{K,0}(\rdown{s-}_K)}
-
\E\! \left[\sigma_d (x, X^d_\infty(s-))\right]\!|_{x=X^d_\infty(s-)}
\right\rVert_{L^2(\P;\R^{d\times d})}^2\,ds
\right)^\frac{1}{2}\nonumber\\
&\leq 2c\left(\int_{0}^{t}
\left\lVert
{X}^{d}_{K,0}(\rdown{s-}_K)-
X^d_\infty(s-)
\right\rVert_{L^2(\P;\R^d)}^2 ds
\right)^\frac{1}{2}
\nonumber\\
&\leq 2c\left(\int_{0}^{t}
\left\lVert
{X}^{d}_{K,0}(\rdown{s-}_K)-
X_\infty^d(\rdown{s-}_K)
\right\rVert_{L^2(\P;\R^d)}^2 ds
\right)^\frac{1}{2}\nonumber\\&\quad
+
2c\left(\int_{0}^{t}
\left\lVert
{X}^{d}_{\infty}(\rdown{s-}_K)-
X_\infty^d(s-)
\right\rVert_{L^2(\P;\R^d)}^2 ds
\right)^\frac{1}{2}
\nonumber\\
&\leq 2c\left(\int_{0}^{t}
\left\lVert
{X}^{d}_{K,0}(\rdown{s-}_K)-
X_\infty^d(\rdown{s-}_K)
\right\rVert_{L^2(\P;\R^d)}^2ds
\right)^\frac{1}{2}
+\frac{\kappa d^\kappa}{\sqrt{K}}
\label{f04b}
\end{align}
and
\begin{align}
&\biggl\lVert
\int_{0}^{t}\int_{\R^d}\Bigl(
\E\! \left[\eta_d (x, {X}^{d}_{K,0}(\rdown{s-}_K),z )\right]\!|_{x={X}^{d}_{K,0}(\rdown{s-}_K)}\nonumber\\&\qquad\qquad\qquad-
\E\! \left[\eta_d (x, X^d_\infty(s-),z )\right]\!|_{x=X^d_\infty(s-)}\Bigr)\tilde{N}^{d,0}(dz,ds)\biggr\rVert_{L^2(\P;\R^d)}
\nonumber\\
&=\biggl(
\int_{0}^{t}
\int_{\R^d}
\Bigl\lVert
\E\! \left[\eta_d (x, {X}^{d}_{K,0}(\rdown{s-}_K),z )\right]\!|_{x={X}^{d}_{K,0}(\rdown{s-}_K)}\nonumber\\&\qquad\qquad\qquad
-
\E\! \left[\eta_d (x, X^d_\infty(s-),z )\right]\!|_{x=X^d_\infty(s-)}
\Bigr\rVert_{L^2(\P;\R^d)}^2\nu^d(dz)\,ds
\biggr)^\frac{1}{2}\nonumber\\
&\leq \sqrt{2c}\left(\int_{0}^{t}
\left\lVert
{X}^{d}_{K,0}(\rdown{s-}_K)-
X^d_\infty(s-)
\right\rVert_{L^2(\P;\R^d)}^2 ds
\right)^\frac{1}{2}
\nonumber\\
&\leq \sqrt{2c}\left(\int_{0}^{t}
\left\lVert
{X}^{d}_{K,0}(\rdown{s-}_K)-
X_\infty^d(\rdown{s-}_K)
\right\rVert_{L^2(\P;\R^d)}^2 ds
\right)^\frac{1}{2}\nonumber\\
&\qquad
+
\sqrt{2c}\left(\int_{0}^{t}
\left\lVert
{X}^{d}_{\infty}(\rdown{s-}_K)-
X_\infty^d(s-)
\right\rVert_{L^2(\P;\R^d)}^2 ds
\right)^\frac{1}{2}
\nonumber\\
&\leq \sqrt{2c}\left(\int_{0}^{t}
\left\lVert
{X}^{d}_{K,0}(\rdown{s-}_K)-
X_\infty^d(\rdown{s-}_K)
\right\rVert_{L^2(\P;\R^d)}^2 ds
\right)^\frac{1}{2}
+\frac{\kappa d^\kappa}{\sqrt{K}}
.\label{f05b}
\end{align}
Now, \eqref{eq:thm_mckv}, 
\eqref{b02b} (with $\delta\gets 0$),  the triangle inequality, and
\eqref{f03b}--\eqref{f05b} imply 
that there exists $\kappa\in (0,\infty)$ such that
for all
$d,K\in \N$, $t\in [0,T]$ we have that
\begin{align}
\left\lVert {X}^d_{K,0}(t)
-X^d_\infty(t)\right\rVert_{L^2(\P;\R^d)}\leq \kappa 
\left(
\int_{0}^{t}
\left\lVert {X}^d_{K,0}(\rdown{s-}_K)
-X^d_\infty(\rdown{s-}_K)\right\rVert_{L^2(\P;\R^d)}^2ds
\right)^\frac{1}{2}
+\frac{\kappa d^\kappa}{\sqrt{K}}
.
\end{align}
This and Gr\"onwall's lemma complete the proof of Lemma~\ref{f07}.
\end{proof}

\subsection{Truncation error} In Lemma~\ref{b04} below we establish the distance between the Euler-Maruyama approximations with truncation ($\delta>0$) and that without truncation ($\delta=0$).
\begin{lemma}[Truncation error]\label{b04}Assume Setting~\ref{s01}, i.e., the case of infinite jump-activity.
For every $d,K\in \N$, $\delta\in (0,\infty)$ let $X^{d}_{K,\delta}$
be the process in Lemma~\ref{a14}. Then there exists $\kappa\in (0,\infty)$ such that for all
$d,K\in \N$, $t\in [0,T]$, $\delta\in (0,\infty)$ we have
$\lVert X^d_{K,0}(t)
-X^d_{K,\delta}(t)
\rVert_{L^2(\P;\R^d)}\leq  \delta^{\frac{q}{2}}\kappa d^\kappa$.
\end{lemma}
\begin{proof}[Proof of Lemma~\ref{b04}]
First, Jensen's inequality and Lemma~\ref{a05f} show that there exist  $\kappa_1,\kappa_2\in (0,\infty)$ such that
for all $ d,K\in \N$, $t\in [0,T]$, $\delta\in (0,\infty)$
we have that
\begin{align}
&\biggl\lVert
\int_{0}^{t}\Bigl(
\E\! \left[\mu_d (x, X^{d}_{K,\delta}(\rdown{s-}_K) )\right]\!|_{x=X^{d}_{K,\delta}(\rdown{s-}_K)}
\nonumber\\
&\qquad\qquad\qquad
-
\E\! \left[\mu_d (x, X^{d}_{K,0}(\rdown{s-}_K) )\right]\!|_{x=X^{d}_{K,0}(\rdown{s-}_K)}\Bigr)ds\biggr\rVert_{L^2(\P;\R^d)}
\nonumber\\
&\leq \kappa_1\biggl(
\int_{0}^{t}
\biggl\lVert
\E\! \left[\mu_d (x, X^{d}_{K,\delta}(\rdown{s-}_K) )\right]\!|_{x=X^{d}_{K,\delta}(\rdown{s-}_K)}\nonumber\\
&\qquad\qquad\qquad-
\E\! \left[\mu_d (x, X^{d}_{K,0}(\rdown{s-}_K) )\right]\!|_{x=X^{d}_{K,0}(\rdown{s-}_K)}
\biggr\rVert_{L^2(\P;\R^d)}^2\,ds
\biggr)^\frac{1}{2}\nonumber\\
&\leq \kappa_2\left(\int_{0}^{t}
\left\lVert
X^{d}_{K,\delta}(\rdown{s-}_K)-
X^{d}_{K,0}(\rdown{s-}_K)
\right\rVert_{L^2(\P;\R^d)}^2 ds
\right)^\frac{1}{2}.\label{f03}
\end{align}
Next,
It\^o's isometry and Lemma~\ref{a05f} show
for all $ d,K\in \N$, $\delta\in (0,\infty)$, $t\in [0,T]$ that
\begin{align}
&\biggl\lVert
\int_{0}^{t}\biggl(
\E\! \left[\sigma_d (x, X^{d}_{K,\delta}(\rdown{s-}_K) )\right]\!|_{x=X^{d}_{K,\delta}(\rdown{s-}_K)}\nonumber\\
&\qquad\qquad\qquad-
\E\! \left[\sigma_d (x, X^{d}_{K,0}(\rdown{s-}_K) )\right]\!|_{x=X^{d}_{K,0}(\rdown{s-}_K)}\biggr)dW^{d,0}(s)\biggr\rVert_{L^2(\P;\R^d)}
\nonumber\\
&=\biggl(
\int_{0}^{t}
\biggl\lVert
\E\! \left[\sigma_d (x, X^{d}_{K,\delta}(\rdown{s-}_K) )\right]\!|_{x=X^{d}_{K,\delta}(\rdown{s-}_K)}
\nonumber\\
&\qquad\qquad\qquad-
\E\! \left[\sigma_d (x, X^{d}_{K,0}(\rdown{s-}_K) )\right]\!|_{x=X^{d}_{K,0}(\rdown{s-}_K)}
\biggr\rVert_{L^2(\P;\R^{d\times d})}^2\,ds
\biggr)^\frac{1}{2}\nonumber\\
&\leq 2c\left(\int_{0}^{t}
\left\lVert
X^{d}_{K,\delta}(\rdown{s-}_K)-
X^{d}_{K,0}(\rdown{s-}_K)
\right\rVert_{L^2(\P;\R^d)}^2 ds
\right)^\frac{1}{2}\label{f04}
\end{align}
and
\begin{align}
&\biggl\lVert
\int_{0}^{t}\int_{A^d_\delta}\biggl(
\E\! \left[\eta_d (x, X^{d}_{K,\delta}(\rdown{s-}_K),z )\right]\!|_{x=X^{d}_{K,\delta}(\rdown{s-}_K)}\nonumber\\
&\qquad\qquad\qquad-
\E\! \left[\eta_d (x, X^{d}_{K,0}(\rdown{s-}_K),z )\right]\!|_{x=X^{d}_{K,0}(\rdown{s-}_K)}\biggr)\tilde{N}^{d,0}(dz,ds)\biggr\rVert_{L^2(\P;\R^d)}
\nonumber\\
&=\biggl(
\int_{0}^{t}
\int_{A^d_\delta}
\biggl\lVert
\E\! \left[\eta_d (x, X^{d}_{K,\delta}(\rdown{s-}_K),z )\right]\!|_{x=X^{d}_{K,\delta}(\rdown{s-}_K)}
\nonumber\\&\qquad\qquad\qquad
-
\E\! \left[\eta_d (x, X^{d}_{K,0}(\rdown{s-}_K),z )\right]\!|_{x=X^{d}_{K,0}(\rdown{s-}_K)}
\biggr\rVert_{L^2(\P;\R^d)}^2\nu^d(dz)\,ds
\biggr)^\frac{1}{2}\nonumber\\
&\leq \sqrt{2c}\left(\int_{0}^{t}
\left\lVert
X^{d}_{K,\delta}(\rdown{s-}_K)-
X^{d}_{K,0}(\rdown{s-}_K)
\right\rVert_{L^2(\P;\R^d)}^2 ds
\right)^\frac{1}{2}.\label{f05}
\end{align}
Next, It\^o's isometry, Tonelli's theorem, \eqref{e13}, and Lemma~\ref{b01} show that there exists $\kappa\in (0,\infty)$ such that for all $d\in \N$, $K\in \N$, $t\in [0,T]$, $\delta\in (0,\infty)$ we have
that
\begin{align}
&
\left\lVert
\int_{0}^{t}\int_{\R^d\setminus A^d_\delta}
\E\! \left[\eta_d (x, X^d_{K,0}(\rdown{s-}_K),z )\right]\!|_{x=X^d_{K,0}(\rdown{s-}_K)}\,\tilde{N}^{d,0}(dz,ds)
\right\rVert_{L^2(\P;\R^d)}\nonumber\\
&=\left(
\int_{0}^{t}\int_{\R^d\setminus A^d_\delta}
\left\lVert
\E\! \left[\eta_d (x, X^d_{K,0}(\rdown{s-}_K),z )\right]\!|_{x=X^d_{K,0}(\rdown{s-}_K)}
\right\rVert_{L^2(\P;\R^d)}^2\nu^d(dz)\,ds\right)^\frac{1}{2}\nonumber
\\
&=\left(
\int_{0}^{t}\int_{\R^d\setminus A^d_\delta}
\int_\Omega\int_\Omega \lVert\eta_d (x,y,z)\rVert^2
\,\P_{X^d_{K,0}(\rdown{s-}_K)}(dy)\,\P_{X^d_{K,0}(\rdown{s-}_K)}(dx)\,
\nu^d(dz)\,ds\right)^\frac{1}{2}\nonumber
\\
&=\left(
\int_{0}^{t}
\int_\Omega\int_\Omega
\int_{\R^d\setminus A^d_\delta}
 \lVert\eta_d (x,y,z)\rVert^2
\nu^d(dz)
\,\P_{X^d_{K,0}(\rdown{s-}_K)}(dy)\,\P_{X^d_{K,0}(\rdown{s-}_K)}(dx)
\,ds\right)^\frac{1}{2}\nonumber
\\
&\leq \left(
\int_{0}^{t}
\int_\Omega\int_\Omega
c\delta^{q}(cd^c+\lVert x\rVert^2+\lVert y\rVert^2)\,\P_{X^d_{K,0}(\rdown{s-}_K)}(dy)\,\P_{X^d_{K,0}(\rdown{s-}_K)}(dx)
\,ds\right)^\frac{1}{2}\nonumber
\\
&\leq \delta^\frac{q}{2}\kappa d^\kappa.\label{f06}
\end{align}
Now,
 \eqref{b02b}, the triangle inequality, and
\eqref{f03}--\eqref{f06} imply 
that there exists $\kappa\in (0,\infty)$ such that
for all
$d\in \N$, $K\in \N$, $t\in [0,T]$, $\delta\in (0,\infty)$ we have that
\begin{align}
&\left\lVert X^d_{K,\delta}(t)
-X^d_{K,0}(t)\right\rVert_{L^2(\P;\R^d)}\nonumber\\&\leq \kappa
\left(
\int_{0}^{t}
\left\lVert X^d_{K,\delta}(\rdown{s-}_K)
-X^d_{K,0}(\rdown{s-}_K)\right\rVert_{L^2(\P;\R^d)}^2ds
\right)^\frac{1}{2}+
 \delta^{\frac{q}{2}}\kappa d^\kappa
.
\end{align}
This and Gr\"onwall's lemma complete the proof of Lemma~\ref{b04}.
\end{proof}
\subsection{Integrability of the approximations}We will show that in both cases (finite and infinite jump-activity) the MLP approximations defined in \eqref{s02} are $L^2$-integrable.

First of all, we formulate a substitution rule for the random variables $\rvV^{d,\theta,\delta}$ introduced in \eqref{f09}.

\begin{remark}[Substitution rule for $\rvV_s^{d,\theta,\delta}$]\label{r01b}
Note that \eqref{f09} implies that in both cases (finite and infinite jump-activity) we have
for all 
$d\in \N$,
$B\in \mathcal{B}(\R^d)$, $s\in [0,T]$, $\delta\in \Delta$, $\theta\in \Theta$ that
\begin{align}
\E \!\left[\1_B(\rvV^{d,\theta,\delta}_s)\right]=
\frac{\nu^d(B\cap A^d_\delta)}{\nu^d(A_\delta^d)}
= \frac{1}{\nu^d( A_\delta^d)}
\int_{A^d_\delta}\1_B(z)\,\nu^d(dz)
\end{align}
when $\nu^d(A^d_\delta)>0$.
This and a standard approximation argument imply that in both cases we have for all 
$d\in \N$,
$s\in [0,T]$, $\delta\in \Delta$, and all
measurable functions $f\colon \R^d\to[0,\infty)$  that
\begin{align}
\nu^d(A_\delta^d)
\E [f(\rvV^{d,\theta,\delta}_s)]= 
\int_{A_\delta^d}f(z)\,\nu^d(dz).
\end{align}
\end{remark}
\begin{lemma}\label{s04}
The following items are true in both cases (finite and infinite jump-activity).
\begin{enumerate}[(i)]
\item \label{s04a}For all $\theta\in \Theta$, $\delta\in \Delta$,
$d,K,m,k\in \N$, 
$n\in \N_0$,
$\ell\in [0,n-1]\cap\Z$, $j\in \{\ell-1,\ell\}$, $i\in \N$  we have that 
\begin{align}
&\sup_{t\in [0,T]}\Biggl[
\left\lVert
X^{d,\theta}_{n,m,K,\delta}(t)\right\rVert_{L^2(\P;\R^d)}
+\left\lVert t
\mu_d (X_{j,m,K,\delta}^{d,\theta} (\rdown{t\unif^{(\theta,n,k,\ell)}-}_K ),X_{j,m,K,\delta}^{d,(\theta,n,k,\ell)} (\rdown{t\unif^{(\theta,n,k,\ell)}-}_K ))
\right\rVert_{L^2(\P;\R^d)}\nonumber\\
&\quad +
\left\lVert
\sigma_d(X_{j,m,K,\delta}^{d,\theta} (\rdown{t-}_K ),X_{j,m,K,\delta}^{d,(\theta,n,k,\ell)} (\rdown{t-}_K ))\right\rVert_{L^2(\P;\R^{d\times d})}\nonumber\\
&\quad+
\left(
\int_{\R^d}
\left\lVert
\eta_d(X_{j,m,K,\delta}^{d,\theta} (\rdown{t-}_K ),X_{j,m,K,\delta}^{d,(\theta,n,k,\ell)} (\rdown{t-}_K ),z)\right\rVert_{L^2(\P;\R^{d})}^2\,\nu^d(dz)
\right)^\frac{1}{2}\nonumber\\
&\quad + 
\left\lVert\nu^d(A^d_\delta)
\eta_d(X_{j,m,K,\delta}^{d,\theta} (\rdown{t-}_K ),X_{j,m,K,\delta}^{d,(\theta,n,k,\ell)} (\rdown{t-}_K ),\rvV^{d,(\theta,n,k,\ell,i),\delta}_{\rdown{t}_K })\right\rVert_{L^2(\P;\R^{d})}\Biggr]<\infty.
\end{align}
\item\label{s04b} For all $\theta\in \Theta$, $\delta\in \Delta$,
$n\in \N_0$,
$d,K,m\in \N$, $t\in [0,T]$ we have that 
\begin{align}
&
X^{d,\theta}_{n,m,K,\delta}(t)=  \xi_d \1_{\N}(n)
\nonumber\\ 
&+ \sum_{\ell=0}^{n-1}\sum_{k=1}^{m^{n-\ell}}
\frac{t}{m^{n-\ell}}
\Bigl[
\mu_d (X_{\ell,m,K,\delta}^{d,\theta} (\rdown{t\unif^{(\theta,n,k,\ell)}-}_K ),X_{\ell,m,K,\delta}^{d,(\theta,n,k,\ell)} (\rdown{t\unif^{(\theta,n,k,\ell)}-}_K ))\nonumber\\
&\qquad\qquad\qquad\qquad
-\1_{\N}(\ell)
\mu_d (X_{\ell-1,m,K,\delta}^{d,\theta} (\rdown{t\unif^{(\theta,n,k,\ell)}-}_K ),X_{\ell-1,m,K,\delta}^{d,(\theta,n,k,\ell)} (\rdown{t\unif^{(\theta,n,k,\ell)}-}_K ))
\Bigr]
\nonumber\\
&+ \sum_{\ell=0}^{n-1}\sum_{k=1}^{m^{n-\ell}}\int_{0}^{t}
\frac{1}{m^{n-\ell}}
\Bigl[
\sigma_d(X_{\ell,m,K,\delta}^{d,\theta} (\rdown{s-}_K ),X_{\ell,m,K,\delta}^{d,(\theta,n,k,\ell)} (\rdown{s-}_K ))\nonumber\\
&\qquad\qquad\qquad\qquad\qquad
-\1_{\N}(\ell)
\sigma_d(X_{\ell-1,m,K,\delta}^{d,\theta} (\rdown{s-}_K ),X_{\ell-1,m,K,\delta}^{d,(\theta,n,k,\ell)} (\rdown{s-}_K ))\Bigr]
dW^{d,\theta}(s)
\nonumber\\
&+ \sum_{\ell=0}^{n-1}\frac{1}{m^{n-\ell}}\sum_{k=1}^{m^{n-\ell}}\int_{0}^{t}\int_{A_\delta^d}
\Bigl[\eta_d(X_{\ell,m,K,\delta}^{d,\theta} (\rdown{s-}_K ),X_{\ell,m,K,\delta}^{d,(\theta,n,k,\ell)} (\rdown{s-}_K ),z)\nonumber\\
&\qquad\qquad\qquad\qquad\qquad\qquad
-\1_{\N}(\ell)
\eta_d(X_{\ell-1,m,K,\delta}^{d,\theta} (\rdown{s-}_K ),X_{\ell-1,m,K,\delta}^{d,(\theta,n,k,\ell)} (\rdown{s-}_K ),z)\Bigr]
\tilde{N}^{d,\theta}(dz,ds)\nonumber\\
&- \sum_{\ell=0}^{n-1}\frac{1}{m^{n-\ell}}\sum_{k=1}^{m^{n-\ell}}
\frac{1}{\mathcal{M}^d_\delta}
\sum_{i=1}^{\mathcal{M}^d_\delta}\int_{0}^{t}\biggl[
\nu^d(A^d_\delta)
\Bigl[\eta_d(X_{\ell,m,K,\delta}^{d,\theta} (\rdown{s-}_K ),X_{\ell,m,K,\delta}^{d,(\theta,n,k,\ell)} (\rdown{s-}_K ),\rvV^{d,(\theta,n,k,\ell,i),\delta}_{\rdown{s}_K })\nonumber\\
&\qquad\qquad\qquad\qquad
-\1_\N(\ell)
\eta_d(X_{\ell-1,m,K,\delta}^{d,\theta} (\rdown{s-}_K ),X_{\ell-1,m,K,\delta}^{d,(\theta,n,k,\ell)} (\rdown{s-}_K ),\rvV^{d,(\theta,n,k,\ell,i),\delta}_{\rdown{s}_K })\Bigr]\nonumber\\
&\qquad\qquad\qquad
-\int_{A^d_\delta}
\Bigl[\eta_d(X_{\ell,m,K,\delta}^{d,\theta} (\rdown{s-}_K ),X_{\ell,m,K,\delta}^{d,(\theta,n,k,\ell)} (\rdown{s-}_K ),z)\nonumber\\
&\qquad\qquad\qquad\qquad\qquad
-\1_\N(\ell)
\eta_d(X_{\ell-1,m,K,\delta}^{d,\theta} (\rdown{s-}_K ),X_{\ell-1,m,K,\delta}^{d,(\theta,n,k,\ell)} (\rdown{s-}_K ),z)\Bigr]\nu^d(dz)\biggr]\,ds,\label{s04c}
\end{align}
where we used the convention that $\sum_{\ell=0}^{-1}(\ldots)=0$.
\end{enumerate}
\end{lemma}
\begin{proof}[Proof of Lemma~\ref{s04}]We prove \eqref{s04a} and \eqref{s04b} by induction on $n\in \N_0$. The base case is clear due to \eqref{s05}. Assume that \eqref{s04a} and \eqref{s04b} are true up to $n-1\in \N_0$. We need to show that they are true for $n\in \N_0$.
Observe that we only need to explain that all the integrals in \eqref{s04b} are well-defined in both cases as then \eqref{s04a} follows from \eqref{s03b}.
By the disintegration theorem, 
the distribution and independence properties, \eqref{u01b},
\eqref{e11}, \eqref{e11b}, the triangle inequality, \eqref{e12}, and the induction hypothesis we have in both cases for all $\theta\in \Theta$, $\delta\in \Delta$,
$d,K,m\in \N$, $\ell\in [0,n-1]\cap\Z$, $j\in \{\ell-1,\ell\}$   that 
\begin{align}
&\sup_{t\in [0,T]}
\left\lVert t
\mu_d (X_{j,m,K,\delta}^{d,\theta} (\rdown{t\unif^{(\theta,n,k,\ell)}-}_K ),X_{j,m,K,\delta}^{d,(\theta,n,k,\ell)} (\rdown{t\unif^{(\theta,n,k,\ell)}-}_K ))
\right\rVert_{L^2(\P;\R^d)}\nonumber\\
&
=\sup_{t\in [0,T]}\left[\sqrt{t}\left(\int_0^t\left\lVert
\mu_d 
(X_{j,m,K,\delta}^{d,\theta} (\rdown{s-}_K ),X_{j,m,K,\delta}^{d,(\theta,n,k,\ell)} (\rdown{s-}_K ))\right\rVert^2_{L^2(\P;\R^d)}ds
\right)^\frac{1}{2}\right]\nonumber\\
&\leq \sqrt{T}\left[\kappa d^\kappa + \kappa\sup_{t\in [0,T]}\left\lVert
X^{d,0}_{j,m,K,\delta}(t)
\right\rVert_{L^2(\P;\R^d)}\right]<\infty,\label{s06}
\end{align}
\begin{align}
&\sup_{t\in [0,T]}
\left\lVert
\sigma_d(X_{j,m,K,\delta}^{d,\theta} (\rdown{t-}_K ),X_{j,m,K,\delta}^{d,(\theta,n,k,\ell)} (\rdown{t-}_K ))\right\rVert_{L^2(\P;\R^{d\times d})}\nonumber\\&\leq cd^c+ 2c\sup_{t\in [0,T]}\left\lVert
X^{d,0}_{j,m,K,\delta}(t)
\right\rVert_{L^2(\P;\R^d)}<\infty,
\end{align} and
\begin{align}
&\sup_{t\in [0,T]}\left(
\int_{\R^d}
\left\lVert
\eta_d(X_{j,m,K,\delta}^{d,\theta} (\rdown{t-}_K ),X_{j,m,K,\delta}^{d,(\theta,n,k,\ell)} (\rdown{t-}_K ),z)\right\rVert_{L^2(\P;\R^{d})}^2\,\nu^d(dz)
\right)^\frac{1}{2}
\nonumber\\&\leq cd^c+ 2c\sup_{t\in [0,T]}\left\lVert
X^{d,0}_{j,m,K,\delta}(t)
\right\rVert_{L^2(\P;\R^d)}<\infty.
\end{align}
Next, the disintegration theorem, 
the distribution and independence properties, \eqref{f09} (see also Remark~\ref{r01b}),
\eqref{e11b}, the triangle inequality, and \eqref{e12} show in both cases that for all $\theta\in \Theta$, $\delta\in \Delta$,
$d,K,m,i\in \N$, $\ell\in [0,n-1]\cap\Z$, $j\in \{\ell-1,\ell\}$ we have that
\begin{align}
&\sup_{t\in [0,T]}
\left(
\left\lVert\nu^d(A^d_\delta)
\eta_d(X_{j,m,K,\delta}^{d,\theta} (\rdown{t-}_K ),X_{j,m,K,\delta}^{d,(\theta,n,k,\ell)} (\rdown{t-}_K ),\rvV^{d,(\theta,n,k,\ell,i),\delta}_{\rdown{t}_K })\right\rVert_{L^2(\P;\R^{d})}^2
\right)^\frac{1}{2}\nonumber\\
&=\sup_{t\in [0,T]}
\left(
\int_{A^d_\delta}\nu^d(A^d_\delta)\left\lVert
\eta_d(X_{j,m,K,\delta}^{d,\theta} (\rdown{t-}_K ),X_{j,m,K,\delta}^{d,(\theta,n,k,\ell)} (\rdown{t-}_K ),z)\right\rVert_{L^2(\P;\R^{d})}^2\nu^d(dz)
\right)^\frac{1}{2}\nonumber\\
&\leq \sqrt{\nu^d(A^d_\delta)}\sup_{t\in [0,T]}\left(cd^c+\sqrt{2c}
\sup_{t\in [0,T]}\left\lVert
X^{d,0}_{j,m,K,\delta}(t)
\right\rVert_{L^2(\P;\R^d)}
\right)<\infty.\label{s07}
\end{align}
Combining \eqref{s06}--\eqref{s07} implies that all the integrals in \eqref{s04b} are well-defined and \eqref{s04b} holds. The triangle inequality, \eqref{s04b}, and \eqref{s06}--\eqref{s07} imply that $\sup_{t\in [0,T]}
\lVert
X^{d,\theta}_{n,m,K,\delta}(t)\rVert_{L^2(\P;\R^d)}<\infty$. The proof of Lemma~\ref{s04} is thus completed.
\end{proof}

\subsection{Statistical error}We will estimate the statistical error of the MLP approximations.
\begin{lemma}[Conditional variance of the first sum]\label{k03}
In both cases (finite and infinite jump-activity) there exists $\kappa\in (0,\infty)$ such that for all 
$d,n,m,K\in \N$, $\delta\in \Delta$, $t\in [0,T]$ we have
\begin{align}
&\Biggl \lVert\Biggl(\V_{\mathcal{G}^{d}_{n-1,m,K,\delta}}
\Biggl[\sum_{\ell=0}^{n-1}\sum_{k=1}^{m^{n-\ell}}
\frac{t}{m^{n-\ell}}
\Bigl[
\mu_d (X_{\ell,m,K,\delta}^{d,0} (\rdown{t\unif^{(0,n,k,\ell)}-}_K ),X_{\ell,m,K,\delta}^{d,(0,n,k,\ell)} (\rdown{t\unif^{(0,n,k,\ell)}-}_K ))\nonumber\\
&\qquad\qquad
-\1_{\N}(\ell)
\mu_d (X_{\ell-1,m,K,\delta}^{d,0} (\rdown{t\unif^{(0,n,k,\ell)}-}_K ),X_{\ell-1,m,K,\delta}^{d,(0,n,k,\ell)} (\rdown{t\unif^{(0,n,k,\ell)}-}_K ))
\Bigr]\Biggr] \Biggr)^\frac{1}{2}\Biggr \rVert_{L^2(\P;\R^d)}
\nonumber\\
&\leq \frac{\kappa d^\kappa}{\sqrt{m^n}}
+\sum_{\ell=1}^{n-1}\frac{\kappa}{\sqrt{m^{n-\ell}}}
\left(\int_{0}^{t}\left\lVert X_{\ell,m,K,\delta}^{d,0} (\rdown{s-}_K ) 
-X_{\ell-1,m,K,\delta}^{d,0} (\rdown{s-}_K )
\right\rVert_{L^2(\P;\R^d)}^2 ds \right)^\frac{1}{2}.
\end{align}
\end{lemma}
\begin{proof}[Proof of Lemma~\ref{k03}]
The estimate is similar to (43) in Step 5 in the proof of Theorem 1 in \cite{Neufeld2025}.
\end{proof}

\begin{lemma}[Conditional variance of the second sum]\label{k04}
In both cases (finite and infinite jump-activity) there exists $\kappa\in (0,\infty)$ such that for all 
$d,n,m,K\in \N$, $\delta\in \Delta$, $t\in [0,T]$ we have
\begin{align}
&\Biggl \lVert\Biggl(\V_{\mathcal{G}^d_{n-1,m,K,\delta}}\Biggl[\sum_{\ell=0}^{n-1}\sum_{k=1}^{m^{n-\ell}}\int_{0}^{t}
\frac{1}{m^{n-\ell}}
\Bigl[
\sigma_d(X_{\ell,m,K,\delta}^{d,0} (\rdown{s-}_K ),X_{\ell,m,K,\delta}^{d,(0,n,k,\ell)} (\rdown{s-}_K ))\nonumber\\
&\qquad\qquad
-\1_{\N}(\ell)
\sigma_d(X_{\ell-1,m,K,\delta}^{d,0} (\rdown{s-}_K ),X_{\ell-1,m,K,\delta}^{d,(0,n,k,\ell)} (\rdown{s-}_K ))\Bigr]
dW^{d,0}(s)\Biggr] \Biggr)^\frac{1}{2}\Biggr \rVert_{L^2(\P;\R^d)}
\nonumber\\
&\leq \frac{\kappa d^\kappa}{\sqrt{m^n}}
+\sum_{\ell=1}^{n-1}\frac{\kappa}{\sqrt{m^{n-\ell}}}
\left(\int_{0}^{t}\left\lVert X_{\ell,m,K,\delta}^{d,0} (\rdown{s-}_K ) 
-X_{\ell-1,m,K,\delta}^{d,0} (\rdown{s-}_K )
\right\rVert_{L^2(\P;\R^d)}^2 ds \right)^\frac{1}{2}.
\end{align}
\end{lemma}

\begin{proof}[Proof of Lemma~\ref{k04}]
The estimate is similar to (44) in Step 5 in the proof of Theorem 1 in \cite{Neufeld2025}: we use the triangle inequality, Bienaym\'e's identity, It\^o's isometry, and the Lipschitz condition of $\sigma_d$ (see \eqref{e11}).
\end{proof}

\begin{lemma}[Conditional variance of the third sum]\label{k06}In both cases (finite and infinite jump-activity) there exists $\kappa\in (0,\infty)$ such that  for all 
$d,n,m,K\in \N$, $\delta\in \Delta$,  $t\in [0,T]$ we have
\begin{align}
&\Biggl \lVert\Biggl(\V_{\mathcal{G}^d_{n-1,m,K,\delta}}\Biggl[ \sum_{\ell=0}^{n-1}\frac{1}{m^{n-\ell}}\sum_{k=1}^{m^{n-\ell}}\int_{0}^{t}\int_{A^d_\delta}
\Bigl[\eta_d(X_{\ell,m,K,\delta}^{d,0} (\rdown{s-}_K ),X_{\ell,m,K,\delta}^{d,(0,n,k,\ell)} (\rdown{s-}_K ),z)\nonumber\\
&\qquad\qquad
-\1_{\N}(\ell)
\eta_d(X_{\ell-1,m,K,\delta}^{d,0} (\rdown{s-}_K ),X_{\ell-1,m,K,\delta}^{d,(0,n,k,\ell)} (\rdown{s-}_K ),z)\Bigr]
\tilde{N}^{d,0}(dz,ds)\Biggr] \Biggr)^\frac{1}{2}\Biggr \rVert_{L^2(\P;\R^d)}\nonumber\\
&\leq  \frac{\kappa d^\kappa}{\sqrt{m^n}}
+\sum_{\ell=1}^{n-1}\frac{\kappa}{\sqrt{m^{n-\ell}}}
\left(\int_{0}^{t}\left\lVert X_{\ell,m,K,\delta}^{d,0} (\rdown{s-}_K ) 
-X_{\ell-1,m,K,\delta}^{d,0} (\rdown{s-}_K )
\right\rVert_{L^2(\P;\R^d)}^2 ds\right)^\frac{1}{2}.
\end{align}
\end{lemma}

\begin{proof}[Proof of Lemma~\ref{k06}]First, the triangle inequality, Bienaym\'e's identity, the tower property, It\^o's isometry, \eqref{e11b}, and \eqref{e12} show that in both cases we have for all 
$d,n,m,K\in \N$, $\delta\in \Delta$, $t\in [0,T]$ that
\begin{align}
&\Biggl \lVert
\Biggl(\V_{\mathcal{G}^d_{n-1,m,K,\delta}}\Biggl[ \sum_{\ell=0}^{n-1}\frac{1}{m^{n-\ell}}\sum_{k=1}^{m^{n-\ell}}\int_{0}^{t}\int_{A^d_\delta}
\Bigl[\eta_d(X_{\ell,m,K,\delta}^{d,0} (\rdown{s-}_K ),X_{\ell,m,K,\delta}^{d,(0,n,k,\ell)} (\rdown{s-}_K ),z)\nonumber\\
&\qquad\qquad
-\1_{\N}(\ell)
\eta_d(X_{\ell-1,m,K,\delta}^{d,0} (\rdown{s-}_K ),X_{\ell-1,m,K,\delta}^{d,(0,n,k,\ell)} (\rdown{s-}_K ),z)\Bigr]
\tilde{N}^{d,0}(dz,ds)\Biggr]\Biggr)^\frac{1}{2}\Biggr \rVert_{L^2(\P;\R^d)}
\nonumber\\
&\leq 
\sum_{\ell=0}^{n-1}
\frac{1}{\sqrt{m^{n-\ell}}}\Biggl \lVert\Biggl(\V_{\mathcal{G}^d_{n-1,m,K,\delta}}\Biggl[\int_{0}^{t}\int_{A^d_\delta}
\Bigl[\eta_d(X_{\ell,m,K,\delta}^{d,0} (\rdown{s-}_K ),X_{\ell,m,K,\delta}^{d,(0,n,k,\ell)} (\rdown{s-}_K ),z)\nonumber\\
&\qquad\qquad
-\1_{\N}(\ell)
\eta_d(X_{\ell-1,m,K,\delta}^{d,0} (\rdown{s-}_K ),X_{\ell-1,m,K,\delta}^{d,(0,n,k,\ell)} (\rdown{s-}_K ),z)\Bigr]
\tilde{N}^{d,0}(dz,ds)\Biggr]\Biggr)^\frac{1}{2}\Biggr \rVert_{L^2(\P;\R^d)}
\nonumber\\
&\leq 
\sum_{\ell=0}^{n-1}
\frac{1}{\sqrt{m^{n-\ell}}}\Biggl \lVert\Biggl(\E_{\mathcal{G}^d_{n-1,m,K,\delta}}\Biggl[\biggl \lVert\int_{0}^{t}\int_{A^d_\delta}
\Bigl[\eta_d(X_{\ell,m,K,\delta}^{d,0} (\rdown{s-}_K ),X_{\ell,m,K,\delta}^{d,(0,n,k,\ell)} (\rdown{s-}_K ),z)\nonumber\\
&\qquad\qquad
-\1_{\N}(\ell)
\eta_d(X_{\ell-1,m,K,\delta}^{d,0} (\rdown{s-}_K ),X_{\ell-1,m,K,\delta}^{d,(0,n,k,\ell)} (\rdown{s-}_K ),z)\Bigr]
\tilde{N}^{d,0}(dz,ds)\biggr \rVert^2\Biggr]\Biggr)^\frac{1}{2}\Biggr \rVert_{L^2(\P;\R^d)}
\nonumber\\
&= 
\sum_{\ell=0}^{n-1}
\frac{1}{\sqrt{m^{n-\ell}}}\Biggl(\E\E_{\mathcal{G}^d_{n-1,m,K,\delta}}\Biggl[\biggl \lVert\int_{0}^{t}\int_{A^d_\delta}
\Bigl[\eta_d(X_{\ell,m,K,\delta}^{d,0} (\rdown{s-}_K ),X_{\ell,m,K,\delta}^{d,(0,n,k,\ell)} (\rdown{s-}_K ),z)\nonumber\\
&\qquad\qquad
-\1_{\N}(\ell)
\eta_d(X_{\ell-1,m,K,\delta}^{d,0} (\rdown{s-}_K ),X_{\ell-1,m,K,\delta}^{d,(0,n,k,\ell)} (\rdown{s-}_K ),z)\Bigr]
\tilde{N}^{d,0}(dz,ds)\biggr \rVert^2\Biggr]\Biggr)^\frac{1}{2}
\nonumber\\
&= 
\sum_{\ell=0}^{n-1}
\frac{1}{\sqrt{m^{n-\ell}}}\Biggl(\E\Biggl[\biggl \lVert\int_{0}^{t}\int_{A^d_\delta}
\Bigl[\eta_d(X_{\ell,m,K,\delta}^{d,0} (\rdown{s-}_K ),X_{\ell,m,K,\delta}^{d,(0,n,k,\ell)} (\rdown{s-}_K ),z)\nonumber\\
&\qquad\qquad
-\1_{\N}(\ell)
\eta_d(X_{\ell-1,m,K,\delta}^{d,0} (\rdown{s-}_K ),X_{\ell-1,m,K,\delta}^{d,(0,n,k,\ell)} (\rdown{s-}_K ),z)\Bigr]
\tilde{N}^{d,0}(dz,ds)\biggr \rVert^2\Biggr]\Biggr)^\frac{1}{2}
\nonumber\\
&\leq \sum_{\ell=0}^{n-1}
\frac{1}{\sqrt{m^{n-\ell}}}\Biggl(\int_{0}^{t}\int_{\R^d}\E\Biggl[\biggl \lVert
\Bigl[\eta_d(X_{\ell,m,K,\delta}^{d,0} (\rdown{s-}_K ),X_{\ell,m,K,\delta}^{d,(0,n,k,\ell)} (\rdown{s-}_K ),z)\nonumber\\
&\qquad\qquad
-\1_{\N}(\ell)
\eta_d(X_{\ell-1,m,K,\delta}^{d,0} (\rdown{s-}_K ),X_{\ell-1,m,K,\delta}^{d,(0,n,k,\ell)} (\rdown{s-}_K ),z)\Bigr]
\biggr \rVert^2\Biggr]\nu^d(dz)\,ds\Biggr)^\frac{1}{2}
\nonumber\\
&\leq  \frac{\sqrt{T}c d^c}{\sqrt{m^n}}
+\sum_{\ell=1}^{n-1}\frac{1}{\sqrt{m^{n-\ell}}}
\left(\int_{0}^{t}2c\left\lVert X_{\ell,m,K,\delta}^{d,0} (\rdown{s-}_K ) 
-X_{\ell-1,m,K,\delta}^{d,0} (\rdown{s-}_K )
\right\rVert_{L^2(\P;\R^d)}^2 ds \right)^\frac{1}{2}.
\end{align}
The proof of Lemma~\ref{k06} is thus completed.
\end{proof}

\begin{lemma}[Conditional variance of the last sum] \label{s09}
In both cases (finite and infinite jump-activity) for all 
$d,n,m,K\in \N$, $\delta\in \Delta$, $t\in [0,T]$ we have
\begin{align}
&\Biggl \lVert\Biggl(\V_{\mathcal{G}^d_{n-1,m,K,\delta}}\Biggl[
\sum_{\ell=0}^{n-1}\frac{1}{m^{n-\ell}}\sum_{k=1}^{m^{n-\ell}}
\frac{1}{\mathcal{M}^d_\delta}
\sum_{i=1}^{\mathcal{M}^d_\delta}\int_{0}^{t}\biggl[\nonumber\\
&\qquad\qquad\qquad
\nu^d(A^d_\delta)
\Bigl[\eta_d(X_{\ell,m,K,\delta}^{d,0} (\rdown{s-}_K ),X_{\ell,m,K,\delta}^{d,(0,n,k,\ell)} (\rdown{s-}_K ),\rvV^{d,(0,n,k,\ell,i),\delta}_{\rdown{s}_K })\nonumber\\
&\qquad\qquad\qquad\qquad\qquad
-\1_\N(\ell)
\eta_d(X_{\ell-1,m,K,\delta}^{d,0} (\rdown{s-}_K ),X_{\ell-1,m,K,\delta}^{d,(0,n,k,\ell)} (\rdown{s-}_K ),\rvV^{d,(0,n,k,\ell,i),\delta}_{\rdown{s}_K })\Bigr]\nonumber\\
&\qquad\qquad\qquad
-\int_{A^d_\delta}
\Bigl[\eta_d(X_{\ell,m,K,\delta}^{d,0} (\rdown{s-}_K ),X_{\ell,m,K,\delta}^{d,(0,n,k,\ell)} (\rdown{s-}_K ),z)\nonumber\\
&\qquad\qquad\qquad\qquad
-\1_\N(\ell)
\eta_d(X_{\ell-1,m,K,\delta}^{d,0} (\rdown{s-}_K ),X_{\ell-1,m,K,\delta}^{d,(0,n,k,\ell)} (\rdown{s-}_K ),z)\,\nu^d(dz)\Bigr]\biggr]ds\Biggr]\Biggr)^\frac{1}{2}\Biggr \rVert_{L^2(\P;\R^d)}\nonumber\\
&\leq\frac{Tcd^c}{\sqrt{m^{n}}}+ \sum_{\ell=1}^{n-1}\frac{\sqrt{T}}{\sqrt{m^{n-\ell}}}
\left(\int_{0}^{t}2c\left\lVert X_{\ell,m,K,\delta}^{d,0} (\rdown{s-}_K ) 
-X_{\ell-1,m,K,\delta}^{d,0} (\rdown{s-}_K )
\right\rVert_{L^2(\P;\R^d)}^2 ds \right)^\frac{1}{2}.\label{s03}
\end{align}
\end{lemma}

\begin{proof}[Proof of Lemma~\ref{s09}]
The triangle inequality,
Bienaym\'e's identity, and the independence and distributional properties show that in both cases for all
$d,n,m,K\in \N$, $\delta\in \Delta$, $t\in [0,T]$ we have that 
\begin{align}
&\Biggl(\V_{\mathcal{G}^d_{n-1,m,K,\delta}}\Biggl[
\sum_{\ell=0}^{n-1}\frac{1}{m^{n-\ell}}\sum_{k=1}^{m^{n-\ell}}
\frac{1}{\mathcal{M}^d_\delta}
\sum_{i=1}^{\mathcal{M}^d_\delta}\int_{0}^{t}\biggl[\nonumber\\
&\qquad\qquad
\nu^d(A^d_\delta)
\Bigl[\eta_d(X_{\ell,m,K,\delta}^{d,0} (\rdown{s-}_K ),X_{\ell,m,K,\delta}^{d,(0,n,k,\ell)} (\rdown{s-}_K ),\rvV^{d,(0,n,k,\ell,i),\delta}_{\rdown{s}_K })\nonumber\\
&\qquad\qquad\qquad\qquad
-\1_\N(\ell)
\eta_d(X_{\ell-1,m,K,\delta}^{d,0} (\rdown{s-}_K ),X_{\ell-1,m,K,\delta}^{d,(0,n,k,\ell)} (\rdown{s-}_K ),\rvV^{d,(0,n,k,\ell,i),\delta}_{\rdown{s}_K })\Bigr]\nonumber\\
&\qquad\qquad
-\int_{A^d_\delta}
\Bigl[\eta_d(X_{\ell,m,K,\delta}^{d,0} (\rdown{s-}_K ),X_{\ell,m,K,\delta}^{d,(0,n,k,\ell)} (\rdown{s-}_K ),z)\nonumber\\
&\qquad\qquad\qquad\qquad
-\1_\N(\ell)
\eta_d(X_{\ell-1,m,K,\delta}^{d,0} (\rdown{s-}_K ),X_{\ell-1,m,K,\delta}^{d,(0,n,k,\ell)} (\rdown{s-}_K ),z)\,\nu^d(dz)\Bigr]\biggr]ds\Biggr]\Biggr)^{\frac{1}{2}}\nonumber\\
&\leq \sum_{\ell=0}^{n-1}\frac{1}{\sqrt{m^{n-\ell}}} \Biggl(\V_{\mathcal{G}^d_{n-1,m,K,\delta}}\Biggl[
\frac{1}{\mathcal{M}^d_\delta}
\sum_{i=1}^{\mathcal{M}^d_\delta}\int_{0}^{t}\biggl[\nonumber\\
&\qquad\qquad
\nu^d(A^d_\delta)
\Bigl[\eta_d(X_{\ell,m,K,\delta}^{d,0} (\rdown{s-}_K ),X_{\ell,m,K,\delta}^{d,(0,n,1,\ell)} (\rdown{s-}_K ),\rvV^{d,(0,n,1,\ell,i),\delta}_{\rdown{s}_K })\nonumber\\
&\qquad\qquad\qquad\qquad
-\1_\N(\ell)
\eta_d(X_{\ell-1,m,K,\delta}^{d,0} (\rdown{s-}_K ),X_{\ell-1,m,K,\delta}^{d,(0,n,1,\ell)} (\rdown{s-}_K ),\rvV^{d,(0,n,1,\ell,i),\delta}_{\rdown{s}_K })\Bigr]\nonumber\\
&\qquad\qquad
-\int_{A^d_\delta}
\Bigl[\eta_d(X_{\ell,m,K,\delta}^{d,0} (\rdown{s-}_K ),X_{\ell,m,K,\delta}^{d,(0,n,1,\ell)} (\rdown{s-}_K ),z)\nonumber\\
&\qquad\qquad\qquad\qquad
-\1_\N(\ell)
\eta_d(X_{\ell-1,m,K,\delta}^{d,0} (\rdown{s-}_K ),X_{\ell-1,m,K,\delta}^{d,(0,n,1,\ell)} (\rdown{s-}_K ),z)\,\nu^d(dz)\Bigr]\biggr]ds\Biggr]\Biggr)^\frac{1}{2}\nonumber\\
&\leq \sum_{\ell=0}^{n-1}\frac{1}{\sqrt{m^{n-\ell}}} \Biggl(\E_{\mathcal{G}^d_{n-1,m,K,\delta}}\Biggl[\Biggl \lVert
\frac{1}{\mathcal{M}^d_\delta}
\sum_{i=1}^{\mathcal{M}^d_\delta}\int_{0}^{t}\biggl[\nonumber\\
&\qquad\qquad
\nu^d(A^d_\delta)
\Bigl[\eta_d(X_{\ell,m,K,\delta}^{d,0} (\rdown{s-}_K ),X_{\ell,m,K,\delta}^{d,(0,n,1,\ell)} (\rdown{s-}_K ),\rvV^{d,(0,n,1,\ell,i),\delta}_{\rdown{s}_K })\nonumber\\
&\qquad\qquad\qquad\qquad
-\1_\N(\ell)
\eta_d(X_{\ell-1,m,K,\delta}^{d,0} (\rdown{s-}_K ),X_{\ell-1,m,K,\delta}^{d,(0,n,1,\ell)} (\rdown{s-}_K ),\rvV^{d,(0,n,1,\ell,i),\delta}_{\rdown{s}_K })\Bigr]\nonumber\\
&\qquad\qquad
-\int_{A^d_\delta}
\Bigl[\eta_d(X_{\ell,m,K,\delta}^{d,0} (\rdown{s-}_K ),X_{\ell,m,K,\delta}^{d,(0,n,1,\ell)} (\rdown{s-}_K ),z)\nonumber\\
&\qquad\qquad\qquad\qquad
-\1_\N(\ell)
\eta_d(X_{\ell-1,m,K,\delta}^{d,0} (\rdown{s-}_K ),X_{\ell-1,m,K,\delta}^{d,(0,n,1,\ell)} (\rdown{s-}_K ),z)\,\nu^d(dz)\Bigr]\biggr]ds\Biggr \rVert^2\Biggr]\Biggr)^\frac{1}{2}\nonumber\\
&\leq\sum_{\ell=0}^{n-1}\frac{1}{\sqrt{m^{n-\ell}}} \Biggl(\E\Biggl[\Biggl\lVert
\frac{1}{\mathcal{M}^d_\delta}
\sum_{i=1}^{\mathcal{M}^d_\delta}\int_{0}^{t}\biggl[
\nu^d(A^d_\delta)
\Bigl[\eta_d(\mathbf{X}(\rdown{s-}_K),X_{\ell,m,K,\delta}^{d,(0,n,1,\ell)} (\rdown{s-}_K ),\rvV^{d,(0,n,1,\ell,i),\delta}_{\rdown{s}_K })\nonumber\\
&\qquad\qquad\qquad\qquad\qquad\qquad\qquad\qquad
-\1_\N(\ell)
\eta_d(\mathbf{Y}(\rdown{s-}_K),X_{\ell-1,m,K,\delta}^{d,(0,n,1,\ell)} (\rdown{s-}_K ),\rvV^{d,(0,n,1,\ell,i),\delta}_{\rdown{s}_K })\Bigr]\nonumber\\
&\qquad\qquad
-\int_{A^d_\delta}
\Bigl[\eta_d(\mathbf{X}(\rdown{s-}_K),X_{\ell,m,K,\delta}^{d,(0,n,1,\ell)} (\rdown{s-}_K ),z)\nonumber\\
&\qquad\qquad\qquad\qquad
-\1_\N(\ell)
\eta_d(\mathbf{Y}(\rdown{s-}_K),X_{\ell-1,m,K,\delta}^{d,(0,n,1,\ell)} (\rdown{s-}_K ),z)\,\nu^d(dz)\Bigr]\biggr]ds\Biggr\rVert^2\Biggr]\Biggr|_{\tiny \begin{aligned}
\mathbf{X}&=X^{d,0}_{\ell,m,K,\delta}\\[-0.12cm]\mathbf{Y}&=X^{d,0}_{\ell-1,m,K,\delta}
\end{aligned}}\Biggr)^\frac{1}{2}.\label{s10}
\end{align}
We estimate the expectation here. The independence and distributional properties, the tower property,  the disintegration theorem, and \eqref{f09} (see also Remark~\ref{r01b})
show that in both cases we have for all $d,n,m,K\in \N$, $\delta\in \Delta$, $t\in [0,T]$, $\ell\in [0,n-1]\cap\Z$, $x,y\in \R^d$ that 
\begin{align}
&
\E\Biggl[\Biggl\lVert
\frac{1}{\mathcal{M}^d_\delta}
\sum_{i=1}^{\mathcal{M}^d_\delta}\int_{0}^{t}\biggl[
\nu^d(A^d_\delta)
\Bigl[\eta_d(x,X_{\ell,m,K,\delta}^{d,(0,n,1,\ell)} (\rdown{s-}_K ),\rvV^{d,(0,n,1,\ell,i),\delta}_{\rdown{s}_K })\nonumber\\
&\qquad\qquad\qquad\qquad\qquad\qquad\qquad\qquad
-\1_\N(\ell)
\eta_d(y,X_{\ell-1,m,K,\delta}^{d,(0,n,1,\ell)} (\rdown{s-}_K ),\rvV^{d,(0,n,1,\ell,i),\delta}_{\rdown{s}_K })\Bigr]\nonumber\\
&\qquad\qquad\qquad
-\int_{A^d_\delta}
\Bigl[\eta_d(x,X_{\ell,m,K,\delta}^{d,(0,n,1,\ell)} (\rdown{s-}_K ),z)\nonumber\\
&\qquad\qquad\qquad\qquad\qquad
-\1_\N(\ell)
\eta_d(y,X_{\ell-1,m,K,\delta}^{d,(0,n,1,\ell)} (\rdown{s-}_K ),z)\,\nu^d(dz)\Bigr]\biggr]ds\Biggr\rVert^2\Biggr]\nonumber\\
&=
\E\Biggl[\Biggl\lVert
\frac{1}{\mathcal{M}^d_\delta}
\sum_{i=1}^{\mathcal{M}^d_\delta}\int_{0}^{t}\biggl[
\nu^d(A^d_\delta)
\Bigl[\eta_d(x,X_{\ell,m,K,\delta}^{d,0} (\rdown{s-}_K ),\rvV^{d,(0,n,1,\ell,i),\delta}_{\rdown{s}_K })\nonumber\\
&\qquad\qquad\qquad\qquad\qquad\qquad\qquad\qquad
-\1_\N(\ell)
\eta_d(y,X_{\ell-1,m,K,\delta}^{d,0} (\rdown{s-}_K ),\rvV^{d,(0,n,1,\ell,i),\delta}_{\rdown{s}_K })\Bigr]\nonumber\\
&\qquad\qquad\qquad
-\int_{A^d_\delta}
\Bigl[\eta_d(x,X_{\ell,m,K,\delta}^{d,0} (\rdown{s-}_K ),z)\nonumber\\
&\qquad\qquad\qquad\qquad\qquad
-\1_\N(\ell)
\eta_d(y,X_{\ell-1,m,K,\delta}^{d,0} (\rdown{s-}_K ),z)\,\nu^d(dz)\Bigr]\biggr]ds\Biggr\rVert^2\Biggr]\nonumber\\
&=
\E\E_{\mathcal{G}_{n-1,m,K,\delta}^{d}}\Biggl[\Biggl\lVert
\frac{1}{\mathcal{M}^d_\delta}
\sum_{i=1}^{\mathcal{M}^d_\delta}\int_{0}^{t}\biggl[
\nu^d(A^d_\delta)
\Bigl[\eta_d(x,X_{\ell,m,K,\delta}^{d,0} (\rdown{s-}_K ),\rvV^{d,(0,n,1,\ell,i),\delta}_{\rdown{s}_K })\nonumber\\
&\qquad\qquad\qquad\qquad\qquad\qquad\qquad\qquad\qquad\qquad
-\1_\N(\ell)
\eta_d(y,X_{\ell-1,m,K,\delta}^{d,0} (\rdown{s-}_K ),\rvV^{d,(0,n,1,\ell,i),\delta}_{\rdown{s}_K })\Bigr]\nonumber\\
&\qquad\qquad\qquad
-\nu^d(A^d_\delta)\E_{\mathcal{G}_{n-1,m,K,\delta}^d}\!\Bigl[\eta_d(x,X_{\ell,m,K,\delta}^{d,0} (\rdown{s-}_K ),\rvV^{d,(0,n,1,\ell,i),\delta}_{\rdown{s}_K })\nonumber\\
&\qquad\qquad\qquad\qquad\qquad\qquad\qquad\qquad
-\1_\N(\ell)
\eta_d(y,X_{\ell-1,m,K,\delta}^{d,0} (\rdown{s-}_K ),\rvV^{d,(0,n,1,\ell,i),\delta}_{\rdown{s}_K })\Bigr]
\biggr]ds\Biggr\rVert^2\Biggr]\nonumber\\
&=\E 
\V_{\mathcal{G}_{n-1,m,K,\delta}^{d}}\Biggl[
\frac{1}{\mathcal{M}^d_\delta}
\sum_{i=1}^{\mathcal{M}^d_\delta}\int_{0}^{t}\biggl[
\nu^d(A^d_\delta)
\Bigl[\eta_d(x,X_{\ell,m,K,\delta}^{d,0} (\rdown{s-}_K ),\rvV^{d,(0,n,1,\ell,i),\delta}_{\rdown{s}_K })\nonumber\\
&\qquad\qquad\qquad\qquad\qquad\qquad\qquad\qquad\qquad\qquad
-\1_\N(\ell)
\eta_d(y,X_{\ell-1,m,K,\delta}^{d,0} (\rdown{s-}_K ),\rvV^{d,(0,n,1,\ell,i),\delta}_{\rdown{s}_K })\Bigr]
\biggr]ds\Biggr]\nonumber\\
&=
\frac{1}{\mathcal{M}^d_\delta}
\E 
\V_{\mathcal{G}_{n-1,m,K,\delta}^{d}}\Biggl[
\int_{0}^{t}\biggl[
\nu^d(A^d_\delta)
\Bigl[\eta_d(x,X_{\ell,m,K,\delta}^{d,0} (\rdown{s-}_K ),\rvV^{d,(0,n,1,\ell,1),\delta}_{\rdown{s}_K })\nonumber\\
&\qquad\qquad\qquad\qquad\qquad\qquad\qquad\qquad\qquad
-\1_\N(\ell)
\eta_d(y,X_{\ell-1,m,K,\delta}^{d,0} (\rdown{s-}_K ),\rvV^{d,(0,n,1,\ell,1),\delta}_{\rdown{s}_K })\Bigr]
\biggr]ds\Biggr].\label{s08}
\end{align}
We estimate the conditional variance here. Note that the integral here is in fact a discrete sum. Then 
Bienaym\'e's identity, the independence and distributional properties, the disintegration theorem, and \eqref{f09} (see also Remark~\ref{r01b})
show that in both cases we have for all $d,n,m,K\in \N$, $\delta\in \Delta$,
 $\ell\in [0,n-1]\cap\Z$, 
$J\in [1,K]\cap\Z$, $t\in (\frac{(J-1)T}{K}, \frac{JT}{K}]$,
$s_0=0,s_1=\frac{T}{K},\ldots,s_{J-1}=\frac{(J-1)T}{K}$, $t=s_J$, $x,y\in \R^d$
that 
\begin{align}
&\V_{\mathcal{G}_{n-1,m,K,\delta}^{d}}\Biggl[
\int_{0}^{t}\biggl[
\nu^d(A^d_\delta)
\Bigl[\eta_d(x,X_{\ell,m,K,\delta}^{d,0} (\rdown{s-}_K ),\rvV^{d,(0,n,1,\ell,1),\delta}_{\rdown{s}_K })\nonumber\\
&\qquad\qquad\qquad\qquad\qquad\qquad
-\1_\N(\ell)
\eta_d(y,X_{\ell-1,m,K,\delta}^{d,0} (\rdown{s-}_K ),\rvV^{d,(0,n,1,\ell,1),\delta}_{\rdown{s}_K })\Bigr]
\biggr]ds\Biggr]\nonumber\\
&=\V_{\mathcal{G}_{n-1,m,K,\delta}^{d}}\Biggl[
\sum_{j=0}^{J-1}\biggl[
\nu^d(A^d_\delta)
\Bigl[\eta_d(x,X_{\ell,m,K,\delta}^{d,0} (s_j ),\rvV^{d,(0,n,1,\ell,1),\delta}_{s_j})\nonumber\\
&\qquad\qquad\qquad\qquad\qquad\qquad
-\1_\N(\ell)
\eta_d(y,X_{\ell-1,m,K,\delta}^{d,0} (s_j),\rvV^{d,(0,n,1,\ell,1),\delta}_{s_j})\Bigr]
\biggr](s_{j+1}-s_j)\Biggr]\nonumber\\
&=
\sum_{j=0}^{J-1}\V_{\mathcal{G}_{n-1,m,K,\delta}^{d}}
\biggl[
\nu^d(A^d_\delta)
\Bigl[\eta_d(x,X_{\ell,m,K,\delta}^{d,0} (s_j ),\rvV^{d,(0,n,1,\ell,1),\delta}_{s_j})\nonumber\\
&\qquad\qquad\qquad\qquad\qquad\qquad
-\1_\N(\ell)
\eta_d(y,X_{\ell-1,m,K,\delta}^{d,0} (s_j),\rvV^{d,(0,n,1,\ell,1),\delta}_{s_j})\Bigr]
\biggr](s_{j+1}-s_j)^2\nonumber\\
&\leq \frac{T}{K}
\sum_{j=0}^{J-1}\V_{\mathcal{G}_{n-1,m,K,\delta}^{d}}
\biggl[
\nu^d(A^d_\delta)
\Bigl[\eta_d(x,X_{\ell,m,K,\delta}^{d,0} (s_j ),\rvV^{d,(0,n,1,\ell,1),\delta}_{s_j})\nonumber\\
&\qquad\qquad\qquad\qquad\qquad\qquad
-\1_\N(\ell)
\eta_d(y,X_{\ell-1,m,K,\delta}^{d,0} (s_j),\rvV^{d,(0,n,1,\ell,1),\delta}_{s_j})\Bigr]
\biggr](s_{j+1}-s_j)\nonumber\\
&\leq \frac{T}{K}\int_{0}^{t}\V_{\mathcal{G}_{n-1,m,K,\delta}^{d}}
\biggl[
\nu^d(A^d_\delta)
\Bigl[\eta_d(x,X_{\ell,m,K,\delta}^{d,0} (\rdown{s-}_K ),\rvV^{d,(0,n,1,\ell,1),\delta}_{\rdown{s}_K })\nonumber\\
&\qquad\qquad\qquad\qquad\qquad\qquad
-\1_\N(\ell)
\eta_d(y,X_{\ell-1,m,K,\delta}^{d,0} (\rdown{s-}_K ),\rvV^{d,(0,n,1,\ell,1),\delta}_{\rdown{s}_K })\Bigr]
\biggr]ds\nonumber\\
&\leq \frac{T}{K}\int_{0}^{t}\E_{\mathcal{G}_{n-1,m,K,\delta}^{d}}\Biggl[
\biggl\lVert
\nu^d(A^d_\delta)
\Bigl[\eta_d(x,X_{\ell,m,K,\delta}^{d,0} (\rdown{s-}_K ),\rvV^{d,(0,n,1,\ell,1),\delta}_{\rdown{s}_K })\nonumber\\
&\qquad\qquad\qquad\qquad\qquad\qquad
-\1_\N(\ell)
\eta_d(y,X_{\ell-1,m,K,\delta}^{d,0} (\rdown{s-}_K ),\rvV^{d,(0,n,1,\ell,1),\delta}_{\rdown{s}_K })\Bigr]
\biggr\rVert^2\Biggr]ds\nonumber\\
&=
\frac{T}{K}
\nu^d(A^d_\delta)
\int_{0}^{t}\int_{A^d_\delta}\E_{\mathcal{G}_{n-1,m,K,\delta}^{d}}\Biggl[
\biggl\lVert
\eta_d(x,X_{\ell,m,K,\delta}^{d,0} (\rdown{s-}_K ),z)\nonumber\\
&\qquad\qquad\qquad\qquad\qquad\qquad
-\1_\N(\ell)
\eta_d(y,X_{\ell-1,m,K,\delta}^{d,0} (\rdown{s-}_K ),z)
\biggr\rVert^2\Biggr]\nu^d(dz)\,ds.
\end{align}
Plugging this into \eqref{s08}, using the tower property, and recalling \eqref{s05b}
we obtain in both cases for all $d,n,m,K\in \N$,  $\delta\in \Delta$,
$\ell\in [0,n-1]\cap\Z$, 
$x,y\in \R^d$
that
\begin{align}
&
\E\Biggl[\Biggl\lVert
\frac{1}{\mathcal{M}^d_\delta}
\sum_{i=1}^{\mathcal{M}^d_\delta}\int_{0}^{t}\biggl[
\nu^d(A^d_\delta)
\Bigl[\eta_d(x,X_{\ell,m,K,\delta}^{d,(0,n,1,\ell)} (\rdown{s-}_K ),\rvV^{d,(0,n,1,\ell,i),\delta}_{\rdown{s}_K })\nonumber\\
&\qquad\qquad\qquad\qquad\qquad\qquad\qquad\qquad
-\1_\N(\ell)
\eta_d(y,X_{\ell-1,m,K,\delta}^{d,(0,n,1,\ell)} (\rdown{s-}_K ),\rvV^{d,(0,n,1,\ell,i),\delta}_{\rdown{s}_K })\Bigr]\nonumber\\
&\qquad\qquad\qquad
-\int_{A^d_\delta}
\Bigl[\eta_d(x,X_{\ell,m,K,\delta}^{d,(0,n,1,\ell)} (\rdown{s-}_K ),z)\nonumber\\
&\qquad\qquad\qquad\qquad\qquad\qquad\qquad\qquad
-\1_\N(\ell)
\eta_d(y,X_{\ell-1,m,K,\delta}^{d,(0,n,1,\ell)} (\rdown{s-}_K ),z)\,\nu^d(dz)\Bigr]\biggr]ds\Biggr\rVert^2\Biggr]\nonumber\\
&\leq \frac{1}{\mathcal{M}^d_\delta}
\frac{T}{K}
\nu^d(A^d_\delta)
\int_{0}^{t}\int_{A^d_\delta}\E \E_{\mathcal{G}_{n-1,m,K,\delta}^{d}}\Biggl[
\biggl\lVert
\eta_d(x,X_{\ell,m,K,\delta}^{d,0} (\rdown{s-}_K ),z)\nonumber\\
&\qquad\qquad\qquad\qquad\qquad\qquad\qquad\qquad
-\1_\N(\ell)
\eta_d(y,X_{\ell-1,m,K,\delta}^{d,0} (\rdown{s-}_K ),z)
\biggr\rVert^2\Biggr]\nu^d(dz)\,ds\nonumber\\
&\leq T\int_{0}^{t}\int_{A^d_\delta}\E \Biggl[
\biggl\lVert
\eta_d(x,X_{\ell,m,K,\delta}^{d,0} (\rdown{s-}_K ),z)\nonumber\\
&\qquad\qquad\qquad\qquad\qquad\qquad
-\1_\N(\ell)
\eta_d(y,X_{\ell-1,m,K,\delta}^{d,0} (\rdown{s-}_K ),z)
\biggr\rVert^2\Biggr]\nu^d(dz)\,ds.\label{k16}
\end{align}
Plugging this into \eqref{s10}, taking  $\lVert\cdot\rVert_{L^2(\P;\R^d)}$, and using \eqref{e11b} we obtain in both cases for all 
$d,n,m,K\in \N$, $\delta\in \Delta$
that
\begin{align}
&\Biggl \lVert\Biggl(\V_{\mathcal{G}^d_{n-1,m,K,\delta}}\Biggl[
\sum_{\ell=0}^{n-1}\frac{1}{m^{n-\ell}}\sum_{k=1}^{m^{n-\ell}}
\frac{1}{\mathcal{M}^d_\delta}
\sum_{i=1}^{\mathcal{M}^d_\delta}\int_{0}^{t}\biggl[\nonumber\\
&\qquad\qquad
\nu^d(A^d_\delta)
\Bigl[\eta_d(X_{\ell,m,K,\delta}^{d,0} (\rdown{s-}_K ),X_{\ell,m,K,\delta}^{d,(0,n,k,\ell)} (\rdown{s-}_K ),\rvV^{d,(0,n,k,\ell,i),\delta}_{\rdown{s}_K })\nonumber\\
&\qquad\qquad\qquad\qquad\qquad
-\1_\N(\ell)
\eta_d(X_{\ell-1,m,K,\delta}^{d,0} (\rdown{s-}_K ),X_{\ell-1,m,K,\delta}^{d,(0,n,k,\ell)} (\rdown{s-}_K ),\rvV^{d,(0,n,k,\ell,i),\delta}_{\rdown{s}_K })\Bigr]\nonumber\\
&\qquad\qquad
-\int_{A^d_\delta}
\Bigl[\eta_d(X_{\ell,m,K,\delta}^{d,0} (\rdown{s-}_K ),X_{\ell,m,K,\delta}^{d,(0,n,k,\ell)} (\rdown{s-}_K ),z)\nonumber\\
&\qquad\qquad\qquad\qquad
-\1_\N(\ell)
\eta_d(X_{\ell-1,m,K,\delta}^{d,0} (\rdown{s-}_K ),X_{\ell-1,m,K,\delta}^{d,(0,n,k,\ell)} (\rdown{s-}_K ),z)\,\nu^d(dz)\Bigr]\biggr]ds\Biggr]\Biggr)^\frac{1}{2}\Biggr \rVert_{L^2(\P;\R^d)}\nonumber\\
&\leq \sum_{\ell=0}^{n-1}\frac{\sqrt{T}}{\sqrt{m^{n-\ell}}}\Biggl \lVert\Biggl(\int_{0}^{t}\int_{A^d_\delta}\E \Biggl[
\biggl\lVert
\eta_d(x,X_{\ell,m,K,\delta}^{d,0} (\rdown{s-}_K ),z)\nonumber\\
&\qquad\qquad\qquad\qquad
-\1_\N(\ell)
\eta_d(y,X_{\ell-1,m,K,\delta}^{d,0} (\rdown{s-}_K ),z)
\biggr\rVert^2\Biggr]\Biggr|_{\substack{x=X^{d,0}_{\ell,m,K,\delta}(\rdown{s-}_K)\\y=X^{d,0}_{\ell-1,m,K,\delta}(\rdown{s-}_K)}}\nu^d(dz)\,ds\Biggr)^\frac{1}{2}\Biggr \rVert_{L^2(\P;\R^d)}\nonumber\\
&=
 \sum_{\ell=0}^{n-1}\frac{\sqrt{T}}{\sqrt{m^{n-\ell}}}\Biggl(\int_{0}^{t}\int_{A^d_\delta}\E \Biggl[\E \Biggl[
\biggl\lVert
\eta_d(x,X_{\ell,m,K,\delta}^{d,0} (\rdown{s-}_K ),z)\nonumber\\
&\qquad\qquad\qquad\qquad
-\1_\N(\ell)
\eta_d(y,X_{\ell-1,m,K,\delta}^{d,0} (\rdown{s-}_K ),z)
\biggr\rVert^2\Biggr]\Biggr|_{\substack{x=X^{d,0}_{\ell,m,K,\delta}(\rdown{s-}_K)\\y=X^{d,0}_{\ell-1,m,K,\delta}(\rdown{s-}_K)}}\Biggr]\nu^d(dz)\,ds\Biggr)^\frac{1}{2}\nonumber\\
&\leq\frac{Tcd^c}{\sqrt{m^{n}}}+ \sum_{\ell=1}^{n-1}\frac{\sqrt{T}}{\sqrt{m^{n-\ell}}}
\left(\int_{0}^{t}2c\left\lVert X_{\ell,m,K,\delta}^{d,0} (\rdown{s-}_K ) 
-X_{\ell-1,m,K,\delta}^{d,0} (\rdown{s-}_K )
\right\rVert_{L^2(\P;\R^d)}^2 ds \right)^\frac{1}{2}.
\end{align}
This completes the proof of Lemma~\ref{s09}.
\end{proof}
\begin{corollary}[Statistical error]\label{k09}In both cases (finite and infinite jump-activity) there exists $\kappa\in (0,\infty)$ such that for all 
$d,n,m,K\in \N$, $\delta\in \Delta$, $t\in [0,T]$ we have
\begin{align}
&
\left\lVert 
X^{d,0}_{n,m,K,\delta}(t)-
\E _{\mathcal{G}^{d}_{n-1,m,K,\delta} }\!\left[
X^{d,0}_{n,m,K,\delta}(t)\right]
\right\rVert_{L^2(\P;\R^d)}\nonumber\\
&
\leq \frac{\kappa d^\kappa }{\sqrt{m^n}}
+\sum_{\ell=0}^{n-1}\frac{(2-\1_{\{n-1\}}(\ell))\kappa}{\sqrt{m^{n-\ell-1}}}
\left(\int_{0}^{t}\left\lVert X_{\ell,m,K,\delta}^{d,0} (\rdown{s-}_K ) 
-X_{K,\delta}^d(\rdown{s-}_K ) 
\right\rVert_{L^2(\P;\R^d)}^2 ds\right)^\frac{1}{2}.
\end{align}
\end{corollary}
\begin{proof}[Proof of Corollary~\ref{k09}]The triangle inequality and Lemmas~\ref{k03}--\ref{s09} prove that in both cases there exists $\kappa\in (0,\infty)$ such that  for all 
$d,n,m,K\in \N$, $\delta\in \Delta$, $t\in [0,T]$ we have that
\begin{align}
&
\left\lVert 
X^{d,0}_{n,m,K,\delta}(t)-
\E _{\mathcal{G}^{d}_{n-1,m,K,\delta} }\!\left[
X^{d,0}_{n,m,K,\delta}(t)\right]
\right\rVert_{L^2(\P;\R^d)}\nonumber\\
&
=\left(\E \!\left[
\left\lVert 
X^{d,0}_{n,m,K,\delta}(t)-
\E _{\mathcal{G}^{d}_{n-1,m,K,\delta} }\!\left[
X^{d,0}_{n,m,K,\delta}(t)\right]
\right\rVert^2\right]\right)^\frac{1}{2}\nonumber\\
&=\left(\E \E_{\mathcal{G}^{d}_{n-1,m,K,\delta} } \!\left[
\left\lVert 
X^{d,0}_{n,m,K,\delta}(t)-
\E _{\mathcal{G}^{d}_{n-1,m,K,\delta} }\!\left[
X^{d,0}_{n,m,K,\delta}(t)\right]
\right\rVert^2
\right]\right)^\frac{1}{2}\nonumber\\
&=\left(\E\!\left[\var_{\mathcal{G}^{d}_{n-1,m,K,\delta}\! }\left[X^{d,0}_{n,m,K,\delta}(t)\right]\right]\right)^\frac{1}{2}\nonumber\\
&=
\left\lVert
\left(\var_{\mathcal{G}^d_{n-1,m,K,\delta}}\!\left
[
X^{d,0}_{n,m,K,\delta}(t)\right]\right)^\frac{1}{2} 
\right\rVert_{L^2(\P;\R)}
\nonumber\\
&\leq  \frac{\kappa d^\kappa}{\sqrt{m^n}}
+\sum_{\ell=1}^{n-1}\frac{\kappa}{\sqrt{m^{n-\ell}}}
\left(\int_{0}^{t}\left\lVert X_{\ell,m,K,\delta}^{d,0} (\rdown{s-}_K ) 
-X_{\ell-1,m,K,\delta}^{d,0} (\rdown{s-}_K )
\right\rVert_{L^2(\P;\R^d)}^2 ds \right)^\frac{1}{2}\nonumber\\
&\leq \frac{\kappa d^\kappa}{\sqrt{m^n}}
+\sum_{\ell=1}^{n-1}\frac{\kappa}{\sqrt{m^{n-\ell}}}\sum_{j=\ell-1}^{\ell}
\left(\int_{0}^{t}\left\lVert X_{j,m,K,\delta}^{d,0} (\rdown{s-}_K ) 
-X_{K,\delta}^{d} (\rdown{s-}_K )
\right\rVert_{L^2(\P;\R^d)}^2 ds \right)^\frac{1}{2}\nonumber\\
&
\leq  \frac{\kappa d^\kappa }{\sqrt{m^n}}
+\sum_{\ell=0}^{n-1}\frac{(2-\1_{\{n-1\}}(\ell))\kappa}{\sqrt{m^{n-\ell-1}}}
\left(\int_{0}^{t}\left\lVert X_{\ell,m,K,\delta}^{d,0} (\rdown{s-}_K ) 
-X_{K,\delta}^d(\rdown{s-}_K )
\right\rVert_{L^2(\P;\R^d)}^2 ds\right)^\frac{1}{2}.
\end{align}
This completes the proof of Corollary~\ref{k09}.
\end{proof}
\subsection{Bias}We will estimate the bias of the MLP approximations.

\begin{lemma}[Conditional expectation of the approximations]
\label{k02}
In both cases we have for all
$d,n,m,K\in \N$, $\delta\in \Delta$, $t\in [0,T]$ that
\begin{align}
&\E _{\mathcal{G}^d_{n-1,m,K,\delta}}\!\left[X^{d,0}_{n,m,K,\delta}(t)\right]\nonumber\\
&= \xi_d+\int_{0}^{t}
\E\! \left[\mu_d (x, X^{d,0}_{n-1,m,K,\delta}(\rdown{s-}_K) )\right]\!|_{x=X^{d,0}_{n-1,m,K,\delta}(\rdown{s-}_K)}\,ds\nonumber\\
&\quad 
+\int_{0}^{t}
\E\! \left[\sigma_d (x, X^{d,0}_{n-1,m,K,\delta}(\rdown{s-}_K) )\right]\!|_{x=X^{d,0}_{n-1,m,K,\delta}(\rdown{s-}_K)}\,dW^{d,0}(s)\nonumber
\\
&\quad +\int_{0}^{t}\int_{A^d_\delta}
\E\! \left[\eta_d (x, X^{d,0}_{n-1,m,K,\delta}(\rdown{s-}_K),z )\right]\!|_{x=X^{d,0}_{n-1,m,K,\delta}(\rdown{s-}_K)}\,\tilde{N}^{d,0}(dz,ds).
\end{align}
\end{lemma}
\begin{proof}[Proof of Lemma~\ref{k02}]
First, the disintegration theorem, the independence and distributional properties, \eqref{u01b}, and \eqref{f09} (see also Remark~\ref{r01b}) imply 
for all
$d,n,m,K\in \N$,  $\delta\in \Delta$, $t\in [0,T]$ 
that
\begin{align}
&t\E_{\mathcal{G}^d_{n-1,m,K,\delta}}\Bigl[
\mu_d (X_{\ell,m,K,\delta}^{d,0} (\rdown{t\unif^{(0,n,k,\ell)}-}_K ),X_{\ell,m,K,\delta}^{d,(0,n,k,\ell)} (\rdown{t\unif^{(0,n,k,\ell)}-}_K ))\nonumber\\
&\qquad\qquad\qquad\qquad
-\1_{\N}(\ell)
\mu_d (X_{\ell-1,m,K,\delta}^{d,0} (\rdown{t\unif^{(0,n,k,\ell)}-}_K ),X_{\ell-1,m,K,\delta}^{d,(0,n,k,\ell)} (\rdown{t\unif^{(0,n,k,\ell)}-}_K ))
\Bigr]\nonumber\\
&=\int_{0}^{t}\biggl(\E\Bigl[
\mu_d ( x,X^{d,0}_{\ell,m,K,\delta}(\rdown{s-}_K))
\Bigr]|_{x=X^{d,0}_{\ell,m,K,\delta}(\rdown{s-}_K)}\nonumber\\
&\qquad\qquad
-\1_\N(\ell)
\E\Bigl[
\mu_d ( x,X^{d,0}_{\ell-1,m,K,\delta}(\rdown{s-}_K))
\Bigr]|_{x=X^{d,0}_{\ell-1,m,K,\delta}(\rdown{s-}_K)}
\biggr)ds,
\end{align}
\begin{align}
&
\E_{\mathcal{G}^d_{n-1,m,K,\delta}}\biggl[\int_{0}^{t}
\Bigl[
\sigma_d(X_{\ell,m,K,\delta}^{d,0} (\rdown{s-}_K ),X_{\ell,m,K,\delta}^{d,(0,n,k,\ell)} (\rdown{s-}_K ))\nonumber\\
&\qquad\qquad\qquad\qquad
-\1_{\N}(\ell)
\sigma_d(X_{\ell-1,m,K,\delta}^{d,0} (\rdown{s-}_K ),X_{\ell-1,m,K,\delta}^{d,(0,n,k,\ell)} (\rdown{s-}_K ))\Bigr]
dW^{d,0}(s)\biggr]
\nonumber
\\
&=\int_{0}^{t}\biggl(\E\Bigl[
\sigma_d ( x,X^{d,0}_{\ell,m,K,\delta}(\rdown{s-}_K))
\Bigr]|_{x=X^{d,0}_{\ell,m,K,\delta}(\rdown{s-}_K)}\nonumber\\
&\qquad\qquad\qquad\qquad
-\1_\N(\ell)
\E\Bigl[
\sigma_d ( x,X^{d,0}_{\ell-1,m,K,\delta}(\rdown{s-}_K))
\Bigr]|_{x=X^{d,0}_{\ell-1,m,K,\delta}(\rdown{s-}_K)}
\biggr)dW^{d,0}(s),
\end{align} 
\begin{align}
&\E_{\mathcal{G}^d_{n-1,m,K,\delta}}\Biggl[\int_{0}^{t}\int_{A^d_\delta}
\Bigl[\eta_d(X_{\ell,m,K,\delta}^{d,0} (\rdown{s-}_K ),X_{\ell,m,K,\delta}^{d,(0,n,k,\ell)} (\rdown{s-}_K ),z)\nonumber\\
&\qquad\qquad\qquad\qquad\qquad
-\1_{\N}(\ell)
\eta_d(X_{\ell-1,m,K,\delta}^{d,0} (\rdown{s-}_K ),X_{\ell-1,m,K,\delta}^{d,(0,n,k,\ell)} (\rdown{s-}_K ),z)\Bigr]
\tilde{N}^{d,0}(dz,ds)\Biggr]\nonumber\\
&=\int_{0}^{t}\int_{A^d_\delta}
\biggl(\E\Bigl[
\eta_d ( x,X^{d,0}_{\ell,m,K,\delta}(\rdown{s-}_K),z)
\Bigr]|_{x=X^{d,0}_{\ell,m,K,\delta}(\rdown{s-}_K)}\nonumber\\
&\qquad\qquad\qquad
-\1_\N(\ell)
\E\Bigl[
\eta_d ( x,X^{d,0}_{\ell-1,m,K,\delta}(\rdown{s-}_K),z)
\Bigr]|_{x=X^{d,0}_{\ell-1,m,K,\delta}(\rdown{s-}_K)}
\biggr)\tilde{N}^{d,0}(dz,ds),
\end{align}
and
\begin{align}
&\E_{\mathcal{G}^d_{n-1,m,K,\delta}}\Biggl[\int_{0}^{t}\biggl[
\nu^d(A^d_\delta)
\Bigl[\eta_d(X_{\ell,m,K,\delta}^{d,0} (\rdown{s-}_K ),X_{\ell,m,K,\delta}^{d,(0,n,k,\ell)} (\rdown{s-}_K ),\rvV^{d,(0,n,k,\ell,i),\delta}_{\rdown{s}_K })\nonumber\\
&\qquad\qquad\qquad\qquad\qquad
-\1_\N(\ell)
\eta_d(X_{\ell-1,m,K,\delta}^{d,0} (\rdown{s-}_K ),X_{\ell-1,m,K,\delta}^{d,(0,n,k,\ell)} (\rdown{s-}_K ),\rvV^{d,(0,n,k,\ell,i),\delta}_{\rdown{s}_K })\Bigr]\nonumber\\
&\qquad\qquad
-\int_{A^d_\delta}
\Bigl[\eta_d(X_{\ell,m,K,\delta}^{d,0} (\rdown{s-}_K ),X_{\ell,m,K,\delta}^{d,(0,n,k,\ell)} (\rdown{s-}_K ),z)\nonumber\\
&\qquad\qquad\qquad\qquad\qquad
-\1_\N(\ell)
\eta_d(X_{\ell-1,m,K,\delta}^{d,0} (\rdown{s-}_K ),X_{\ell-1,m,K,\delta}^{d,(0,n,k,\ell)} (\rdown{s-}_K ),z)\Bigr]\nu^d(dz)\biggr]\,ds\Biggr]=0.
\end{align}
Adding these equations and applying Lemma~\ref{s04} complete the proof of Lemma~\ref{k02}.
\end{proof}

\begin{lemma}[Bias]\label{s11}
For every $d,K\in \N$, $\delta\in [0,1]$ let $X^d_{K,\delta}$ be the process in Lemma~\ref{a14}.  Then  there exists $\kappa\in (0,\infty)$ such that in both cases for all
$d,n,m,K\in \N$, 
$\delta\in \Delta $,
 $t\in [0,T]$ we have
\begin{align}
&\left\lVert
\E_{\mathcal{G}^d_{n-1,m,K,\delta}} \!
\left[X^{d,0}_{n,m,K,\delta}(t)\right]-X_{K,\delta}^d(t)\right\rVert_{L^2(\P;\R^d)}\nonumber\\
&\leq \kappa \left(\int_{0}^{t}\left\lVert X^{d,0}_{n-1,m,K,\delta}(\rdown{s-}_K)-X_{K,\delta}^d(\rdown{s-}_K)\right\rVert_{L^2(\P;\R^d)}^2ds\right)^\frac{1}{2}.
\end{align}
\end{lemma}
\begin{proof}[Proof of Lemma~\ref{s11}]First, Jensen's inequality and Lemma~\ref{a05f} show that in both cases there exists $\kappa\in(0,\infty)$ such that for all
$d,n,m,K\in \N$,  $t\in [0,T]$, $\delta\in\Delta$ we have that
\begin{align}
&
\biggl\lVert
\int_{0}^{t}\Bigl(
\E\! \left[\mu_d (x, X^{d,0}_{n-1,m,K,\delta}(\rdown{s-}_K) )\right]\!|_{x=X^{d,0}_{n-1,m,K,\delta}(\rdown{s-}_K)}\nonumber\\
&\qquad\qquad
-
\E\! \left[\mu_d (x, X^d_{K,\delta}(\rdown{s-}_K) )\right]\!|_{x=X^d_{K,\delta}(\rdown{s-}_K)}\Bigr)ds\biggr\rVert_{L^2(\P;\R^d)}\nonumber\\
&\leq \kappa
\biggl(
\int_{0}^{t}\Bigl\lVert
\E\! \left[\mu_d (x, X^{d,0}_{n-1,m,K,\delta}(\rdown{s-}_K) )\right]\!|_{x=X^{d,0}_{n-1,m,K,\delta}(\rdown{s-}_K)}\nonumber\\
&\qquad\qquad\qquad
-
\E\! \left[\mu_d (x, X^d_{K,\delta}(\rdown{s-}_K) )\right]\!|_{x=X^d_{K,\delta}(\rdown{s-}_K)}\Bigr\rVert_{L^2(\P;\R^d)}^2ds
\biggr)^\frac{1}{2}\nonumber\\
&\leq\kappa \left(\int_{0}^{t}\left\lVert X^{d,0}_{n-1,m,K,\delta}(\rdown{s-}_K)-X_{K,\delta}^d(\rdown{s-}_K)\right\rVert_{L^2(\P;\R^d)}^2ds\right)^\frac{1}{2}
.\label{k05}
\end{align}
Next, 
It\^o's isometry and Lemma~\ref{a05f} imply that there exists $\kappa\in(0,\infty)$ such that in both cases for all
$d,n,m,K\in \N$, $\delta\in \Delta$, $t\in [0,T]$ we have that
\begin{align}
&
\biggl\lVert
\int_{0}^{t}\Bigl(
\E\! \left[\sigma_d (x, X^{d,0}_{n-1,m,K,\delta}(\rdown{s-}_K) )\right]\!|_{x=X^{d,0}_{n-1,m,K,\delta}(\rdown{s-}_K)}\nonumber\\
&\qquad\qquad
-
\E\! \left[\sigma_d (x, X^d_{K,\delta}(\rdown{s-}_K) )\right]\!|_{x=X^d_{K,\delta}(\rdown{s-}_K)}\Bigr)dW^{d,0}(s)\biggr\rVert_{L^2(\P;\R^d)}\nonumber\\
&=
\biggl(
\int_{0}^{t}\Bigl\lVert
\E\! \left[\sigma_d (x, X^{d,0}_{n-1,m,K,\delta}(\rdown{s-}_K) )\right]\!|_{x=X^{d,0}_{n-1,m,K,\delta}(\rdown{s-}_K)}\nonumber\\
&\qquad\qquad\qquad
-
\E\! \left[\sigma_d (x, X^d_{K,\delta}(\rdown{s-}_K) )\right]\!|_{x=X^d_{K,\delta}(\rdown{s-}_K)}\Bigr\rVert_{L^2(\P;\R^{d\times d})}^2
ds
\biggr)^\frac{1}{2}\nonumber\\
&\leq\kappa \left(\int_{0}^{t}\left\lVert X^{d,0}_{n-1,m,K,\delta}(\rdown{s-}_K)-X_{K,\delta}^d(\rdown{s-}_K)\right\rVert_{L^2(\P;\R^d)}^2ds\right)^\frac{1}{2}
\label{k07}
\end{align}
and
\begin{align}
&
\biggl\lVert
\int_{0}^{t}\int_{A^d_\delta}\Bigl(
\E\! \left[\eta_d (x, X^{d,0}_{n-1,m,K,\delta}(\rdown{s-}_K),z )\right]\!|_{x=X^{d,0}_{n-1,m,K,\delta}(\rdown{s-}_K)}\nonumber\\
&\qquad\qquad
-
\E\! \left[\eta_d (x, X^d_{K,\delta}(\rdown{s-}_K),z )\right]\!|_{x=X^d_{K,\delta}(\rdown{s-}_K)}\Bigr)\tilde{N}^{d,0}(dz,ds)\biggr\rVert_{L^2(\P;\R^d)}\nonumber\\
&=
\biggl(
\int_{0}^{t}\int_{A^d_\delta}\Bigl\lVert
\E\! \left[\eta_d (x, X^{d,0}_{n-1,m,K,\delta}(\rdown{s-}_K),z )\right]\!|_{x=X^{d,0}_{n-1,m,K,\delta}(\rdown{s-}_K)}\nonumber\\
&\qquad\qquad\qquad
-
\E\! \left[\eta_d (x, X^d_{K,\delta}(\rdown{s-}_K),z )\right]\!|_{x=X^d_{K,\delta}(\rdown{s-}_K)}\Bigr\rVert_{L^2(\P;\R^d)}^2\nu^d(dz)\,
ds
\biggr)^\frac{1}{2}\nonumber\\
&\leq\kappa \left(\int_{0}^{t}\left\lVert X^{d,0}_{n-1,m,K,\delta}(\rdown{s-}_K)-X_{K,\delta}^d(\rdown{s-}_K)\right\rVert_{L^2(\P;\R^d)}^2ds\right)^\frac{1}{2}
.\label{k08}
\end{align}
Lemmas~\ref{k02} and \ref{a14}, the triangle inequality, and \eqref{k05}--\eqref{k08} complete the proof of 
Lemma~\ref{s11}.
\end{proof}

Next, we combine the statistical error and the bias to have the MLP errors.
\begin{corollary}[MLP error]\label{k90}
For every  $d,K\in \N$, $\delta\in [0,1]$ let $X^d_{K,\delta}$ be the process in Lemma~\ref{a14}. Then  there exists $\kappa\in (0,\infty)$ such that in both cases for all 
$d,n,m,K\in \N$, $\delta\in \Delta$, $t\in [0,T]$ we have that
$
\left\lVert 
X^{d,0}_{n,m,K,\delta}(t)-X_{K,\delta}^{d}(t)
\right\rVert_{L^2(\P;\R^d)}\leq \kappa^n d^\kappa e^\frac{m}{2}m^{-\frac{n}{2}}.
$
\end{corollary}
\begin{proof}[Proof of Corollary~\ref{k90}]
Combining Corollary~\ref{k09}, Lemma~\ref{s11}, and the triangle inequality implies that 
there exists $\kappa\in (0,\infty)$ such that in both cases for all 
$d,n,m,K\in \N$, $\delta\in \Delta$, $t\in [0,T]$ we have that
\begin{align}
&
\left\lVert 
X^{d,0}_{n,m,K,\delta}(t)-
X_{K,\delta}^d(t)
\right\rVert_{L^2(\P;\R^d)}\nonumber\\
&\leq 
\left\lVert 
X^{d,0}_{n,m,K,\delta}(t)-
\E_{\mathcal{G}^{d}_{n-1,m,K,\delta} }\!\left[X^{d,0}_{n,m,K,\delta}(t)\right]
\right\rVert_{L^2(\P;\R^d)}
+
\left\lVert 
\E_{\mathcal{G}^{d}_{n-1,m,K,\delta} }\!\left[X^{d,0}_{n,m,K,\delta}(t)\right]
-
X_{K,\delta}^d(t)
\right\rVert_{L^2(\P;\R^d)}
\nonumber\\
&\leq \frac{\kappa d^\kappa }{\sqrt{m^n}}
+\sum_{\ell=0}^{n-1}\frac{\kappa}{\sqrt{m^{n-\ell-1}}}
\left(\int_{0}^{t}\left\lVert X_{\ell,m,K,\delta}^{d,0} (\rdown{s-}_K ) 
-X_{K,\delta}^d(\rdown{s-}_K )
\right\rVert_{L^2(\P;\R^d)}^2 ds\right)^\frac{1}{2}.
\label{k17}\end{align}
Hence,  for all 
$d,n,m,K\in \N$, $\delta\in \Delta$, $t\in [0,T]$ we have that
\begin{align}
&\sup_{u\in [0,t]}
\left\lVert 
X^{d,0}_{n,m,K,\delta}(u)-
X_{K,\delta}^d(u)
\right\rVert_{L^2(\P;\R^d)}\nonumber\\
&\leq \frac{\kappa d^\kappa }{\sqrt{m^n}}
+\sum_{\ell=0}^{n-1}\frac{\kappa}{\sqrt{m^{n-\ell-1}}}
\left(\int_{0}^{t}\sup_{u\in[0,s]}\left\lVert X_{\ell,m,K,\delta}^{d,0} (u) 
-X_{K,\delta}^d(u)
\right\rVert_{L^2(\P;\R^d)}^2ds \right)^\frac{1}{2}.
\end{align}
This, \cite[Lemma~3.11]{HJKN2020}, and \eqref{s05}
 show for all $n,m,d,K\in \N$,
$t\in [0,T]$, $\delta\in \Delta$ that
\begin{align}
&
\sup_{u\in[0,t]}
\left\lVert 
X^{d,0}_{n,m,K,\delta}(u)-X_{K,\delta}^{d}(u)
\right\rVert_{L^2(\P;\R^d)} \nonumber \\
&
\leq 
\left(\kappa d^\kappa + \kappa T^{1/2} \sup_{t\in [0,T]} \left\lVert X^{d}_{K,\delta}(t)\right\rVert_{L^2(\P;\R^d)}  \right)
e^{\frac{m}{2}}m^{-\frac{n}{2}}
\left(1+\kappa T^{1/2}\right)^{n-1}.
\end{align}
This and Lemma~\ref{b01} complete the proof of Corollary~\ref{k90}.
\end{proof}

%\subsection{The truncation error}We will establish the truncation error of the approximations.

\subsection{Completion of the proofs of Theorems~\ref{p01} and \ref{p02}}

\begin{proof}[Proof of Theorem~\ref{p01}]
Combining Corollary~\ref{k90} (with $\delta\gets 0$), Lemma~\ref{f07}, and the triangle inequality show that
there exists $\kappa\in (0,\infty)$ such that  for all 
$d, n,m,K\in \N$,  $t\in [0,T]$ we have that
 
\begin{align}
\left\lVert 
X^{d,0}_{n,m,K,0}(t)-X^{d}_\infty(t)
\right\rVert_{L^2(\P;\R^d)}&\leq
\left\lVert 
X^{d,0}_{n,m,K,0}(t)-X^{d}_{K,0}(t)
\right\rVert_{L^2(\P;\R^d)}+
\left\lVert 
X^{d}_{K,0}(t)-
X^{d}_{\infty}(t)
\right\rVert_{L^2(\P;\R^d)}\nonumber\\
&\leq 
\kappa^n d^\kappa e^\frac{m}{2}m^{-\frac{n}{2}}+\frac{\kappa d^\kappa}{\sqrt{K}}.\label{k19}
\end{align}
This
completes the proof of Theorem~\ref{p01}.
\end{proof}

\begin{proof}[Proof of Theorem~\ref{p02}]
For every $d\in \N$, 
$G\in C(\R^d,\R^d)$,
 $K\in \N$, $\theta\in \Theta$, $k\in [1,K]\cap\Z$ we have a.s.\ that
\begin{align}
\int_{\frac{(k-1)T}{K}}^\frac{kT}{K}\int_{\R^d}G(z)N^{d,\theta}(dz,ds)= \sum_{j=1}^{P_k}G(\rvV_j)
\end{align}
where $P_k$ is Poisson distributed with intensity
 $\nu^d(\R^d)T/K$
and $\rvV_j$, $j\in \N$, are independent random variables with distribution $\nu^d(\cdot)/\nu^d(\R^d)$ and independent of $P_k$. 
This means each part $N^{d,\theta}(dz,(\frac{(k-1)T}{K},\frac{kT}{K} ])$ of the random measure can be represented by saving $P_k+1$ random variables. This number can be unbounded. However, its expectation equals  $\nu^d(\R^d)T/K+1$.
The cost $\mathsf{C}_{n,m,K}^d$ contains the following costs:
For each $\ell \in [0,n-1]\cap\Z$, $k\in [1,m^{n-\ell}]\cap\Z$:
\begin{itemize}
\item We need to prepare two processes
$X^{d,\theta}_{\ell,m,K,0}$ and $X^{d,\theta}_{\ell-1,m,K,0}$, which brings the cost 
$\mathsf{C}^d_{\ell,m,K}+\1_\N(\ell)\mathsf{C}^d_{\ell-1,m,K}$.
\item We need to prepare two processes
$X^{d,(\theta,n,k,\ell)}_{\ell,m,K,0}$ and $X^{d,(\theta,n,k,\ell)}_{\ell-1,m,K,0}$, 
which come from a new source of Brownian and Poisson randomness,
which brings the cost 
$\mathsf{C}^d_{\ell,m,K}+\1_\N(\ell)\mathsf{C}^d_{\ell-1,m,K}+ K\mathsf{cost}_{\mathsf{rv,d}}+
K(\nu^d(\R^d)T/K+1)$. Here,
$K$ is the length of the $d$-dimensional Brownian motion and
$K\mathsf{cost}_{\mathsf{rv},d}$ is the cost to generate it. 
Furthermore, each random measure $N^\theta$ is discretized into $K$ parts and each part is represented by $\nu^d(\R^d)T/K+1$ random variables (in expectation) as we have seen at the beginning of this proof.

\item We need to prepare a scalar random  variable, $\mathfrak{u}^{(\theta,n,k,\ell)}$, which brings the cost $\mathsf{cost}_{\mathsf{rv},1} $.
\item We need to evaluate $\mu_d$ at two points which brings the cost $2\mathsf{cost}_{\mu_d}$.
\item We need to evaluate $\sigma_d$ at 
each point of the two processes
$((X_{\ell,m,K,0}^{d,\theta} (\frac{kT}{K} ),X_{\ell,m,K,0}^{d,(\theta,n,k,\ell)} (\frac{kT}{K} )))_{k\in [1,K]}$ and 
$((X_{\ell-1,m,K,0}^{d,\theta} (\frac{kT}{K} ),X_{\ell-1,m,K,0}^{d,(\theta,n,k,\ell)} (\frac{kT}{K} )))_{k\in [1,K]}$ which brings the cost $2K\mathsf{cost}_{\sigma_d}$.
\item We need to evaluate $\eta_d$ at $2K\nu^d(\R^d)T/K=2\nu^d(\R^d)T$ points (in expectation) that brings the cost $2\nu^d(\R^d)T\mathsf{cost}_{\eta_d}$.
\item We need to evaluate $\eta_d$ at $2\lceil \nu^d(\R^d)\rceil K$ points which brings the cost $2\lceil \nu^d(\R^d)\rceil K\mathsf{cost}_{\eta_d} $.
\end{itemize}
This, \eqref{s12}, \eqref{s13}, and the fact that
$\forall\,d,m,K\in \N\colon\mathsf{C}^d_{0,m,K}=0$
 show that there exists $\kappa\in (0,\infty)$ such that for all $d,n,m,K\in \N$ we have that
\begin{align}
\mathsf{C}^d_{n,m,K}&\leq \sum_{\ell=0}^{n-1}
m^{n-\ell} \left(2\mathsf{C}^d_{\ell,m,K}+2\1_\N(\ell)\mathsf{C}^d_{\ell-1,m,K}+\kappa d^\kappa K\right)\nonumber\\
&\leq \kappa d^\kappa Kn m^n+\sum_{\ell=1}^{n-1}
m^{n-\ell} \left(2\mathsf{C}^d_{\ell,m,K}+2\mathsf{C}^d_{\ell-1,m,K}\right).
\end{align}
Hence, there exists $\kappa\in (0,\infty)$ such that for all $d,n,m,K\in \N$ we have that
\begin{align}
m^{-n}\mathsf{C}^d_{n,m,K}\leq \kappa d^\kappa Kn+\sum_{\ell=1}^{n-1}2m^{-\ell}\mathsf{C}^d_{\ell,m,K}+ 
2m^{-(\ell-1)}\mathsf{C}^d_{\ell-1,m,K}.
\end{align}
This and 
\cite[Corollary~2.3]{HKN2022} (applied with
$(a_n)_{n\in \N_0}\gets (m^{-n}\mathsf{C}^d_{n,m,K})_{n\in \N_0}$,
$\kappa\gets2$, $\lambda\gets 2$,
$c_1\gets 0$, $c_2\gets \kappa d^\kappa K$, $c_3\gets 0$, $c_4\gets0$,
$\beta\gets \frac{1+2+\sqrt{(1+2)^2+4\cdot 2}}{2}$ in the notation of 
\cite[Corollary~2.3]{HKN2022}) show 
that there exists $\kappa\in (0,\infty)$ such that for all $d,n,m,K\in \N$ we have that
$
m^{-n}\mathsf{C}^d_{n,m,K}\leq 
\frac{3}{2}\kappa d^\kappa K\frac{4^n-1}{4-1}
$. Hence, there exists $\kappa\in (0,\infty)$ such that for all $d,n,m,K\in \N$ we have that
$\mathsf{C}^d_{n,m,K}\leq \kappa d^\kappa K4^n m^n$.
This implies that there exists $\kappa\in (0,\infty)$ such that for all $d,n,m,K\in \N$ we have that
\begin{align}
\mathsf{C}^d_{n,n,n^n}\leq \kappa d^\kappa4^n n^{2n}.\label{k10b}
\end{align}
Next,
Theorem~\ref{p01} demonstrates that there exists $\kappa\in (0,\infty)$ such that for all $d,n\in \N$, $t\in[0,T]$ we have that
\begin{align}
\left\lVert 
X^{d,0}_{n,n,n^n,0}(t)-X^{d}_\infty(t)
\right\rVert_{L^2(\P;\R^d)}\leq \kappa^n d^\kappa e^\frac{n}{2}n^{-\frac{n}{2}}+\frac{\kappa d^\kappa}{n^{n/2}}.
\end{align}
Choosing a large $\kappa$ then implies that there exists $\kappa\in (1,\infty)$ such that for all $d,n\in \N$, $t\in[0,T]$ we have that
\begin{align}\label{k11}
\left\lVert 
X^{d,0}_{n,n,n^n,0}(t)-X^{d}_\infty(t)
\right\rVert_{L^2(\P;\R^d)}\leq\kappa^n d^\kappa n^{-n/2}
\end{align}
From now on we fix $\kappa\in (1,\infty)$ such that both \eqref{k10b} and \eqref{k11} hold. 
Furthermore, for every $\gamma\in (0,1)$ let 
\begin{align}\label{k12}
\mathfrak{C}_\gamma=\sup_{n\in \N}\left(\kappa^{6n}4^{4n}n^2 n^{-{\gamma n/2}}\right)
\end{align}
Since $(n^{\gamma n/2})_{n\in \N} $ increases faster than any exponential function, one sees that $\mathfrak{C}_\gamma<\infty$ for all $\gamma\in (0,1)$.
Next, \eqref{k10b} implies for all $d,n\in \N$ that
\begin{align}
\mathsf{C}^d_{(n+1),(n+1),(n+1)^{n+1}}\leq \kappa  d^\kappa 4^{n+1} (n+1)^{2(n+1)}\leq \kappa  d^\kappa4^{n+1} (2n)^{2(n+1)}\leq \kappa  d^\kappa 4^{4n} n^2 n^{2n}.
\end{align}
This, \eqref{k11}, and \eqref{k12} show for all $d,n\in \N$ that
\begin{align}
&
\mathsf{C}^d_{(n+1),(n+1),(n+1)^{n+1}}\left(
\sup_{t\in [0,T]}
\left\lVert 
X^{d,0}_{n,n,n^n,0}(t)-X^{d}_\infty(t)
\right\rVert_{L^2(\P;\R^d)}\right)^{4+\gamma}\nonumber\\
&\leq 
 \kappa    d^\kappa 4^{4n} n^2 n^{2n} (\kappa^n d^\kappa n^{-n/2})^{4+\gamma}\leq \kappa^{6n}d^{6\kappa}4^{4n}n^2 n^{-{\gamma n/2}}\leq  d^{6\kappa} \mathfrak{C}_\gamma.\label{k13}
\end{align}
Now, for every $d\in \N$, $\epsilon\in (0,1)$ let 
\begin{align}
\mathsf{n}_{d,\epsilon} =\inf\!\left(\left \{
n\in \N\colon \sup_{k\in [n,\infty)\cap\Z}
\sup _{t\in [0,T]}
\left\lVert
X^{d,0}_{k,k,k^k,0}(t)-X^{d}_\infty(t)
\right\rVert_{L^2(\P;\R^d)}<\epsilon
\right\}\cup\{\infty\}\right).\label{c16}
\end{align}
Then \eqref{k10b}, \eqref{k13}, and \eqref{k12} show for all $d\in \N$, $\epsilon\in (0,1)$ that
\begin{align}
&
\mathsf{C}_{\mathsf{n}_{d,\epsilon},\mathsf{n}_{d,\epsilon}, (\mathsf{n}_{d,\epsilon})^{\mathsf{n}_{d,\epsilon}}}^d\epsilon^{4+\gamma}\nonumber\\
&\leq\1_{\{1\}}(\mathsf{n}_{d,\epsilon})\mathsf{C}^d_{1,1,1}
+\1_{[2,\infty)} (\mathsf{n}_{d,\epsilon})
\mathsf{C}_{\mathsf{n}_{d,\epsilon},\mathsf{n}_{d,\epsilon}, (\mathsf{n}_{d,\epsilon})^{\mathsf{n}_{d,\epsilon}}}^d
\left(
\sup_{t\in [0,T]}\left\lVert 
X^{d,0}_{n,n,n^n,0}(t)-X_\infty^{d,0}(t)
\right\rVert_{L^2(\P;\R^d)}\right)^{4+\gamma}|_{n=\mathsf{n}_{d,\epsilon}-1}\nonumber\\
&\leq\1_{\{1\}}(\mathsf{n}_{d,\epsilon}) 4\kappa d^\kappa+
\1_{[2,\infty)} (\mathsf{n}_{d,\epsilon})
d^{6\kappa}\mathfrak{C}_\gamma\nonumber\\
&\leq d^{6\kappa}\mathfrak{C}_\gamma.\label{c18}
\end{align}
This completes the proof of Theorem~\ref{p02}.
\end{proof}

\subsection{Completion of the proofs of Theorems~\ref{p01b} and \ref{p02b}}
\begin{proof}[Proof of Theorem~\ref{p01b}]Combining Corollary~\ref{k90}, Lemmas~\ref{f07} and~\ref{b04}, and the triangle inequality 
shows that  there exists $\kappa\in (0,\infty)$ such that  for all 
$d\in \N$, $\delta\in (0,1]$, $n,m,K\in \N$,  $t\in [0,T]$ we have that
\begin{align}
&\left\lVert 
X^{d,0}_{n,m,K,\delta}(t)-X^{d}_\infty(t)
\right\rVert_{L^2(\P;\R^d)}\nonumber\\
&\leq 
\left\lVert 
X^{d,0}_{n,m,K,\delta}(t)-
X^{d}_{K,\delta}(t)
\right\rVert_{L^2(\P;\R^d)}
+
\left\lVert 
X^{d}_{K,\delta}(t)-
X^{d}_{K,0}(t)
\right\rVert_{L^2(\P;\R^d)}
+
\left\lVert 
X^{d}_{K,0}(t)-X^d_\infty(t)
\right\rVert_{L^2(\P;\R^d)}\nonumber\\
&\leq \kappa^n d^\kappa
 e^\frac{m}{2}m^{-\frac{n}{2}}+\delta^\frac{q}{2}\kappa d^\kappa+\frac{\kappa d^\kappa}{\sqrt{K}}.\label{k18}
\end{align}
This 
completes the proof of Theorem~\ref{p01b}.
\end{proof}

\begin{proof}[Proof of Theorem~\ref{p02b}]First, \eqref{p03}, and \eqref{s13b} show for all $d\in \N$, $\delta\in (0,1]$ that
\begin{align}
\nu^d(A^d_\delta)=
\nu^d(
\{z\in \R^d\colon \lVert z\rVert\geq \delta\})\leq \frac{1}{\delta^2}\int_{\lVert z\rVert\geq \delta}(1\wedge\lVert z\rVert^2)\,\nu^d(dz)
\leq \frac{cd^c}{\delta^2}.\label{p04}
\end{align}
Next, for every $d,K\in \N$, 
$G\in C(\R^d,\R^d)$,
 $\theta\in \Theta$, $k\in [1,K]\cap\Z$ we have a.s.\ that
\begin{align}
\int_{\frac{(k-1)T}{K}}^\frac{kT}{K}\int_{A^d_\delta}G(z)N^{d,\theta}(dz,ds)= \sum_{j=1}^{P_k}G(\rvV_j)
\end{align}
where $P_k$ is Poisson distributed with intensity $\nu^d(A^d_\delta)T/K$
and $\rvV_j$, $j\in \N$, are independent random variables with distribution $\frac{\nu^d((\cdot)\cap A^d_\delta)}{\nu^d(A^d_\delta)}$ and independent of $P_k$. 
This means each part $N^{d,\theta}(dz,(\frac{(k-1)T}{K},\frac{kT}{K} ])$ of the random measure can be represented by saving $P_k+1$ random variables. This number can be unbounded. However, its expectation equals to $\nu^d(A^d_\delta)T/K+1$.
The cost $\mathsf{C}_{n,m,K,\delta}^d$ contains the following costs:
For each $\ell \in [0,n-1]\cap\Z$, $k\in [1,m^{n-\ell}]\cap\Z$:
\begin{itemize}
\item We need to prepare two processes
$X^{d,\theta}_{\ell,m,K,\delta}$ and $X^{d,\theta}_{\ell-1,m,K,\delta}$, which brings the cost 
$\mathsf{C}^d_{\ell,m,K,\delta}+\1_\N(\ell)\mathsf{C}^d_{\ell-1,m,K,\delta}$.
\item We need to prepare two processes
$X^{d,(\theta,n,k,\ell)}_{\ell,m,K,\delta}$ and $X^{d,(\theta,n,k,\ell)}_{\ell-1,m,K,\delta}$, 
which come from a new source of Brownian and Poisson randomness,
which brings the cost 
$\mathsf{C}^d_{\ell,m,K,\delta}+\1_\N(\ell)\mathsf{C}^d_{\ell-1,m,K,\delta}+ K\mathsf{cost}_{\mathsf{rv,d}}+
K(\nu^d(A^d_\delta)T/K+1)$. Here,
$K$ is the length of the $d$-dimensional Brownian motion and
$K\mathsf{cost}_{\mathsf{rv},d}$ is the cost to generate it. 
Furthermore, each random measure $N^\theta$ is discretized into $K$ parts and each part is represented by $\nu^d(A^d_\delta)T/K+1$ random variables (in expectation) as we have seen at the beginning of this proof.

\item We need to prepare a scalar random  variable, $\mathfrak{u}^{(\theta,n,k,\ell)}$, which brings the cost $\mathsf{cost}_{\mathsf{rv},1} $
\item We need to evaluate $\mu_d$ at two points which brings the cost $2\mathsf{cost}_{\mu_d}$
\item We need to evaluate $\sigma_d$ at 
each point of the two processes
$((X_{\ell,m,K,\delta}^{d,\theta} (\frac{kT}{K} ),X_{\ell,m,K,\delta}^{d,(\theta,n,k,\ell)} (\frac{kT}{K} )))_{k\in [1,K]}$ and 
$((X_{\ell-1,m,K,\delta}^{d,\theta} (\frac{kT}{K} ),X_{\ell-1,m,K,\delta}^{d,(\theta,n,k,\ell)} (\frac{kT}{K} )))_{k\in [1,K]}$ which brings the cost $2K\mathsf{cost}_{\sigma_d}$
\item We need to evaluate $\eta_d$ at $2K\nu^d(A^d_\delta)T/K=2\nu^d(A^d_\delta)T$ points (in expectation) that brings the cost $2\nu^d(A^d_\delta)T\mathsf{cost}_{\eta_d}$.
\item We need to evaluate $\eta_d$ at $2\lceil \nu^d(A^d_\delta)\rceil K$ points which brings the cost $2\lceil \nu^d(A^d_\delta)\rceil K\mathsf{cost}_{\eta_d} $.
\end{itemize}
This, \eqref{s12b}, \eqref{p04}, and the fact that
$\forall\,d,m,K\in \N\colon\mathsf{C}^d_{-1,m,K}=\mathsf{C}^d_{0,m,K}=0$
 show that there exists $\kappa\in (0,\infty)$ such that for all $d,
n,m,K\in \N$, $\delta\in (0,1]$ we have that
\begin{align}
\mathsf{C}^d_{n,m,K,\delta}&\leq \sum_{\ell=0}^{n-1}
m^{n-\ell} \left(2\mathsf{C}^d_{\ell,m,K,\delta}+2\1_\N(\ell)\mathsf{C}^d_{\ell-1,m,K,\delta}+\frac{\kappa d^\kappa K}{\delta^2}\right)\nonumber\\
&\leq \frac{\kappa d^\kappa K}{\delta^2}n m^n+\sum_{\ell=1}^{n-1}
m^{n-\ell} \left(2\mathsf{C}^d_{\ell,m,K,\delta}+2\mathsf{C}^d_{\ell-1,m,K,\delta}\right).
\end{align}
Hence, there exists $\kappa\in (0,\infty)$ such that for all $d,n,m,K\in \N$, $\delta\in (0,1]$ we have that
\begin{align}
m^{-n}\mathsf{C}^d_{n,m,K,\delta}\leq \frac{\kappa d^\kappa K}{\delta^2}n+\sum_{\ell=1}^{n-1}2m^{-\ell}\mathsf{C}^d_{\ell,m,K,\delta}+ 
2m^{-(\ell-1)}\mathsf{C}^d_{\ell-1,m,K,\delta}.
\end{align}
This and 
\cite[Corollary~2.3]{HKN2022}(applied with
$(a_n)_{n\in \N_0}\gets (m^{-n}\mathsf{C}^d_{n,m,K,\delta})_{n\in \N_0}$,
$\kappa\gets2$, $\lambda\gets 2$,
$c_1\gets 0$, $c_2\gets \frac{\kappa d^\kappa K}{\delta^2}$, $c_3\gets 0$, $c_4\gets0$,
$\beta\gets \frac{1+2+\sqrt{(1+2)^2+4\cdot 2}}{2}$ in the notation of 
\cite[Corollary~2.3]{HKN2022}) show 
that there exists $\kappa\in (0,\infty)$ such that for all $d, n,m,K\in \N$, $\delta\in (0,1]$ we have that
$
m^{-n}\mathsf{C}^d_{n,m,K,\delta}\leq 
\frac{3}{2}\frac{\kappa d^\kappa K}{\delta^2}\frac{4^n-1}{4-1}
$. Hence, there exists $\kappa\in (0,\infty)$ such that for all $d,n,m,K\in \N$, $\delta\in (0,\infty)$ we have that
$\mathsf{C}^d_{n,m,K,\delta}\leq \kappa \frac{d^\kappa K}{\delta^2}4^n m^n$.
This implies that there exists $\kappa\in (0,\infty)$ such that for all $d,n\in \N$ we have that
\begin{align}
\mathsf{C}^d_{n,n,n^n,n^{-n/q}}\leq \frac{\kappa d^\kappa}{n^{-2n/q}}4^n n^{2n}=\kappa d^\kappa 4^n n^{2n(1+\frac{1}{q})}.\label{k10}
\end{align}
Next,
Theorem~\ref{p01b} demonstrates that there exists $\kappa\in (0,\infty)$ such that for all $d, n\in \N$, $t\in[0,T]$, $\delta\in (0,1]$ we have that
\begin{align}
\left\lVert 
X^{d,0}_{n,n,n^n,\delta}(t)-X^{d}_\infty(t)
\right\rVert_{L^2(\P;\R^d)}\leq \kappa^n d^\kappa e^\frac{n}{2}n^{-\frac{n}{2}}+\delta^\frac{q}{2}\kappa d^\kappa+\frac{\kappa d^\kappa}{\sqrt{n^n}}.
\end{align}
Choosing a large $\kappa$ then implies that there exists $\kappa\in (0,\infty)$ such that for all $d,n\in \N$, $t\in[0,T]$ we have that
\begin{align}\label{k11b}
\left\lVert 
X^{d,0}_{n,n,n^n, n^{-n/q}}(t)-X^{d}_\infty(t)
\right\rVert_{L^2(\P;\R^d)}\leq\kappa^n d^\kappa n^{-n/2}.
\end{align}
From now on we fix $\kappa\in (1,\infty)$ such that both \eqref{k10} and \eqref{k11b} hold. 
Furthermore, for every $\gamma\in (0,1)$ let 
\begin{align}\label{k12b}
\mathfrak{C}_\gamma=\sup_{n\in \N}\left( \kappa^{6n(1+\frac{1}{q})}32^{4n(1+\frac{1}{q})} n^{2(1+\frac{1}{q})} n^{-{\gamma n/2}}\right).
\end{align}
Since $(n^{\gamma n/2})_{n\in \N} $ increases faster than any exponential function, we see  that $\mathfrak{C}_\gamma<\infty$ for all $\gamma\in (0,1)$.
Next, \eqref{k10} implies for all $d,n\in \N$ that
\begin{align}
\mathsf{C}^d_{(n+1),(n+1),(n+1)^{n+1}, (n+1)^{-\frac{n+1}{q}}}&\leq \kappa  d^\kappa 4^{n+1} (n+1)^{2(n+1)(1+\frac{1}{q})}\nonumber\\
&\leq \kappa  d^\kappa 16^n (2n)^{2(n+1)(1+\frac{1}{q})}\nonumber\\
&\leq \kappa  d^\kappa 32^{4n(1+\frac{1}{q})} n^{2(1+\frac{1}{q})} n^{2n(1+\frac{1}{q})}.
\end{align}
This, \eqref{k11b}, and \eqref{k12b} show for all $d,n\in \N$ that
\begin{align}
&
\mathsf{C}^d_{(n+1),(n+1),(n+1)^{n+1},(n+1)^{-\frac{n+1}{q}}}\left(
\sup_{t\in [0,T]}
\left\lVert 
X^{d,0}_{n,n,n^n, n^{-n/q}}(t)-X^{d}_\infty(t)
\right\rVert_{L^2(\P;\R^d)}\right)^{4(1+\frac{1}{q})+\gamma}\nonumber\\
&\leq 
 \kappa  d^\kappa 32^{4n(1+\frac{1}{q})} n^{2(1+\frac{1}{q})} n^{2n(1+\frac{1}{q})} (\kappa^n d^\kappa n^{-n/2})^{4(1+\frac{1}{q})+\gamma}\nonumber\\
&\leq \kappa^{6n(1+\frac{1}{q})}d^{6\kappa(1+\frac{1}{q})}32^{4n(1+\frac{1}{q})} n^{2(1+\frac{1}{q})} n^{-{\gamma n/2}}\nonumber\\
&
\leq  d^{6\kappa(1+\frac{1}{q})} \mathfrak{C}_\gamma.\label{k13b}
\end{align}
Now, for every $d\in \N$, $\epsilon\in (0,1)$ let 
\begin{align}
\mathsf{n}_{d,\epsilon} =\inf\!\left(\left \{
n\in \N\colon \sup_{k\in [n,\infty)\cap\Z}
\sup _{t\in [0,T]}
\left\lVert
X^{d,0}_{k,k,k^k, k^{-k/q}}(t)-X^{d}_\infty(t)
\right\rVert_{L^2(\P;\R^d)}<\epsilon
\right\}\cup\{\infty\}\right).\label{c16b}
\end{align} 
By \eqref{k11b} we have that $\mathsf{n}_{d,\epsilon}<\infty$.
Then \eqref{k10}, \eqref{k13b}, and \eqref{k12b} show for all $d\in \N$, $\epsilon\in (0,1)$ that
\begin{align}
&
\mathsf{C}_{\mathsf{n}_{d,\epsilon},\mathsf{n}_{d,\epsilon}, (\mathsf{n}_{d,\epsilon})^{\mathsf{n}_{d,\epsilon}}, (\mathsf{n}_{d,\epsilon})^{- \mathsf{n}_{d,\epsilon}/q }}^d
\epsilon^{4(1+\frac{1}{q})+\gamma}\nonumber\\
&\leq\1_{\{1\}}(\mathsf{n}_{d,\epsilon})\mathsf{C}^d_{1,1,1,1}
+\1_{[2,\infty)} (\mathsf{n}_{d,\epsilon})
\mathsf{C}_{\mathsf{n}_{d,\epsilon},\mathsf{n}_{d,\epsilon}, (\mathsf{n}_{d,\epsilon})^{\mathsf{n}_{d,\epsilon}}, (\mathsf{n}_{d,\epsilon})^{- \mathsf{n}_{d,\epsilon}/q }}^d
\nonumber\\
&\quad
\left(
\sup_{t\in [0,T]}\left\lVert 
X^{d,0}_{n,n,n^n, n^{-n/q}}(t)-X_\infty^{d,0}(t)
\right\rVert_{L^2(\P;\R^d)}\right)^{4(1+\frac{1}{q})+\gamma}|_{n=\mathsf{n}_{d,\epsilon}-1}\nonumber\\
&\leq\1_{\{1\}}(\mathsf{n}_{d,\epsilon}) \kappa d^\kappa+
\1_{[2,\infty)} (\mathsf{n}_{d,\epsilon})
d^{6\kappa(1+\frac{1}{q})}\mathfrak{C}_\gamma\nonumber\\
&\leq d^{6\kappa(1+\frac{1}{q})}\mathfrak{C}_\gamma.
\end{align}
This completes the proof of Theorem~\ref{p02b}.
\end{proof}

\bibliographystyle{acm}
\bibliography{References}

\end{document}